\documentclass[preprint,12pt]{elsarticle}

\usepackage{amssymb}
\usepackage{subcaption} 
\usepackage{amsmath}
\usepackage{amsthm}
\newtheorem{remark}{Remark}[section]
\newtheorem{theorem}{Theorem}[section] % 定理
\newtheorem{lemma}[theorem]{Lemma} % 引理
\numberwithin{equation}{section}
\numberwithin{figure}{section}
\catcode`\@=11
\@addtoreset{equation}{section} % Makes \section reset 'equation' counter.
\numberwithin{equation}{section} \numberwithin{figure}{section}
\numberwithin{table}{section}

\newcommand{\bx}{{\mathbf x}}

\newcommand{\bn}{{\mathbf n}}

\newcommand{\be}{\begin{equation}}
\newcommand{\ee}{\end{equation}}

\newcommand{\bse}{\begin{subequations}}
\newcommand{\ese}{\end{subequations}}

\def\beqlb{\begin{eqnarray}}\def\eeqlb{\end{eqnarray}}
\def\beqnn{\begin{eqnarray*}}\def\eeqnn{\end{eqnarray*}}
\journal{ arXiv }%Computers $\&$  Mathematics with Applications}
\begin{document}
\begin{frontmatter}

\title{A First-Order Linear Energy Stable Scheme for the  Cahn-Hilliard Equation with  Dynamic Boundary Conditions under the Effect of Hyperbolic Relaxation }

\author[MY]{Minghui Yu}

\affiliation[MY]{organization={School of Science, Beijing University of Posts and Telecommunications},
            city={Beijing},
            postcode={100876}, 
            country={China}}
            \ead{ ymh@bupt.edu.cn}
            
\author[RC]{Rui Chen}
\affiliation[RC]{organization={School of Science, Key Laboratory of Mathematics and Information Networks (Ministry of Education), Beijing University of Posts and Telecommunications},
            city={Beijing},
            postcode={100876}, 
            country={China}}
\ead{Corresponding author: ruichen@bupt.edu.cn}  

%\include{pages/abstract}  
%% Abstract
\begin{abstract}
In this paper we focus on the Cahn-Hilliard equation with dynamic boundary conditions, by adding two hyperbolic relaxation terms to the system. We verify that the energy of the total system is decreasing with time. By adding two stabilization terms, we have constructed a first-order temporal accuracy numerical scheme, which is linear and energy stable. Then we prove that the scheme is of first-order in time by the error estimates.  At last we carry out enough numerical results to validate the the temporal convergence and the energy stability of such  scheme. Moreover, we have present the differences of the numerical results with and without the hyperbolic terms, which show that the hyperbolic terms can help the total energy decreasing slowly.
%This paper combines dynamic boundary conditions with hyperbolic relaxation terms, resulting in a dynamic boundary condition with hyperbolic relaxation. A first-order energy stable scheme is constructed for this system. Through error analysis, we verify the energy stability of the scheme, which is also first-order in time. Finally, we provide sufficient numerical results to verify the convergence and energy stability of the scheme, and demonstrate the influence of the parameters in front of the inertia term on the equation. Specifically, as the parameters increase, the energy decreases more slowly, and the convergence becomes slower.
\end{abstract}

%% Keywords
\begin{keyword}
    Hyperbolic Cahn-Hilliard equation\sep Dynamic boundary conditions\sep Error estimates\sep Linear numerical scheme\sep Energy stability.
\MSC 65N08 \sep 65N12     
\end{keyword}
\date{}
\end{frontmatter}

\section{Introduction}

The Cahn-Hilliard equation was first proposed by John W. Cahn and John E. Hilliard in 1958 to describe the phase separation phenomenon in binary mixtures (such as alloys and solutions) \cite{Cahn1958}. This equation has become a cornerstone in materials science, describing the phase separation process in binary alloys accurately, especially in the early stages of spinodal decomposition. The Cahn-Hilliard equation assumes that the material is isotropic and has been widely applied in theoretical studies of phase separation processes \cite{Cherfils2011, Novick-Cohen2008, Novick-Cohen2008b}. For example, it not only simulates spontaneous heterogenization in binary mixtures such as spinodal decomposition, but also describes mechanisms of pattern formation such as nucleation and growth and coarsening \cite{Bates1993, Abels2012, Anderson1998}.

As a representative of diffuse interface models, the Cahn-Hilliard equation avoids the explicit interface tracking issues of classical sharp-interface models by dividing the components of the mixture into thin layers, thus improving computational efficiency \cite{Cahn1961}. Moreover, this model can naturally handle complex geometries and topological changes of interfaces, significantly simplifying the computation process \cite{Du2020}. The Cahn-Hilliard equation and its variants have been widely applied in many fields, including block copolymers \cite{Kim2016}, image inpainting \cite{Bertozzi2007, Bertozzi2007b}, tumor growth models \cite{Garcke2016, Oden2010, Oden2013}, two-phase flow \cite{Gurtin1996, Hohenberg1977}, and moving contact line problems \cite{Jacqmin2000, Pruss2006}.

The Cahn-Hilliard equation is usually equipped with periodic boundary conditions or homogeneous Neumann boundary conditions. 
Then Liu et al. \cite{Liu2019} have proposed the Cahn-Hilliard type dynamic boundary condition for the Cahn-Hilliard equation. In their model, the system is energy-stable and conserves mass both in the bulk and on the boundary. Other variants of the Cahn-Hilliard equation, particularly those with dynamic boundary conditions, also exist in the literature (see references \cite{Goldstein2011,Knopf2021}). Numerous studies have investigated energy-stable numerical schemes for the Cahn-Hilliard equation under classical boundary conditions, particularly periodic and Neumann boundary conditions, such as the stabilization method \cite{HLT07}, the convex splitting approach \cite{Chen16,Shen12,Wise10}, the Lagrange multiplier approach \cite{BGG2011,GT2013,GT2014,TG2014}, the Invariant Energy Quadratization (IEQ) approach \cite{Chen17,GX19,Yang16,Yang17}, the Scaler Auxiliary Variable (SAV) approach \cite{Shen17,HSY20} and other approaches \cite{CY19,CLS19,QSZ15,ZY20}. Meanwhile, several studies have also examined energy-stable numerical schemes for the Cahn-Hilliard equation with dynamic boundary conditions (see references \cite{Liu2019,Metzger2023,Bao2021,Bao2021a,
Meng2023,Knopf2021a,Cherfils2010,Cherfils2014,Fukao2017,Israel2014}).

Considering the delay in the separation of phases,
 Galenko et al. \cite{Galenko2001,Galenko2005,Galenko2007,LZG09,Galenko2013} have introduced the hyperbolic relaxation term to the Cahn-Hilliard system.
  Compared to the original equation, the equation with the inertial term is a hyperbolic equation with relaxation characteristics, which leads to different mathematical features in numerical solutions and introduces new challenges \cite{Galenko2013}. Additionally, the introduction of the hyperbolic term provides a deeper understanding of the dynamics of phase separation, especially in describing the delay of rapid phase transitions. There are also some works on designing the energy stable schemes for the hyperbolic Cahn-Hilliard model. Yang et al. \cite{YZH2018,CMY2023} have constructed energy stable schemes for the viscous Cahn-Hilliard equation with hyperbolic relaxation by the IEQ approach. Meanwhile, they show the error analysis for the second-order semi-discrete temporal discretization schemes. Wu et al. \cite{Wu2007} have investigated the well-posedness and asymptotic behavior of solutions to the parabolic-hyperbolic phase field system with dynamic boundary conditions.

%In practical applications, to simulate the non-equilibrium decomposition in glasses caused by deep undercooling, Galenko et al. proposed the introduction of an inertial term into the Cahn-Hilliard equation. The modified equation showed high agreement with experimental results \cite{Galenko2001, Galenko2005, Galenko2007}. Compared to the original equation, the equation with the inertial term is a hyperbolic equation with relaxation characteristics, which leads to different mathematical features in numerical solutions and introduces new challenges \cite{Galenko2013}. Additionally, the introduction of the inertial term provides a deeper understanding of the dynamics of phase separation, especially in describing rapid phase transitions.

%In recent years, the study of the Cahn-Hilliard equation has extended to situations with dynamic boundary conditions, considering the interaction between fluids and solid interfaces \cite{Abels2020, Barrett2020, Wu2018}. Dynamic boundary conditions are crucial for simulating the behavior of material interfaces, especially in multi-phase flow and the separation behavior of complex boundaries. They can more realistically reproduce the movement and evolution of actual interfaces \cite{Abels2020}. These studies have not only improved the understanding of material interface phenomena but also provided new insights for solving practical problems.

Inspired by the Cahn-Hilliard model \cite{Liu2019} and  hyperbolic effects, we incorporate hyperbolic terms into both the bulk equation and the dynamic boundary condition. We find that this  hyperbolic model with the hyperbolic dynamic boundary condition simultaneously satisfies the energy dissipation law and preserves mass conservation in the bulk and on the boundary under specific conditions. Then we utilize a stabilization approach to construct a first-order temporal discretization scheme that is both linear and energy stable. For spatial approximation, we adopt a central finite difference discretization.
%This paper primarily investigates the Cahn-Hilliard equation with dynamic boundary conditions and inertial terms. By adjusting the parameters in front of the inertial terms, the paper explores the effects of these changes on the evolution of phase interfaces, phase separation velocity, stability, and boundary dynamic behavior. In particular, under dynamic boundary conditions, the inertial term not only affects the internal phase separation process but also alters the material exchange and flow characteristics at the boundaries. This study aims to reveal how these parameters interact with the physical properties of the system and provide a new perspective for understanding the behavior of the Cahn-Hilliard equation with dynamic boundary conditions and inertial terms. This will help provide a more comprehensive view for the theoretical study and application of Cahn-Hilliard systems.

The contributions of this paper are present as follows.
\begin{itemize}%[leftmargin=1mm]
\item To the best of our knowledge, it is the first time to investigate the hyperbolic Cahn-Hilliard equation with the hyperbolic dynamic boundary condition. This hyperbolic model holds the energy dissipation law. 

\item We construct a first-order linear energy stable scheme for the model by the stabilization method. Meanwhile, we give the rigorous analysis to prove the scheme is of first-order time accuracy.
    
\item We carry out enough numerical cases to illustrate the time accuracy and the energy decay in the scheme. Moreover we verify that the hyperbolic terms can delay the spinodal decomposition (or coarsening) from the numerical tests.
\end{itemize}

The remainder of this paper is organized as follows: In Section \ref{sec2}, we introduce the governing equations with hyperbolic relaxation, which is energy stable. In Section \ref{sec3}, we construct a linear, energy stable and first-order temporal accuracy  semi-discreate scheme by adding two stabilization terms. In Section \ref{sec4}, we show the  error analysis of the numerical scheme. In Section \ref{sec5} we provide enough numerical results to show the temporal accuracy and illustrate the effect of the hyperbolic relaxation terms. Finally, we present the concluding remarks in the last Section.

\section{The governing equations}\label{sec2}

In the first place, we recall that the Liu-Wu model \cite{Liu2019} in the following form:
\begin{align}
    \label{ch1}
    &\phi_t= M_1 \Delta\mu, & \text { in } \Omega \times(0, T], \\
    \label{ch2}
    &\mu=  -\Delta \phi+ f(\phi), &\text { in } \Omega \times(0, T], \\
    \label{nbc1}
    &\partial_\bn\mu=0,& \text { on } \Gamma \times(0, T],
    \\
    \label{ch3}
    &\left.\phi\right|_{\Gamma}=\psi, & \text { on } \Gamma \times(0, T], \\
    \label{ch4}
    &\psi_t= M_2 \Delta_{\Gamma}\mu_{\Gamma}, & \text { on } \Gamma \times(0, T], \\
    \label{ch5}
    &\mu_{\Gamma}=-\Delta_{\Gamma} \psi+g(\psi)+\partial_{\mathbf{n}} \phi, & \text { on } \Gamma \times(0, T],
\end{align}
where $T$ is a finite time, $\Omega\subset\mathbb{R}^d(d=2,3)$ is the bounded domain with its boundary $\Gamma=\partial\Omega$, $\bn$ denotes the unit normal vector on $\Gamma$, $\phi:=\phi(\bx,t)$ stands for the phase-field variable, $M_1$ and $M_1$ are  relaxation parameters with the positive value, $\Delta_{\Gamma}$ denotes the Laplace-Beltrami operator on $\Gamma$.
$f(\phi)=F'(\phi)$. $F(\phi)$ is the double well (Ginzburg-Landau) potential,
\begin{align}
    \label{equ:double-well potential}
    F(\phi)=\frac{1}{4\varepsilon^2}(\phi^2-1)^2,
\end{align}
where $\varepsilon$ is a positive constant that measure  the width of the interface,  $\mu$ and $\mu_{\Gamma}$ stand for the chemical potentials in the bulk and on the boundary respectively, which are obtained  from the total energy.

The total energy reads as follows, consisting of the bulk energy and the surface energy,
\begin{align}
    E^{total}(\phi,\psi) &= E^{bulk}(\phi)+E^{surf}(\psi),\\
    \label{equ:E_bulk}
    E^{bulk}\left(\phi\right)&=\int_{\Omega} F\left(\phi\right)+\frac{1}{2} \left|\nabla \phi\right|^2 d \bx,\\
    \label{equ:E_surf}
    E^{surf}\left(\psi\right)&=\int_{\Gamma}  G\left(\psi\right)+\frac{1}{2}\left|\nabla_{\Gamma} \psi\right|^2 d S,
\end{align}
where $\nabla_{\Gamma}$ is the tangential or surface gradient operator on $\Gamma$, $g(\psi)=G'(\psi)$ and $G(\psi)$ is also the nonlinear potential. Ones can choose the typical potential for moving contact line problems \cite{Ma2017,Chen2018a}, or choose the double well (Ginzburg-Landau) potential \eqref{equ:double-well potential} as surface potential.

It is easy to find that the Liu-Wu model \eqref{ch1}-\eqref{ch5} satisfies the following energy dissipation law and the mass conservation law,
\beqlb
&&\frac{d}{dt}E^{total}(\phi,\psi)=-M_1\int_{\Omega}|\nabla\mu|^2d \bx-M_2\int_{\Gamma}|\nabla_{\Gamma}\mu_{\Gamma}|^2d S,\\
&&\int_{\Omega}\phi(\bx,t)d\bx=\int_{\Omega}\phi(\bx,0)d\bx,\;\int_{\Gamma}\psi(\bx,t)d S=\int_{\Gamma}\psi(\bx,0)d S.
\eeqlb

\begin{remark}
Liu-Wu model assumes that there has no mass exchange between the bulk and the boundary. While
Goldstein et al. \cite{Goldstein2011} have proposed a Cahn-Hilliard model (called GMS model) by assuming that there has  mass exchange between the bulk and the boundary. Morover Knopf et al. \cite{Knopf2021} have proposed a new model (called KLLM model), which can be  regarded as an interpolation between the GMS model \cite{Goldstein2011} and the Liu-Wu model \cite{Liu2019}. In this model, a relaxation parameter is introduced into the boundary condition. When this parameter approaches zero, the model converges to the GMS model, whereas when it tends to infinity, it reduces to the Liu-Wu model.
\end{remark}

By adding two hyperbolic terms to the Liu-Wu model \eqref{ch1}-\eqref{ch5}, we have the following hyperbolic Cahn-Hilliard equation with the hyperbolic Cahn-Hilliard type dynamic boundary condition,

\begin{align}
&\beta_1 \phi_{tt} + \phi_t = M_1 \Delta \mu, && \text { in } \Omega \times(0, T], \label{2.9}\\
&\mu = -\Delta \phi +  f(\phi), && \text { in } \Omega \times(0, T], \label{2.10}\\
&\partial_\mathbf{n}\mu=0,&& \text { on } \Gamma \times(0, T], \label{2.14}\\
&\phi|_\Gamma = \psi, &&\text { on } \Gamma \times(0, T] ,\label{2.11}\\
&\beta_2 \psi_{tt} + \psi_t = M_2 \Delta_\Gamma \mu_\Gamma,&& \text { on } \Gamma \times(0, T],\label{2.12}\\
&\mu_\Gamma = -\Delta_\Gamma \psi + g(\psi)+ \partial_\mathbf{n} \phi,&& \text { on } \Gamma \times(0, T], \label{2.13}
\end{align}
where $\beta_1\geq0$ and $\beta_2\geq0$ are the relaxation parameters. When $\beta_1=\beta_2=0$, the system reduces to the standard Liu-Wu model \eqref{ch1}-\eqref{ch5} that conserves the mass density in the bulk and on the surface. When $\beta_1>0$ and $\beta_2>0$, the mass conservation is maintained only provided that $\int_{\Omega}\phi_t(\bx,t)d\bx=0$ and $\int_{\Gamma}\psi_t(\bx,t)dS=0$. To find this, by taking the  $L^2(\Omega)$ inner product of \eqref{2.9} with $1$ and  $L^2(\Gamma)$ inner product of \eqref{2.12} with $1$ respectively, we can derive immediately, 
\beqlb
&&\beta_1\frac{d}{dt}\int_{\Omega}\phi_t(\bx,t)d\bx+\int_{\Omega}\phi_t(\bx,t)d\bx=0,\\
&&\beta_2\frac{d}{dt}\int_{\Gamma}\psi_t(\bx,t)dS+\int_{\Gamma}\psi_t(\bx,t)dS=0.
\eeqlb
Then we deduce the solutions from the ODE systems,
\beqlb
&&\int_{\Omega}\phi_t(\bx,t)d\bx
=e^{-\frac{1}{\beta_1}t}\int_{\Omega}\phi_t(\bx,0)d\bx,\\
&&\int_{\Gamma}\psi_t(\bx,t)dS
=e^{-\frac{1}{\beta_2}t}\int_{\Gamma}\psi_t(\bx,0)dS.
\eeqlb
Thus by setting $\int_{\Omega}\phi_t(\bx,0)d\bx=0$ and $\int_{\Gamma}\psi_t(\bx,0)dS=0$, we have
\beqlb
&&\int_{\Omega}\phi_t(\bx,t)d\bx=\int_{\Omega}\phi_{tt}(\bx,t)d\bx=0,\\
&&\int_{\Gamma}\psi_t(\bx,t)dS=\int_{\Gamma}\psi_{tt}(\bx,t)dS=0.
\eeqlb

Define the inverse Laplace operator $\Delta^{-1}$ and the inverse Laplace-Beltrami operator $\Delta_{\Gamma}^{-1}$ such that $W_1=\Delta^{-1}\omega_1$ (with $\int_{\Omega}\omega_1 d\bx=0$) and $W_2=\Delta_{\Gamma}^{-1}\omega_2$ (with $\int_{\Gamma}\omega_2 dS=0$), iff
\beqlb
&&\Delta W_1=\omega_1,\;\int_{\Omega}\omega_1 d\bx=0,\;\partial_{\bn} W_1|_{\Gamma}=0,\\
&&\Delta_{\Gamma} W_2=\omega_2,\;\int_{\Gamma}\omega_2 dS=0.
\eeqlb
Next we will derive the energy dissipation law for the system \eqref{2.9}-\eqref{2.13}. 
Here and after, for any function $f,g\in L^2(\Omega)$, we use $(f,g)_{\Omega}=\int_{\Omega}fgd\bx$, $(f,g)_{\Gamma}=\int_{\Gamma}fg dS$, $||f||^2=(f,f)_{\Omega}$ and $||f||_{\Gamma}^2=(f,f)_{\Gamma}$.

\begin{theorem}
The model \eqref{2.9}-\eqref{2.13} is energy stable in the sense that 
\beqlb
\frac{d}{dt}\mathcal{E}(\phi,\psi)=-\frac{1}{M_1}||\nabla \Delta ^{-1}\phi_t||^2-\frac{1}{M_2}||\nabla_{\Gamma}\Delta_{\Gamma}^{-1}\psi_t||_{\Gamma}^2,
\eeqlb
where the energy $\mathcal{E}(\phi,\psi)=E^{total}(\phi,\psi)+\frac{\beta_1}{2M_1}||\nabla\Delta^{-1}\phi_t||^2+\frac{\beta_2}{2M_2}||\nabla_{\Gamma}\Delta_{\Gamma}^{-1}\psi_t||_{\Gamma}^2$.
\end{theorem}
\begin{proof}
     We introduce two variables $\Phi=\phi_t$ and $\Psi=\psi_t$. Since $\int_\Omega \Phi \, dx = \int_\Omega \Phi_t \, dx = 0$ and $\int_\Omega \Psi \, dx = \int_\Omega \Psi_t \, dx = 0$, applying the $\Delta^{-1}$ operator to \eqref{2.9} and $\Delta_{\Gamma}^{-1}$ operator to \eqref{2.12}, we obtain the following equations,
\begin{align}
& \beta_1 \Delta^{-1} \Phi_{t} + \Delta^{-1} \Phi = M_1 (-\Delta \phi + f(\phi) ),  \label{2.15}\\
& \beta_2 \Delta^{-1}_\Gamma \Psi_{t} + \Delta^{-1}_\Gamma \Psi = M_2 (-\Delta_\Gamma \psi + g(\psi)+ \partial_{\bn} \phi ). \label{2.16}
\end{align}

By taking the $L^2(\Omega)$ inner product of \eqref{2.15} with $\frac{1}{M_1}\Phi$ and the $L^2(\Gamma)$ inner product of \eqref{2.16} with $\frac{1}{M_2}\Psi$, we obtain:
\begin{align}
&\frac{ \beta_1}{M_1} (\Delta^{-1} \Phi_{t},\Phi)_\Omega + \frac{1}{M_1}(\Delta^{-1} \Phi,\Phi)_\Omega = -(\partial_{\bn} \phi,\phi_t)_\Gamma+\frac{d}{dt}\int_\Omega(\frac{|\nabla \phi|^2}{2} +  F(\phi) )d\bx, \label{2.17}\\
& \frac{\beta_2}{M_2} (\Delta^{-1}_\Gamma \Psi_{t},\Psi)_\Gamma +\frac{1}{M_2} (\Delta^{-1}_\Gamma \Psi,\Psi)_\Gamma =(\partial_{\bn} \phi,\phi_t)_\Gamma + \frac{d}{dt}\int_\Gamma(\frac{|\nabla_\Gamma \psi|^2}{2} + G(\psi))dS. \label{2.18}
\end{align}

We define $p=\Delta^{-1}\Phi$ and $q=\Delta_\Gamma^{-1}\Psi$. Substituting these into the inner product formula above and simplifying, we obtain,
\begin{align}
&  (\Delta^{-1} \Phi,\Phi)_\Omega= (p,\Delta p)_\Omega=-\|\nabla p\|^2, \label{2.19}\\
& (\Delta^{-1} \Phi_{t},\Phi)_\Omega=(\Phi_{t},\Delta^{-1} \Phi)_\Omega= (\Delta p_{t},p)_\Omega=-\frac{1}{2}\frac{d}{dt}\|\nabla p\|^2, \label{2.20}\\
&  (\Delta^{-1}_\Gamma \Psi,\Psi)_\Gamma= (q,\Delta_\Gamma q)_\Gamma=-\|\nabla_\Gamma q\|_\Gamma^2, \label{2.21}\\
& (\Delta^{-1}_\Gamma \Psi_{t},\Psi)_\Gamma=(\Psi_{t},\Delta^{-1}_\Gamma \Psi)_\Gamma= (\Delta_\Gamma q_{t},q)_\Gamma=-\frac{1}{2}\frac{d}{dt}\|\nabla_\Gamma q\|_\Gamma^2. \label{2.22}
\end{align}

By combining the above formulas, we obtain the following energy dissipation law,
    \begin{align}
        & \frac{d}{dt}\Big{(}\int_\Omega(\frac{|\nabla \phi|^2}{2} + F(\phi) +\frac{\beta_1}{2M_1}|\nabla p|^2)d\bx+ \int_\Gamma(\frac{|\nabla_\Gamma \psi|^2 }{2}+ G(\psi)+\frac{\beta_2}{2M_2}|\nabla_\Gamma q|^2)dS \Big{)}\notag\\
         =&-\frac{1 }{M_1}\|\nabla p\|^2-\frac{1 }{M_2}\|\nabla_\Gamma q\|_{\Gamma}^2\le 0 .
    \end{align}

\end{proof} 

\section{A first-order energy stable scheme}\label{sec3}
In this section, we directly present the numerical scheme of the equation as follows, then prove the energy stability of the numerical scheme, and conduct a simple error analysis.

Assuming that $\phi^n$ and $\phi^{n-1}$ with $n\geq 1$ are known, we update $\phi^{n+1}$ as follows,
\begin{align}
&\beta_1\frac{\Phi^{n+1}-\Phi^{n}}{\tau} + \Phi^{n+1} = M_1 \Delta \mu^{n+1},\;&\text{in } \Omega,\label{3.1} \\
&\mu^{n+1} = -\Delta \phi^{n+1} + f(\phi^n) + s_1 (\phi^{n+1}-\phi^n),\;&\text{in } \Omega,\label{3.2}
\\&\Phi^{n+1} = \frac{\phi^{n+1}-\phi^{n}}{\tau},\;&\text{in } \Omega, \label{3.6}\\
&\phi^{n+1}|_\Gamma = \psi^{n+1},\;&\text{on } \Gamma,\label{3.3}\\
&\beta_2\frac{\Psi^{n+1}-\Psi^{n}}{\tau} + \Psi^{n+1} = M_2 \Delta_\Gamma \mu_\Gamma^{n+1},\;&\text{on } \Gamma,\label{3.4}\\
&\mu_\Gamma^{n+1} =  -\Delta_\Gamma \psi^{n+1} + g(\psi^n) + \partial_{\bn} \phi^{n+1} + s_2(\psi^{n+1}-\psi^n),\;&\text{on } \Gamma, \label{3.5}\\
&\Psi^{n+1} = \frac{\psi^{n+1}-\psi^{n}}{\tau},\;&\text{on } \Gamma, \label{3.7}\\
&\partial_{\bn}\mu^{n+1}=0,\;&\text{on } \Gamma, \label{3.8} 
\end{align}
where $s_1$ and $s_2$ are two stabilizers to be determined, $N$ is the number of time steps with $1 \le n < N$, and $\tau = T / N$ is the time step size. Next we will show the energy stability of the scheme.

\begin{theorem}
If $\displaystyle s_1 \ge  \frac{1}{2}\max_{\xi\in\mathbb{R}} F''(\xi)$ and $\displaystyle s_2 \ge \frac{1}{2} \max_{\eta\in\mathbb{R}} G''(\eta)$,  the scheme  \eqref{3.1}-\eqref{3.8} is energy stable in the sense that

\begin{align}
\frac{\mathcal{E}(\phi^{n+1},\psi^{n+1})-\mathcal{E}(\phi^{n},\psi^{n})}{\tau}\le -\frac{1 }{M_1}\|\nabla p^n\|^2 -\frac{1 }{M_2}\|\nabla_\Gamma q^n\|_{\Gamma}^2,
\end{align}
where $p^n = \Delta^{-1} \Phi^n$ and $q^n = \Delta^{-1}_\Gamma \Psi^n$, and the energy

\begin{equation}
\mathcal{E}(\phi^{n},\psi^n)=\frac{||\nabla \phi^n||^2}{2} + (F(\phi^n),1)_{\Omega}+\frac{\beta_1}{2M_1}||\nabla p^n||^2+ \frac{||\nabla_\Gamma \psi^n||_{\Gamma}^2}{2} + (G(\psi^n),1)_{\Gamma}+\frac{\beta_2}{2M_2}||\nabla_\Gamma q^n||_{\Gamma}^2. 
\end{equation}

\end{theorem}

\begin{proof}
By applying the inverse Laplace operator $\Delta^{-1}$ to \eqref{3.1}, we obtain,

\begin{equation}
    \beta_1 \Delta^{-1}\frac{\Phi^{n+1}-\Phi^{n}}{\tau}+\Delta^{-1}\Phi^{n+1} = M_1 \mu^{n+1}.
    \label{3.11}
\end{equation}
By taking the $L^2(\Omega)$ inner product of \eqref{3.11} with $\frac{1}{M_1}\Phi^{n+1}$, we have

\begin{equation}
    \frac{\beta_1}{M_1}(\Delta^{-1}\frac{\Phi^{n+1}-\Phi^{n}}{\tau},\Phi^{n+1})_\Omega+\frac{1}{M_1}(\Delta^{-1}\Phi^{n+1},\Phi^{n+1})_\Omega = (\mu^{n+1},\Phi^{n+1})_\Omega.
    \label{3.12}
\end{equation}
Noticing that $p^{n+1} = \Delta^{-1}\Phi^{n+1}$,  we deduce
    \begin{align}
    &\frac{\beta_1}{M_1}(\Delta^{-1}\frac{\Phi^{n+1}-\Phi^{n}}{\tau},\Phi^{n+1})_\Omega+\frac{1}{M_1}(\Delta^{-1}\Phi^{n+1},\Phi^{n+1})_\Omega\notag \\
    =&\frac{\beta_1}{M_1\tau}(p^{n+1}-p^n,\Delta p^{n+1})_\Omega+\frac{1}{M_1}(p^{n+1},\Delta p^{n+1})_\Omega\notag\\
    =&-\frac{\beta_1}{M_1\tau}(\nabla p^{n+1}-\nabla p^n,\nabla p^{n+1})_\Omega-\frac{1}{M_1}||\nabla p^{n+1}||^2\notag\\
    =&-\frac{\beta_1}{2M_1\tau}(||\nabla p^{n+1}||^2-||\nabla p^n||^2+||\nabla p^{n+1}-\nabla p^{n}||^2)-\frac{1}{M_1}||\nabla p^{n+1}||^2,\label{3.13}
    \end{align}
    and 
    \begin{align}
    &(\mu^{n+1},\Phi^{n+1})_\Omega=(\mu^{n+1},\frac{\phi^{n+1}-\phi^{n}}{\tau})_\Omega\notag\\
    =&(\frac{\phi^{n+1}-\phi^{n}}{\tau},-\Delta \phi^{n+1} + f(\phi^n) + s_1 (\phi^{n+1}-\phi^n))_\Omega\notag\\
    =&(\frac{\phi^{n+1}-\phi^{n}}{\tau},-\Delta \phi^{n+1})_\Omega + (\frac{\phi^{n+1}-\phi^{n}}{\tau},f(\phi^n) )_\Omega+ (\frac{\phi^{n+1}-\phi^{n}}{\tau},s_1 (\phi^{n+1}-\phi^n))_\Omega\notag\\
    =&-(\partial_{\bn} \phi,\frac{\phi^{n+1}-\phi^{n}}{\tau})_\Gamma +\frac{1}{\tau}(\nabla \phi^{n+1}-\nabla \phi^n,\nabla \phi^{n+1})_\Omega\notag\\
    &+\frac{1}{\tau}(F(\phi^{n+1})-F(\phi^{n}),1)_\Omega-\frac{F''(\xi)}{2\tau}||\phi^{n+1}-\phi^n||^2+\frac{s_1}{\tau}||\phi^{n+1}-\phi^n||^2\notag\\
    =&-(\partial_{\bn} \phi,\frac{\phi^{n+1}-\phi^{n}}{\tau})_\Gamma +\frac{1}{2\tau}(||\nabla \phi^{n+1}||^2-||\nabla \phi^{n}||^2+||\nabla \phi^{n+1}-\nabla \phi^{n}||^2)\notag\\
    &+\frac{1}{\tau}(F(\phi^{n+1})-F(\phi^{n}),1)_\Omega+(\frac{s_1}{\tau}-\frac{F''(\xi)}{2\tau})||\phi^{n+1}-\phi^n||^2,\label{3.14}
    \end{align}
 where we use the identity
    \beqlb
    (2a,(a-b))=|a|^2-|b|^2+|a-b|^2,\label{id1}
    \eeqlb
    and the Taylor expansion
   \beqlb (f\left(\phi^{n}\right),\left(\phi^{n+1}-\phi^{n}\right))_{\Omega}=(F\left(\phi^{n+1}\right)-F\left(\phi^{n}\right),1)_{\Omega}-\frac{F^{\prime \prime}(\xi)}{2}||\phi^{n+1}-\phi^{n}||^{2}.
   \eeqlb
Similarly, by applying the inverse Laplace-Beltrami operator $\Delta_{\Gamma}^{-1}$ to \eqref{3.4}, we obtain
\begin{equation}
    \beta_2 \Delta_\Gamma^{-1}\frac{\Psi^{n+1}-\Psi^{n}}{\tau}+\Delta_\Gamma^{-1}\Psi^{n+1} = M_2 \mu^{n+1}_\Gamma.
    \label{3.15}
\end{equation}
By taking the $L^2(\Gamma)$ inner product of \eqref{3.15} with $\frac{ 1}{M_2}\Psi^{n+1}$, we have
\begin{equation}
    \frac{\beta_2}{M_2}(\Delta_\Gamma^{-1}\frac{\Psi^{n+1}-\Psi^{n}}{\tau},\Psi^{n+1})_\Gamma+\frac{1}{M_2}(\Delta^{-1}_\Gamma\Psi^{n+1},\Psi^{n+1})_\Gamma = (\mu_\Gamma^{n+1},\Psi^{n+1})_{\Gamma}.
    \label{3.16}
\end{equation}
Noting that $q^{n+1} = \Delta_\Gamma^{-1} \Psi^{n+1}$,  we can get
    \begin{align}
    &\frac{\beta_2}{M_2}(\Delta_\Gamma^{-1}\frac{\Psi^{n+1}-\Psi^{n}}{\tau},\Psi^{n+1})_\Gamma+\frac{1}{M_2}(\Delta^{-1}_\Gamma\Psi^{n+1},\Psi^{n+1})_\Gamma\notag \\
    =&\frac{\beta_2}{M_2\tau}(q^{n+1}-q^n,\Delta_\Gamma q^{n+1})_\Gamma + \frac{1}{M_2}(q^{n+1},\Delta_\Gamma q^{n+1})_\Gamma\notag\\
    =&-\frac{\beta_2}{M_2\tau}(\nabla_\Gamma q^{n+1}-\nabla_\Gamma q^n,\nabla_\Gamma q^{n+1})_\Gamma -\frac{1}{M_2}||\nabla_\Gamma q^{n+1}||_\Gamma^2\notag\\
    =&-\frac{\beta_2}{2M_2\tau}(||\nabla_\Gamma q^{n+1}||_\Gamma^2-||\nabla_\Gamma q^n||_\Gamma^2+||\nabla_\Gamma q^{n+1}-\nabla_\Gamma q^{n}||_\Gamma^2)-\frac{1}{M_2}||\nabla_\Gamma q^{n+1}||_\Gamma^2,\label{3.17}
    \end{align}
and
    \begin{align}
    &(\mu_\Gamma^{n+1},\Psi^{n+1})_{\Gamma}=( \mu_{\Gamma}^{n+1},\frac{\psi^{n+1}-\psi^{n}}{\tau})_\Gamma\notag\\
    =&(\frac{\psi^{n+1}-\psi^{n}}{\tau},-\Delta_\Gamma \psi^{n+1} + g(\psi^n) +\partial_{\bn}\phi^{n+1}+ s_2 (\psi^{n+1}-\psi^n))_\Gamma\notag\\
    =&(\partial_{\bn} \phi,\frac{\psi^{n+1}-\psi^{n}}{\tau})_\Gamma+\frac{1}{\tau}(\nabla_\Gamma \psi^{n+1}-\nabla_\Gamma \psi^n,\nabla_\Gamma \psi^{n+1})_\Gamma\notag\\
    &+\frac{1}{\tau}(G(\psi^{n+1})-G(\psi^{n}),1)_\Gamma-\frac{G''(\eta)}{2\tau}||\psi^{n+1}-\psi^n||_\Gamma^2+\frac{s_2}{\tau}||\psi^{n+1}-\psi^n||_\Gamma^2\notag\\
    &=(\partial_{\bn} \phi,\frac{\psi^{n+1}-\psi^{n}}{\tau})_\Gamma +\frac{1}{2\tau}(||\nabla_\Gamma \psi^{n+1}||_{\Gamma}^2-||\nabla_\Gamma \psi^{n}||_{\Gamma}^2+||\nabla_\Gamma \psi^{n+1}-\nabla_\Gamma \psi^{n}||_{\Gamma}^2)\notag\\
    &+\frac{1}{\tau}(G(\psi^{n+1})-G(\psi^{n}),1)_\Gamma+(\frac{s_2}{\tau}-\frac{G''(\eta)}{2\tau})||\psi^{n+1}-\psi^n||_\Gamma^2,\label{3.18}
    \end{align}
where we use the Taylor expansion
\beqlb
(g\left(\psi^{n}\right),\left(\psi^{n+1}-\psi^{n}\right))_{\Gamma}=(G\left(\psi^{n+1}\right)-G\left(\psi^{n}\right),1)_{\Gamma}-\frac{G^{\prime \prime}(\eta)}{2}||\psi^{n+1}-\psi^{n}||_{\Gamma}^{2}.
\eeqlb
By combining all the above equations , we have
    \begin{align}
&\frac{\mathcal{E}(\phi^{n+1},\psi^{n+1})-\mathcal{E}(\phi^{n},\psi^n)}{\tau}+\frac{1}{2\tau}(||\nabla \phi^{n+1}-\nabla \phi^{n}||^2+||\nabla_\Gamma \psi^{n+1}-\nabla_\Gamma \psi^{n}||_\Gamma^2)\notag\\
        &+(\frac{s_1}{\tau}-\frac{F''(\xi)}{2\tau})||\phi^{n+1}-\phi^n||^2+(\frac{s_2}{\tau}-\frac{G''(\eta)}{2\tau})||\psi^{n+1}-\psi^n||_\Gamma^2\notag\\
&+\frac{\beta_1}{2M_1\tau}||\nabla p^{n+1}-\nabla p^{n}||^2+\frac{\beta_2}{2M_2\tau}||\nabla_\Gamma q^{n+1}-\nabla_\Gamma q^{n}||_\Gamma^2\notag\\
        &+\frac{1}{M_1}||\nabla p^{n+1}||^2+\frac{1}{M_2}||\nabla_\Gamma q^{n+1}||_\Gamma^2 = 0.\label{3.19}
    \end{align}
Therefore, under the conditions $\displaystyle s_1 \ge  \frac{1}{2}\max_{\xi\in\mathbb{R}} F''(\xi)$ and $\displaystyle s_2 \ge \frac{1}{2} \max_{\eta\in\mathbb{R}} G''(\eta)$, the following energy dissipation law holds

\begin{equation}
\frac{\mathcal{E}(\phi^{n+1},\psi^{n+1})-\mathcal{E}(\phi^{n},\psi^n)}{\tau}\le-\frac{1}{M_1}||\nabla p^{n+1}||^2-\frac{1}{M_2}||\nabla_\Gamma q^{n+1}||_\Gamma^2 \le 0.
    \label{3.20}
\end{equation}
\end{proof} 

\section{Error estimates}
\label{sec4}
In this section we will show the error estimates for the phase function $\phi$ and $\psi$ in the semi-discrete  scheme \eqref{3.1}-\eqref{3.8}. 

We assume that the derivatives of $F'$ and $G'$ satisfy the Lipschitz condition,

\begin{equation}
    \begin{aligned}
        \max_{\phi\in R}|F''(\phi)|\le L_1, \\
        \max_{\psi\in R}|G''(\psi)|\le L_2.
    \end{aligned}
\end{equation}
This condition is necessary for error estimation.

For a sequence of the functions
$f^0,f^1,f^2,\dots,f^N$ in the Hilbert space $H$, we denote the
sequence by ${f_{\tau}}$ and define the following discrete norm for
${f_{\tau}}$:
\begin{align}
\label{equ:def_f_tau_norm}
\|f_{\tau}\|_{l^\infty(H)}=\max_{0\leq n\leq N}(\|f^n\|_H).
\end{align}

The meaning of 
$f \lesssim g$ 
 is that there is a generic constant $C$ such
that $f \leqslant C g$, where $C$ is independent of $\tau$
but possibly depends on the data and the solution.

Firstly we rewrite the PDE system \eqref{2.9}-\eqref{2.13} in the following truncated form,
\begin{align}
&\beta_1\frac{\Phi(t^{n+1})-\Phi(t^{n})}{\tau} + \Phi(t^{n+1}) = M_1 \Delta \mu(t^{n+1}) + R^{n+1}_\phi,& \text { in } \Omega,\label{4.2}\\
&\mu(t^{n+1}) = -\Delta \phi(t^{n+1}) +  F'(\phi(t^{n})) + s_1 (\phi(t^{n+1})-\phi(t^{n}))+R^{n+1}_{\mu},& \text { in } \Omega,\label{4.3}
\\
&\Phi(t^{n+1}) = \frac{\phi(t^{n+1})-\phi(t^{n})}{\tau} + R_{\Phi}^{n+1},& \text { in } \Omega,\label{4.7}\\
&\phi(t^{n+1})|_\Gamma = \psi(t^{n+1}),& \text { on } \Gamma,\label{4.4}\\
&\beta_2\frac{\Psi(t^{n+1})-2\Psi(t^{n})}{\tau} + \Psi(t^{n+1}) = M_2 \Delta_\Gamma \mu_\Gamma(t^{n+1})+ R^{n+1}_\psi,& \text { on } \Gamma,\label{4.5}\\
&\mu_\Gamma(t^{n+1}) =  -\Delta_\Gamma \psi(t^{n+1}) + G'(\phi(t^{n})) + \partial_{\bn} \phi(t^{n+1}) + s_2(\psi(t^{n+1})-\psi(t^{n}))+R^{n+1}_{\Gamma},& \text { on } \Gamma,\label{4.6}\\
&\Psi(t^{n+1}) = \frac{\psi(t^{n+1})-\psi(t^{n})}{\tau} + R_{\Psi}^{n+1},& \text { on } \Gamma,\label{4.8}\\
&\partial_\mathbf{n} \mu(t^{n+1}) = 0,&\text { on } \Gamma, \label{4.9}
\end{align}
where the truncation errors
\begin{align}
&R^{n+1}_{\phi} = \beta_1\frac{\Phi(t^{n+1})-\Phi(t^{n})}{\tau} + \Phi(t^{n+1})-\beta_1 \phi_{tt}(t^{n+1})-\phi_t(t^{n+1}),\\     &R^{n+1}_{\mu} =F'(\phi(t^{n+1}))-F'(\phi(t^{n}))-s_1(\phi(t^{n+1})-\phi(t^n)),\\&R^{n+1}_{\Phi} =   \phi_t(t^{n+1})-\frac{\phi(t^{n+1})-\phi(t^{n})}{\tau},\\
&R^{n+1}_{\psi} = \beta_2\frac{\Psi(t^{n+1})-2\Psi(t^{n})}{\tau} + \Psi(t^{n+1})-\beta_2 \psi_{tt}(t^{n+1})-\psi_t(t^{+1}),\\
&R^{n+1}_{\Gamma} = G'(\phi(t^{n+1}))-G'(\phi(t^{n}))-s_2(\psi(t^{n+1})-\psi(t^n)),\\
&R^{n+1}_{\Psi} =\psi_t(t^{n+1})-\frac{\psi(t^{n+1})-\psi(t^{n})}{\tau}.
    \end{align}

We assume that the exact solution of the system \eqref{2.9}-\eqref{2.14} possesses the following regularity,
\begin{equation}
\begin{aligned}
&\phi,\phi_t,\phi_{tt},\phi_{ttt} \in L^{\infty}(0,T;H^3(\Omega)),\\
&\Delta^{-1}\phi,\Delta^{-1}\phi_t,\Delta^{-1}\phi_{tt},\Delta^{-1}\phi_{ttt} \in L^{\infty}(0,T;H^3(\Omega)),\\
&\mu \in L^{\infty}(0,T;H^2(\Omega)),\\
&\psi,\psi_t,\psi_{tt},\psi_{ttt} \in L^{\infty}(0,T;H^3(\Gamma)),\\
&\Delta_{\Gamma}^{-1}\psi,\Delta_{\Gamma}^{-1}\psi_t,\Delta_{\Gamma}^{-1}\psi_{tt},\Delta_{\Gamma}^{-1}\psi_{ttt} \in L^{\infty}(0,T;H^3(\Gamma)),\\
&\mu_\Gamma \in L^{\infty}(0,T;H^2(\Gamma)).\\
    \end{aligned}
    \label{4.11}
\end{equation}

By using the Taylor expansion, the following lemma can be easily proven.

%Firstly we show the $L^{\infty}(\Omega\cup \Gamma)$ stability of $\phi^n$ and $\psi^n$ in the following Lemma, which plays an important point to derive the error estimate. Due to the regularity assumptions \eqref{4.11}, we define a positve constant $\kappa$ such that
%\begin{align}
%    \kappa=\max_{0\leq t\leq T}||\phi(t)||_{L^{\infty}(\Omega\cup \Gamma)}+1=\max_{0\leq t\leq T}\max\{||\phi(t)||_{L^{\infty}(\Omega)},||\psi(t)||_{L^{\infty}(\Gamma)}\}+1.
%\end{align}

\begin{lemma}
Under the regularity assumption \eqref{4.11}, the truncation errors satisfy,
\begin{equation}
\begin{aligned}
&\|\nabla R_{\Phi,\tau}\|_{l^\infty(L^2({\Omega}))}+\|\nabla R_{\phi,\tau}\|_{l^\infty(L^2({\Omega}))}+\|\nabla R_{\mu,\tau}\|_{l^\infty(L^2({\Omega}))}\lesssim \tau,\\
&\|R_{\Phi,\tau}\|_{l^\infty(L^2({\Omega}))}+\|R_{\phi,\tau}\|_{l^\infty(L^2({\Omega}))}+\|R_{\mu,\tau}\|_{l^\infty(L^2({\Omega}))}\lesssim \tau,\\
&\|\nabla_{\Gamma}R_{\Psi,\tau}\|_{l^\infty(L^2({\Gamma}))}+\|\nabla_{\Gamma}R_{\psi,\tau}\|_{l^\infty(L^2({\Gamma}))}+\|\nabla_{\Gamma}R_{\Gamma,\tau}\|_{l^\infty(L^2({\Gamma}))}\lesssim \tau,\\
&\|R_{\Psi,\tau}\|_{l^\infty(L^2({\Gamma}))}+\|R_{\psi,\tau}\|_{l^\infty(L^2({\Gamma}))}+\|R_{\Gamma,\tau}\|_{l^\infty(L^2({\Gamma}))}\lesssim \tau,\\
&||\Delta ^{-1}R_{\phi,\tau}||_{l^\infty(L^2({\Omega}))}+||\Delta ^{-1}R_{\Phi,\tau}||_{l^\infty(L^2({\Omega}))}\lesssim \tau.\\
&||\Delta_{\Gamma} ^{-1}R_{\psi,\tau}||_{l^\infty(L^2({\Gamma}))}+||\Delta_{\Gamma} ^{-1}R_{\Psi,\tau}||_{l^\infty(L^2({\Gamma}))}\lesssim \tau,\\
&||\nabla\Delta ^{-1}R_{\phi,\tau}||_{l^\infty(L^2({\Omega}))}\lesssim \tau,\;\;||\nabla_{\Gamma}\Delta_{\Gamma}^{-1}R_{\psi,\tau}||_{l^\infty(L^2({\Gamma}))}\lesssim \tau.
    \end{aligned}
\end{equation}
Here the corresponding sequences of the truncation errors  are denoted as $\{R_{\Phi,\tau}\}$, $\{R_{\phi,\tau}\}$, $\{R_{\Psi,\tau}\}$, $\{R_{\psi,\tau}\}$, $\{R_{\mu,\tau}\}$, and $\{R_{\Gamma,\tau}\}$ with  the time step size $\tau$.
\end{lemma}

Then we establish the error estimate as follows.

\begin{theorem}
    If the exact solution is sufficiently smooth, or satisfies the assumption \eqref{4.11}, the solution $(\phi^m, \psi^m)$ for $0 \le m \le \left[\frac{T}{\tau}\right] - 1$ satisfies the following error estimate,
    \begin{equation}
\begin{aligned}
&\|e_{\phi,\tau}\|_{l^\infty(H^1({\Omega}))}+\|e_{\psi,\tau}\|_{l^\infty(H^1({\Gamma}))}\lesssim \tau,\\
&\|e_{\phi,\tau}\|_{l^\infty(L^2({\Omega}))}+\|e_{\psi,\tau}\|_{l^\infty(L^2({\Gamma}))}\lesssim \tau.\\
    \end{aligned}
\end{equation}
    Here, the error functions are denoted as,
\begin{equation}
\begin{aligned}
            & e^n_\phi=\phi(t^n)-\phi^n,\quad e^n_\mu=\mu(t^n)-\mu^n,\\
            & e^n_\psi=\psi(t^n)-\psi^n,\quad e^n_\Gamma=\mu_\Gamma(t^n)-\mu_\Gamma^n,\\
            & e^n_{\Phi}= \Phi(t^n)-\Phi^n,\quad e^n_{\Psi}=\Psi(t^n)-\Psi^n,\\
            & e^n_\phi|_\Gamma = e^n_\psi.
    \end{aligned}
\end{equation}
   Here the corresponding sequences of the error functions are defined as $\{e_{\phi,\tau}\}$, $\{e_{\Phi,\tau}\}$, $\{e_{\psi,\tau}\}$, $\{e_{\Psi,\tau}\}$, $\{e_{\mu,\tau}\}$, and $\{e_{\Gamma,\tau}\}$ with the time step $\tau$.

\end{theorem}

\begin{proof}
    Using mathematical induction, we prove that when $m=0$, we have $||e_\phi^0|| = ||e_\psi^0||_{\Gamma} = ||\nabla e_\phi^0|| = ||\nabla_\Gamma e^0_\psi||_{\Gamma} = 0$. Clearly, the above inequality holds. Now, assume that the error inequality holds for all $n \le m$. We need to show that the inequality also holds for $e^{m+1}_\phi$ and $e^{m+1}_\psi$. For all $n \le m$, by combining the numerical schemes and the truncation error equations, we obtain the following error equation,
    \begin{align}
&\frac{\beta_1}{\tau}(e_{\Phi}^{n+1}-e_{\Phi}^{n})+e_{\Phi}^{n+1}=M_1\Delta e_{\mu}^{n+1}+R_{\phi}^{n+1},& \text { in } \Omega,  \label{4.15}\\
&e_{\mu}^{n+1}=-\Delta e_{\phi}^{n+1}+F^{\prime}(\phi(t^{n}))-F^{\prime}(\phi^{n})+s_{1}(e_{\phi}^{n+1}-e_{\phi}^{n})+R_{\mu}^{n+1},& \text { in } \Omega, \label{4.16} \\
&e_{\Phi}^{n+1} = \frac{1}{\tau}(e_{\phi}^{n+1}-e_{\phi}^{n}) + R_{\Phi}^{n+1},& \text { in } \Omega,\label{4.21}\\
&\partial_{\mathbf{n}}e_{\mu}^{n+1}=0,& \text { on } \Gamma,\label{4.17} \\
&e_{\phi}^{n+1}|_{\Gamma}=e_{\psi}^{n+1},& \text { on } \Gamma,\label{4.18} \\
&\frac{\beta_2}{\tau}(e_{\Psi}^{n+1}-e^n_{\Psi})+e_{\Psi}^{n+1} =M_2\Delta_{\Gamma}e_{\Gamma}^{n+1}+R_{\psi}^{n+1},& \text { on } \Gamma, \label{4.19} \\
&e_{\Gamma}^{n+1}=-\Delta_{\Gamma}e_{\psi}^{n+1}+G^{\prime}(\psi(t^{n}))-G^{\prime}(\psi^{n})+\partial_{\mathbf{n}}e_{\phi}^{n+1}+s_{2}(e_{\psi}^{n+1}-e_{\psi}^{n})+R_{\Gamma}^{n+1},& \text { on } \Gamma,\label{4.20}\\
&e_{\Psi}^{n+1} = \frac{1}{\tau}(e_{\psi}^{n+1}-e_{\psi}^{n}) + R_{\Psi}^{n+1},& \text { on } \Gamma.\label{4.22}
\end{align}
    By applying the inverse Laplace 
 operator $\Delta^{-1}$ to \eqref{4.15} , we obtain
\begin{equation}
\frac{\beta_1}{\tau}\Delta^{-1}(e_\Phi^{n+1}-e_\Phi^{n})+\Delta^{-1}e_\Phi^{n}=M_1 e_{\mu}^{n+1}+\Delta^{-1}R_{\phi}^{n+1}.
    \label{4.23}
\end{equation}
    By taking the $L^2(\Omega)$ inner product of \eqref{4.23} with $\frac{e^{n+1}_\phi-e^{n}_\phi}{M_1}$, we have
\begin{align}
&\frac{\beta_1}{M_1\tau}(\Delta^{-1}(e_\Phi^{n+1}-e_\Phi^{n}),e^{n+1}_\phi-e^{n}_\phi)_\Omega+\frac{1}{M_1}(\Delta^{-1}e_{\Phi}^{n+1},e^{n+1}_\phi-e^{n}_\phi)_\Omega \notag\\
 =& (e_{\mu}^{n+1},e^{n+1}_\phi-e^{n}_\phi)_\Omega+\frac{1}{M_1}(\Delta^{-1}R_{\phi}^{n+1},{e^{n+1}_\phi-e^{n}_\phi})_{\Omega}. \label{4.24}
    \end{align}
    To simplify the left hand side of the \eqref{4.24}, using \eqref{4.21} and letting $u^{n+1} =\Delta^{-1}e^{n+1}_\Phi$, we deduce
\begin{align}
&\frac{\beta_1}{M_1\tau}(\Delta^{-1}(e_\Phi^{n+1}-e_\Phi^{n}),e^{n+1}_\phi-e^{n}_\phi)_\Omega =\frac{\beta_1}{M_1}(\Delta^{-1}(e_\Phi^{n+1}-e_\Phi^{n}),e_\Phi^{n+1}-R_\Phi^{n+1})_\Omega\notag\\
=&\frac{\beta_1}{M_1}(u^{n+1}-u^{n},\Delta u^{n+1})_\Omega-\frac{\beta_1}{M_1}(\Delta^{-1}(e_\Phi^{n+1}-e_\Phi^{n}),R_{\Phi}^{n+1})_\Omega\notag\\
        =&-\frac{\beta_1}{M_1}(\nabla u^{n+1}-\nabla u^{n},\nabla u^{n+1})_\Omega -\frac{\beta_1}{M_1}(e_\Phi^{n+1}-e_\Phi^{n},\Delta^{-1}R_{\Phi}^{n+1})_\Omega\notag\\
        =&-\frac{\beta_1}{2M_1}(\|\nabla u^{n+1}\|^2-\|\nabla u^{n}\|^2+\|\nabla u^{n+1}-\nabla u^{n}\|^2)-\frac{\beta_1}{M_1}(e_\Phi^{n+1}-e_\Phi^{n},\Delta^{-1}R_{\Phi}^{n+1})_\Omega.  \label{4.25}
    \end{align}
    and
\begin{align}
        &\frac{1}{M_1}(\Delta^{-1}e_{\Phi}^{n+1},e^{n+1}_\phi-e^{n}_\phi)_\Omega=\frac{\tau}{M_1}(\Delta^{-1}e_{\Phi}^{n+1},e^{n+1}_\Phi-R_{\Phi}^{n+1})_\Omega\notag\\
    =&\frac{\tau}{M_1}(u^{n+1},\Delta u^{n+1})_\Omega-\frac{\tau}{M_1}(\Delta^{-1}e_{\Phi}^{n+1},R_{\Phi}^{n+1})_\Omega\notag\\
        =&-\frac{\tau}{M_1}\|\nabla u^{n+1}\|^2-\frac{\tau}{M_1}(e_{\Phi}^{n+1},\Delta^{-1}R_{\Phi}^{n+1})_\Omega.\label{4.26}
    \end{align}

    Thus, by combining \eqref{4.25} with \eqref{4.26} and using \eqref{4.15}, the left hand side of  \eqref{4.24} can be written as
\begin{align}
&\frac{\beta_1}{M_1\tau}(\Delta^{-1}(e_\Phi^{n+1}-e_\Phi^{n}),e^{n+1}_\phi-e^{n}_\phi)_\Omega+\frac{1}{M_1}(\Delta^{-1}e_{\Phi}^{n+1},e^{n+1}_\phi-e^{n}_\phi)_\Omega \notag\\
= &-\frac{\beta_1}{2M_1}(\|\nabla u^{n+1}\|^2-\|\nabla u^{n}\|^2+\|\nabla u^{n+1}-\nabla u^{n}\|^2)-\frac{\beta_1}{M_1}(e_\Phi^{n+1}-e_\Phi^{n},\Delta^{-1}R_{\Phi}^{n+1})_\Omega\notag\\
        &-\frac{\tau}{M_1}\|\nabla u^{n+1}\|^2-\frac{\tau}{M_1}(e_{\Phi}^{n+1},\Delta^{-1}R_{\Phi}^{n+1})_\Omega\notag\\
        =&-\frac{\beta_1}{2M_1}(\|\nabla u^{n+1}\|^2-\|\nabla u^{n}\|^2+\|\nabla u^{n+1}-\nabla u^{n}\|^2)-\frac{\tau}{M_1}\|\nabla u^{n+1}\|^2\notag\\
        &-\frac{\tau}{M_1}(M_1\Delta e_{\mu}^{n+1}+R_{\phi}^{n+1},\Delta^{-1}R_{\Phi}^{n+1})_\Omega.\label{4.27}
        \end{align}
    By using \eqref{4.16}, we obtain the right hand side of \eqref{4.24},
        \begin{align}
        &(e_{\mu}^{n+1},e^{n+1}_\phi-e^{n}_\phi)_\Omega+\frac{1}{M_1}(\Delta^{-1}R_{\phi}^{n+1},{e^{n+1}_\phi-e^{n}_\phi})_\Omega\notag\\
        =&(-\Delta e_{\phi}^{n+1}+(F^{\prime}(\phi(t^{n}))-F^{\prime}(\phi^{n}))+s_{1}(e_{\phi}^{n+1}-e_{\phi}^{n})+R_{\mu}^{n+1}+\frac{\Delta^{-1}R_{\phi}^{n+1}}{M_1},{e^{n+1}_\phi-e^{n}_\phi})_\Omega\notag\\
        =&(-\Delta e_{\phi}^{n+1},{e^{n+1}_\phi-e^{n}_\phi})_\Omega+((F^{\prime}(\phi(t^{n}))-F^{\prime}(\phi^{n})),{e^{n+1}_\phi-e^{n}_\phi})_\Omega\notag\\
        &+s_{1}((e_{\phi}^{n+1}-e_{\phi}^{n}),{e^{n+1}_\phi-e^{n}_\phi})_\Omega+(R_{\mu}^{n+1}+\frac{\Delta^{-1}R_{\phi}^{n+1}}{M_1},{e^{n+1}_\phi-e^{n}_\phi})_\Omega\notag\\
        =&-(\partial_\mathbf{n} e_\phi^{n+1},{e^{n+1}_\phi-e^{n}_\phi})_\Gamma+\frac12(\|\nabla e^{n+1}_\phi\|^2-\|\nabla e_\phi^{n}\|^2+\|\nabla e_\phi^{n+1}-\nabla e_\phi^{n}\|^2)+{s_1}\|e_\phi^{n+1}-e_\phi^n\|^2\notag\\
    &+((F^{\prime}(\phi(t^{n}))-F^{\prime}(\phi^{n})),{e^{n+1}_\phi-e^{n}_\phi})_\Omega+(R_{\mu}^{n+1}+\frac{\Delta^{-1}R_{\phi}^{n+1}}{M_1},{e^{n+1}_\phi-e^{n}_\phi})_\Omega.\label{4.28}
        \end{align}
    By combining \eqref{4.27} and \eqref{4.28}, we have
        \begin{align}
        &\frac{\beta_1}{2M_1}(\|\nabla u^{n+1}\|^2-\|\nabla u^{n}\|^2+\|\nabla u^{n+1}-\nabla u^{n}\|^2) +\frac{\tau}{M_1}\|\nabla u^{n+1}\|^2\notag\\
        &+\frac12(\|\nabla e^{n+1}_\phi\|^2-\|\nabla e_\phi^{n}\|^2+\|\nabla e_\phi^{n+1}-\nabla e_\phi^{n}\|^2)+{s_1}\|e_\phi^{n+1}-e_\phi^{n}\|^2\notag\\
        =&(\partial_{\bn} e_\phi^{n+1},{e^{n+1}_\phi-e^{n}_\phi})_\Gamma-((F^{\prime}(\phi(t^{n}))-F^{\prime}(\phi^{n})),{e^{n+1}_\phi-e^{n}_\phi})_\Omega\notag\\
        &-(R_{\mu}^{n+1}+\frac{\Delta^{-1}R_{\phi}^{n+1}}{M_1},{e^{n+1}_\phi-e^{n}_\phi})_\Omega-\frac{\tau}{M_1}(M_1\Delta e_{\mu}^{n+1}+R_{\phi}^{n+1},\Delta^{-1}R_{\Phi}^{n+1})_\Omega.\label{4.29}
        \end{align}
    Then taking the $L^2(\Omega)$ inner product of \eqref{4.21} with $\tau e_{\phi}^{n+1}$, we have
        \begin{align}
    (e_{\phi}^{n+1}-e_{\phi}^{n},e_{\phi}^{n+1})_{\Omega}=& \frac{1}{2} (||e_{\phi}^{n+1}||^2-||e_{\phi}^{n}||^2+
    ||e_{\phi}^{n+1}-e_{\phi}^{n}||^2)\notag\\      =&\tau(e_{\Phi}^{n+1},e_{\phi}^{n+1})_{\Omega}-\tau(R_{\Phi}^{n+1},e_{\phi}^{n+1})_{\Omega}\notag\\
    =&-\tau(\nabla u^{n+1},\nabla e_{\phi}^{n+1})_{\Omega}-\tau(R_{\Phi}^{n+1},e_{\phi}^{n+1})_{\Omega}.\label{4.21s}
        \end{align}
  
    Similarly, by applying the inverse Laplace Beltrami operator $\Delta_\Gamma^{-1}$ to \eqref{4.19} , we obtain
    \begin{equation}
        \frac{\beta_2}{\tau}\Delta_\Gamma^{-1}(e_\Psi^{n+1}-e_\Psi^{n})+\Delta_\Gamma^{-1}e_{\Psi}^{n+1}=M_2 e_{\Gamma}^{n+1}+\Delta^{-1}_\Gamma R_{\psi}^{n+1}.
        \label{4.30}
    \end{equation}
    By taking the $L^2(\Gamma)$ inner product of \eqref{4.30} with $\frac{e^{n+1}_\psi-e^{n}_\psi}{M_2}$, we get
        \begin{align}
        &\frac{\beta_2}{M_2\tau}(\Delta_\Gamma^{-1}(e_\Psi^{n+1}-e_\Psi^{n}),e^{n+1}_\psi-e^{n}_\psi)_\Gamma+\frac{1}{M_2}(\Delta_\Gamma^{-1}e_{\Psi}^{n+1},e^{n+1}_\psi-e^{n}_\psi)_\Gamma \notag\\
        &= (e_{\Gamma}^{n+1},{e^{n+1}_\psi-e^{n}_\psi})_\Gamma+(\Delta_\Gamma^{-1}R_{\psi}^{n+1},\frac{e^{n+1}_\psi-e^{n}_\psi}{M_2})_\Gamma.\label{4.31}
        \end{align}
    To simplify the left hand side of the \eqref{4.31}, using \eqref{4.22} and letting $v^{n+1} =\Delta_\Gamma^{-1}e_\Psi^{n+1}$, we obtain
        \begin{align}
        &\frac{\beta_2}{M_2\tau}(\Delta_\Gamma^{-1}(e_\Psi^{n+1}-e_\Psi^{n}),e^{n+1}_\psi-e^{n}_\psi)_\Gamma=\frac{\beta_2}{M_2}(\Delta_\Gamma^{-1}(e_\Psi^{n+1}-e_\Psi^{n}),e^{n+1}_\Psi-R^{n+1}_\Psi)_\Gamma \notag\\
        =&\frac{\beta_2}{M_2}(v^{n+1}-v^{n},\Delta_\Gamma v^{n+1})_\Gamma-\frac{\beta_2}{M_2}(\Delta_\Gamma^{-1}(e_\Psi^{n+1}-e_\Psi^{n}),R^{n+1}_\Psi)_\Gamma\notag\\
        =&-\frac{\beta_2}{M_2}(\nabla_\Gamma v^{n+1}-\nabla_\Gamma v^{n},\nabla_\Gamma v^{n+1})_\Gamma -\frac{\beta_2}{M_2}((e_\Psi^{n+1}-e_\Psi^{n}),\Delta_\Gamma^{-1}R^{n+1}_\Psi)_\Gamma\notag\\
        =&-\frac{\beta_2}{2M_2}(\|\nabla_\Gamma v^{n+1}\|_\Gamma^2-\|\nabla_\Gamma v^{n}\|_\Gamma^2+\|\nabla_\Gamma (v^{n+1}- v^{n})\|_\Gamma^2) -\frac{\beta_2}{M_2}((e_\Psi^{n+1}-e_\Psi^{n}),\Delta_\Gamma^{-1}R^{n+1}_\Psi)_\Gamma,\label{4.32}
        \end{align}
    and
        \begin{align}
        &\frac{1}{M_2}(\Delta_\Gamma^{-1}e_{\Psi}^{n+1},e^{n+1}_\psi-e^{n}_\psi)_\Gamma=\frac{\tau}{M_2}(\Delta_\Gamma^{-1}e_{\Psi}^{n+1},e^{n+1}_\Psi-R^{n+1}_\Psi)_\Gamma\notag\\
        =&\frac{\tau}{M_2}(v^{n+1},\Delta_\Gamma v^{n+1})_\Gamma-\frac{\tau}{M_2}(\Delta_\Gamma^{-1}e_{\Psi}^{n+1},R^{n+1}_\Psi)_\Gamma
        =-\frac{\tau}{M_2}\|\nabla_\Gamma v^{n+1}\|_\Gamma^2-\frac{\tau}{M_2}(e_{\Psi}^{n+1},\Delta_\Gamma^{-1}R^{n+1}_\Psi)_\Gamma.\label{4.33}
        \end{align}

    By combining \eqref{4.32} with \eqref{4.33} and using \eqref{4.19}, the left hand side of the \eqref{4.31} can be written as

        \begin{align}
        &\frac{\beta_2}{M_2\tau}(\Delta_\Gamma^{-1}(e_\Psi^{n+1}-e_\Psi^{n}),e^{n+1}_\psi-e^{n}_\psi)_\Gamma+\frac{1}{M_2}(\Delta_\Gamma^{-1}e_{\Psi}^{n+1},e^{n+1}_\psi-e^{n}_\psi)_\Gamma \notag\\
        =&-\frac{\beta_2}{2M_2}(\|\nabla_\Gamma v^{n+1}\|_\Gamma^2-\|\nabla_\Gamma v^{n}\|_\Gamma^2+\|\nabla_\Gamma (v^{n+1}-v^{n})\|_\Gamma^2) -\frac{\beta_2}{M_2}((e_\Psi^{n+1}-e_\Psi^{n}),\Delta_\Gamma^{-1}R^{n+1}_\Psi)_\Gamma\notag\\
        &-\frac{\tau}{M_2}\|\nabla_\Gamma v^{n+1}\|_\Gamma^2-\frac{\tau}{M_2}(e_{\Psi}^{n+1},\Delta_\Gamma^{-1}R^{n+1}_\Psi)_\Gamma\notag\\
        =&-\frac{\beta_2}{2M_2}(\|\nabla_\Gamma v^{n+1}\|_\Gamma^2-\|\nabla_\Gamma v^{n}\|_\Gamma^2+\|\nabla_\Gamma (v^{n+1}-v^{n})\|_\Gamma^2)-\frac{\tau}{M_2}\|\nabla_\Gamma v^{n+1}\|_\Gamma^2\notag\\
        &-\frac{\tau}{M_2}(M_2\Delta_\Gamma e_{\Gamma}^{n+1}+R_{\psi}^{n+1},\Delta_\Gamma^{-1}R_{\Psi}^{n+1})_\Gamma.\label{4.34}
        \end{align}
    By using \eqref{4.20}, we obtain the right hand side of the \eqref{4.31},
        \begin{align}
        &(e_{\Gamma}^{n+1},{e^{n+1}_\psi-e^{n}_\psi})_\Gamma+(\Delta_\Gamma^{-1}R_{\psi}^{n+1},\frac{e^{n+1}_\psi-e^{n}_\psi}{M_2})_\Gamma \notag\\
        =&(-\Delta_{\Gamma}e_{\psi}^{n+1}+G^{\prime}(\psi(t^{n}))-G^{\prime}(\psi^{n})+\partial_{\mathbf{n}}e_{\phi}^{n+1}+s_{2}(e_{\psi}^{n+1}-e_{\psi}^{n}),{e^{n+1}_\psi-e^{n}_\psi})_\Gamma\notag\\
&+(R_{\Gamma}^{n+1}+\frac{\Delta_\Gamma^{-1}R_{\psi}^{n+1}}{M_2},{e^{n+1}_\psi-e^{n}_\psi})_\Gamma\notag\\
        =&(-\Delta_\Gamma e_{\psi}^{n+1},{e^{n+1}_\psi-e^{n}_\psi})_\Gamma+(\partial_{\bn} e_\phi^{n+1},{e^{n+1}_\psi-e^{n}_\psi})_\Gamma+((G^{\prime}(\psi(t^{n}))-G^{\prime}(\psi^{n})),{e^{n+1}_\psi-e^{n}_\psi})_\Gamma\notag\\
        &+(s_{2}(e_{\psi}^{n+1}-e_{\psi}^{n}),{e^{n+1}_\psi-e^{n}_\psi})_\Gamma+(R_{\Gamma}^{n+1}+\frac{\Delta_\Gamma^{-1}R_{\psi}^{n+1}}{M_2},{e^{n+1}_\psi-e^{n}_\psi})_\Gamma\notag\\
        =&(\partial_{\bn} e_\phi^{n+1},{e^{n+1}_\psi-e^{n}_\psi})_\Gamma+\frac12(\|\nabla_\Gamma e^{n+1}_\psi\|_\Omega^2-\|\nabla_\Gamma e_\psi^{n}\|_\Gamma^2+\|\nabla_\Gamma (e_\psi^{n+1}-e_\psi^{n})\|_\Gamma^2)+{s_2}\|e_\psi^{n+1}-e_\psi^n\|^2_\Gamma\notag\\
        &+((G^{\prime}(\psi(t^{n}))-G^{\prime}(\psi^{n})),{e^{n+1}_\psi-e^{n}_\psi})_\Gamma+(R_{\Gamma}^{n+1}+\frac{\Delta_\Gamma^{-1}R_{\psi}^{n+1}}{M_2},{e^{n+1}_\psi-e^{n}_\psi})_\Gamma.\label{4.35}
        \end{align}
    By combining \eqref{4.34} and \eqref{4.35}, we have
        \begin{align}
        &\frac{\beta_2}{2M_2}(\|\nabla_\Gamma v^{n+1}\|_\Gamma^2-\|\nabla_\Gamma v^{n}\|_\Gamma^2+\|\nabla_\Gamma (v^{n+1}- v^{n})\|_\Gamma^2)+\frac{\tau}{M_2}\|\nabla_\Gamma v^{n+1}\|_\Gamma^2\notag\\
        &+\frac12(\|\nabla_\Gamma e^{n+1}_\psi\|_\Omega^2-\|\nabla_\Gamma e_\psi^{n}\|_\Gamma^2+\|\nabla_\Gamma (e_\psi^{n+1}- e_\psi^{n})\|_\Gamma^2)+{s_2}\|e_\psi^{n+1}-e_\psi^n\|^2_\Gamma\notag\\
        =&-(\partial_{\bn} e_\phi^{n+1},{e^{n+1}_\psi-e^{n}_\psi})_\Gamma-((G^{\prime}(\psi(t^{n}))-G^{\prime}(\psi^{n})),{e^{n+1}_\psi-e^{n}_\psi})_\Gamma\notag\\
        &-(R_{\Gamma}^{n+1}+\frac{\Delta_\Gamma^{-1}R_{\psi}^{n+1}}{M_2},{e^{n+1}_\psi-e^{n}_\psi})_\Gamma-\frac{\tau}{M_2}(M_2\Delta_\Gamma e_{\Gamma}^{n+1}+R_{\psi}^{n+1},\Delta_\Gamma^{-1}R_{\Psi}^{n+1})_\Gamma.\label{4.36}
        \end{align}
    By taking the $L^2(\Gamma)$ inner product of \eqref{4.22} with $\tau e_{\psi}^{n+1}$, we have
        \begin{align}
    (e_{\psi}^{n+1}-e_{\psi}^{n},e_{\psi}^{n+1})_{\Gamma}=& \frac{1}{2} (||e_{\psi}^{n+1}||_{\Gamma}^2-||e_{\psi}^{n}||_{\Gamma}^2+
    ||e_{\psi}^{n+1}-e_{\psi}^{n}||_{\Gamma}^2)\notag\\      =&\tau(e_{\Psi}^{n+1},e_{\psi}^{n+1})_{\Gamma}-\tau(R_{\Psi}^{n+1},e_{\psi}^{n+1})_{\Gamma}\notag\\
    =&-\tau(\nabla_{\Gamma} v^{n+1},\nabla_{\Gamma}e_{\psi}^{n+1})_{\Gamma}-\tau(R_{\Psi}^{n+1},e_{\psi}^{n+1})_{\Gamma}.\label{4.22s}
        \end{align}
   
    By combining \eqref{4.29}, \eqref{4.21s}, \eqref{4.36}, and \eqref{4.22s}, we can
obtain
        \begin{align}
        &\frac{\beta_1}{2M_1}(\|\nabla u^{n+1}\|^2-\|\nabla u^{n}\|^2+\|\nabla (u^{n+1}- u^{n})\|^2) +\frac{\tau}{M_1}\|\nabla u^{n+1}\|^2\notag\\
        &+\frac12(\|\nabla e^{n+1}_\phi\|^2-\|\nabla e_\phi^{n}\|^2+\|\nabla (e_\phi^{n+1}- e_\phi^{n})\|^2)+{s_1}\|e_\phi^{n+1}-e_\phi^n\|^2\notag\\
        &+\frac{1}{2} (||e_{\phi}^{n+1}||^2-||e_{\phi}^{n}||^2+
    ||e_{\phi}^{n+1}-e_{\phi}^{n}||^2)\notag\\
        &+\frac{\beta_2}{2M_2}(\|\nabla_\Gamma v^{n+1}\|_\Gamma^2-\|\nabla_\Gamma v^{n}\|_\Gamma^2+\|\nabla_\Gamma (v^{n+1}-v^{n})\|_\Gamma^2)+\frac{\tau}{M_2}\|\nabla_\Gamma v^{n+1}\|_\Gamma^2\notag\\
        &+\frac12(\|\nabla_\Gamma e^{n+1}_\psi\|_\Gamma^2-\|\nabla_\Gamma e_\psi^{n}\|_\Gamma^2+\|\nabla_\Gamma (e_\psi^{n+1}- e_\psi^{n})\|_\Gamma^2)+{s_2}\|e_\psi^{n+1}-e_\psi^n\|^2_\Gamma\notag\\
        &+\frac{1}{2} (||e_{\psi}^{n+1}||_{\Gamma}^2-||e_{\psi}^{n}||_{\Gamma}^2+
    ||e_{\psi}^{n+1}-e_{\psi}^{n}||_{\Gamma}^2)\notag\\
        =&-\frac{\tau}{M_1}(M_1\Delta e_{\mu}^{n+1}+R_{\phi}^{n+1},\Delta^{-1}R_{\Phi}^{n+1})_\Omega-\frac{\tau}{M_2}(M_2\Delta_\Gamma e_{\Gamma}^{n+1}+R_{\psi}^{n+1},\Delta_\Gamma^{-1}R_{\Psi}^{n+1})_\Gamma  (:=\text{term } A_{1}) \notag\\
        &-(R_{\mu}^{n+1}+\frac{\Delta^{-1}R_{\phi}^{n+1}}{M_1},{e^{n+1}_\phi-e^{n}_\phi})_\Omega-(R_{\Gamma}^{n+1}+\frac{\Delta_\Gamma^{-1}R_{\psi}^{n+1}}{M_2},{e^{n+1}_\psi-e^{n}_\psi})_\Gamma(:=\text{term } A_{2})\notag\\
        &-((F^{\prime}(\phi(t^{n}))-F^{\prime}(\phi^{n})),{e^{n+1}_\phi-e^{n}_\phi})_\Omega-((G^{\prime}(\psi(t^{n}))-G^{\prime}(\psi^{n})),{e^{n+1}_\psi-e^{n}_\psi})_\Gamma (:=\text{term } A_{3})\notag\\
&-\tau(\nabla u^{n+1},\nabla e_{\phi}^{n+1})_{\Omega}-\tau(\nabla_{\Gamma} v^{n+1},\nabla_{\Gamma}e_{\psi}^{n+1})_{\Gamma}(:=\text{term } A_4)\notag\\
&-\tau(R_{\Phi}^{n+1},e_{\phi}^{n+1})_{\Omega}-\tau(R_{\Psi}^{n+1},e_{\psi}^{n+1})_{\Gamma}(:=\text{term } A_5)\label{4.37}
        \end{align}

    To simplify the calculations, we define $H^n = F'(\phi(t^n)) - F'(\phi^n)$. Then it can be rewritten  as

    \begin{equation}
            H^n = e_\phi^n\int_0^1F''(s\phi(t^n)+(1-s)\phi^n)ds.
        \label{4.38}
    \end{equation}
    We obtain $\|H^n\| \lesssim \|e_\phi^n\|$ since $F''$ is bounded. By taking the gradient of $H^n$, we have
\begin{align}
            \nabla H^n&=F''(\phi(t^n))\nabla\phi(t^n)-F''(\phi^n)\nabla \phi^n\notag\\&=(F''(\phi(t^n))-F''(\phi^n))\nabla\phi(t^n)+F''(\phi^n)\nabla e_\phi^n.\label{4.39}
        \end{align}

    Since $F''$ is bounded and satisfies the Lipschitz condition as well as condition \eqref{4.11}, we have

    \begin{equation}
            \|\nabla H^n\| \lesssim \|e_\phi^n\|\|\phi(t^n)\|_{H^3(\Omega)}+\|\nabla e^n_\phi\|\lesssim\|e_\phi^n\|+\|\nabla e^n_\phi\|.%\lesssim\|\nabla e^n_\phi\|_\Omega
        \label{4.40}
    \end{equation}
    Similarly, we define $\tilde{H}^n = G'(\psi(t^n)) - G'(\psi^n)$. Since $G''$ is bounded and satisfies the Lipschitz condition as well as condition \eqref{4.11}, we have

        \begin{align}
            \|\tilde{H}^n\|_\Gamma \lesssim \|e_\psi^n\|_\Gamma,\quad \|\nabla_\Gamma \tilde{H}^n\|_\Gamma \lesssim\|e_\psi^n\|_\Gamma+\|\nabla e^n_\psi\|_\Gamma.\label{4.41}%\lesssim\|\nabla e^n_\psi\|_\Gamma
        \end{align}

    For the term $A_1$, we have
    \begin{align}
            &-\frac{\tau}{M_1}(M_1\Delta e_{\mu}^{n+1}+R_{\phi}^{n+1},\Delta^{-1}R_{\Phi}^{n+1})_\Omega-\frac{\tau}{M_2}(M_2\Delta_\Gamma e_{\Gamma}^{n+1}+R_{\psi}^{n+1},\Delta_\Gamma ^{-1}R_{\Psi}^{n+1})_\Gamma\notag\notag\\
            =&-\tau(e_{\mu}^{n+1},R_{\Phi}^{n+1})_\Omega-\frac{\tau}{M_1}(R_{\phi}^{n+1},\Delta^{-1}R_{\Phi}^{n+1})_\Omega-\tau(e_{\Gamma}^{n+1},R_{\Psi}^{n+1})_\Gamma-\frac{\tau}{M_2}(R_{\psi}^{n+1},\Delta_\Gamma ^{-1}R_{\Psi}^{n+1})_\Gamma\notag\\
            =&-\tau(-\Delta e_{\phi}^{n+1}+(F^{\prime}(\phi(t^{n}))-F^{\prime}(\phi^{n}))+s_{1}(e_{\phi}^{n+1}-e_{\phi}^{n})+R_{\mu}^{n+1},R_{\Phi}^{n+1})_\Omega\notag\\
            &-\tau(-\Delta_{\Gamma}e_{\psi}^{n+1}+(G^{\prime}(\psi(t^{n}))-G^{\prime}(\psi^{n}))+\partial_{\mathbf{n}}e_{\phi}^{n+1}+s_{2}(e_{\psi}^{n+1}-e_{\psi}^{n})+R_{\Gamma}^{n+1},R_{\Psi}^{n+1})_\Gamma\notag\\
            &-\frac{\tau}{M_1}(R_{\phi}^{n+1},\Delta^{-1}R_{\Phi}^{n+1})_\Omega-\frac{\tau}{M_2}(R_{\psi}^{n+1},\Delta_\Gamma ^{-1}R_{\Psi}^{n+1})_\Gamma\notag \\ 
       % \end{aligned} \nonumber \displaybreak[1] \\  % 此处允许分页
        %\begin{aligned}
            =&-\tau(\nabla e_{\phi}^{n+1},\nabla R_\Phi^{n+1})_\Omega-{\tau}(H^{n},R_\Phi^{n+1})_\Omega-s_{1}\tau(e_{\phi}^{n+1}-e_{\phi}^{n},R_\Phi^{n+1})_\Omega-\tau(R_{\mu}^{n+1},R_{\Phi}^{n+1})_\Omega\notag\\
            &-\tau(\nabla_\Gamma e_{\psi}^{n+1},\nabla_\Gamma R_{\Psi}^{n+1})_\Gamma-{\tau}(\tilde{H}^n,R_{\Psi}^{n+1})_\Gamma-s_{2}\tau(e_{\psi}^{n+1}-e_{\psi}^{n},R_{\Psi}^{n+1})_\Gamma-\tau(R_{\Gamma}^{n+1},R_{\Psi}^{n+1})_\Gamma\notag\\
            &-\frac{\tau}{M_1}(R_{\phi}^{n+1},\Delta^{-1}R_{\Phi}^{n+1})_\Omega-\frac{\tau}{M_2}(R_{\psi}^{n+1},\Delta_\Gamma ^{-1}R_{\Psi}^{n+1})_\Gamma\notag\\
            \le&\tau\|\nabla e_{\phi}^{n+1}\|\|\nabla R_\Phi^{n+1}\|+s_{1}\tau\|e_{\phi}^{n+1}-e_{\phi}^{n}\|\|R_\Phi^{n+1}\|+\tau\|R_{\mu}^{n+1}\|\|R_{\Phi}^{n+1}\|\notag\\
            &+\tau\|\nabla_\Gamma e_{\psi}^{n+1}\|_\Gamma\|\nabla_\Gamma R_{\Psi}^{n+1}\|_\Gamma+s_{2}\tau\|e_{\psi}^{n+1}-e_{\psi}^{n}\|_\Gamma\|R_{\Psi}^{n+1}\|_\Gamma+\tau\|R_{\Gamma}^{n+1}\|_\Gamma\|R_{\Psi}^{n+1}\|_\Gamma\notag\\
            &+\frac{\tau}{M_1}\|R_{\phi}^{n+1}\|\|\Delta^{-1}R_{\Phi}^{n+1}\|+\frac{\tau}{M_2}\|R_{\psi}^{n+1}\|_\Gamma\|\Delta_\Gamma ^{-1}R_{\Psi}^{n+1}\|_\Gamma\notag\\
            &+{\tau}||H^{n}||||R_\Phi^{n+1}||+{\tau}||\tilde{H}^n||_{\Gamma}||R_{\Psi}^{n+1}||_\Gamma\notag\\
            \leq&\frac{\tau}{2}\|\nabla e_{\phi}^{n+1}\|^2+\frac{\tau}{2}\|\nabla R_\Phi^{n+1}\|^2+\frac{s_{1}\tau}{2}\|e_{\phi}^{n+1}-e_{\phi}^{n}\|^2+\frac{s_{1}\tau}{2}\|R_\Phi^{n+1}\|^2\notag\\
            &+\frac{\tau}{2}\|\nabla_\Gamma e_{\psi}^{n+1}\|^2_\Gamma+\frac{\tau}{2}\|\nabla_\Gamma R_{\Psi}^{n+1}\|^2_\Gamma+\frac{s_{2}\tau}{2}\|e_{\psi}^{n+1}-e_{\psi}^{n}\|^2_\Gamma+\frac{s_{2}\tau}{2}\|R_{\Psi}^{n+1}\|^2_\Gamma\notag\\
            &+\frac{\tau}{2}\|R_{\mu}^{n+1}\|^2+\frac{\tau}{2}\|R_{\Phi}^{n+1}\|^2+\frac{\tau}{2}\|R_{\Gamma}^{n+1}\|^2_\Gamma+\frac{\tau}{2}\|R_{\Psi}^{n+1}\|_\Gamma^2\notag\\
            &+\frac{\tau}{2M_1}\|R_{\phi}^{n+1}\|^2+\frac{\tau}{2M_1}\|\Delta^{-1}R_{\Phi}^{n+1}\|^2+\frac{\tau}{2M_2}\|R_{\psi}^{n+1}\|^2_\Gamma+\frac{\tau}{2M_2}\|\Delta_\Gamma ^{-1}R_{\Psi}^{n+1}\|^2_\Gamma\notag\\
            &+\frac{\tau}{2}||H^{n}||^2+\frac{\tau}{2}||R_\Phi^{n+1}||^2+\frac{\tau}{2}||\tilde{H}^n||_{\Gamma}^2+\frac{\tau}{2}||R_{\Psi}^{n+1}||_\Gamma^2\notag\\
            \leq&C_1 \tau^3 + \frac{\tau}{2}\|\nabla e_{\phi}^{n+1}\|^2+\frac{\tau}{2}||e_{\phi}^{n}||^2+\frac{s_{1}\tau}{2}\|e_{\phi}^{n+1}-e_{\phi}^{n}\|^2\notag\\
            &+\frac{\tau}{2}\|\nabla_\Gamma e_{\psi}^{n+1}\|^2_\Gamma+\frac{\tau}{2}||e_{\psi}^{n}||_{\Gamma}^2+\frac{s_{2}\tau}{2}\|e_{\psi}^{n+1}-e_{\psi}^{n}\|^2_\Gamma,
        \label{4.42}
    \end{align}
    where $C_1$ is a constant independent of $\tau$ and we  use the estimates for the truncation terms $R^{n+1}_\phi,R^{n+1}_\psi,R^{n+1}_\mu,R^{n+1}_\Gamma,R^{n+1}_\Phi,R^{n+1}_\Psi$.%as follows:
    %\begin{align}
%&||R_{\phi}^{n+1}||^2\leq\tau^2\int_{t^n}^{t^{n+1}}\\
 %       &||\nabla R_{\phi}^{n+1}||^2
  %  \end{align}

For the term $A_2$, we have
\begin{align}
            &-(R_{\mu}^{n+1}+\frac{\Delta^{-1}R_{\phi}^{n+1}}{M_1},{e^{n+1}_\phi-e^{n}_\phi})_\Omega-(R_{\Gamma}^{n+1}+\frac{\Delta_\Gamma ^{-1}R_{\Psi}^{n+1}}{M_2},{e^{n+1}_\psi-e^{n}_\psi})_\Gamma\notag\\
            =&-\tau(R_{\mu}^{n+1}+\frac{\Delta^{-1}R_{\phi}^{n+1}}{M_1},{e^{n+1}_\Phi-R^{n+1}_\Phi})_\Omega-\tau(R_{\Gamma}^{n+1}+\frac{\Delta_\Gamma ^{-1}R_{\Psi}^{n+1}}{M_2},{e^{n+1}_\Psi-R^{n+1}_\Psi})_\Gamma\notag\\
            =&-\tau(R_{\mu}^{n+1}+\frac{\Delta^{-1}R_{\phi}^{n+1}}{M_1},{\Delta u^{n+1}-R^{n+1}_\Phi})_\Omega-\tau(R_{\Gamma}^{n+1}+\frac{\Delta_\Gamma ^{-1}R_{\Psi}^{n+1}}{M_2},{\Delta _{\Gamma}v^{n+1}-R^{n+1}_\Psi})_\Gamma\notag\\
            =&\tau(\nabla R_{\mu}^{n+1}+\frac{1}{M_1}\nabla {\Delta^{-1}R_{\phi}^{n+1}},\nabla u^{n+1})_\Omega+\tau(R_{\mu}^{n+1}+\frac{\Delta^{-1}R_{\phi}^{n+1}}{M_1},R^{n+1}_\Phi)_\Omega\notag\\
            &+\tau(\nabla_{\Gamma}R_{\Gamma}^{n+1}+\frac{1}{M_2}\nabla_{\Gamma}{\Delta_\Gamma ^{-1}R_{\Psi}^{n+1}},\nabla_\Gamma v^{n+1})_\Gamma+\tau(R_{\Gamma}^{n+1}+\frac{\Delta_\Gamma ^{-1}R_{\Psi}^{n+1}}{M_2},R^{n+1}_\Psi)_\Gamma\notag\\
            \le &{2\tau M_1}\|\nabla R_{\mu}^{n+1}\|^2+\frac{2\tau}{M_1}\|\nabla \Delta^{-1}R_{\phi}^{n+1}\|_\Omega^2+\frac{\tau}{4M_1}\|\nabla u^{n+1}\|^2\notag\\
            &+{2\tau M_1}\|R_{\mu}^{n+1}\|^2+\frac{2\tau}{M_1}\|\Delta^{-1}R_{\phi}^{n+1}\|^2+\frac{\tau}{4M_1}\| R_\Phi^{n+1}\|^2\notag\\
            &+{2\tau M_2}\|\nabla_\Gamma R_{\Gamma}^{n+1}\|^2_\Gamma+\frac{2\tau}{M_2}\|\nabla_\Gamma \Delta_\Gamma^{-1}R_{\psi}^{n+1}\|_\Gamma^2+\frac{\tau}{4M_2}\|\nabla_\Gamma v^{n+1}\|_\Gamma^2\notag\\
            &+{2\tau M_2}\|R_{\Gamma}^{n+1}\|^2_\Gamma+\frac{2\tau}{M_2}\|\Delta_\Gamma^{-1}R_{\psi}^{n+1}\|_\Gamma^2+\frac{\tau}{4M_2}\| R_\Psi^{n+1}\|_\Gamma^2\notag\\
            \le&C_2 \tau^3 +\frac{\tau}{4M_1}\|\nabla u^{n+1}\|^2+\frac{\tau}{4M_2}\|\nabla_\Gamma v^{n+1}\|_\Gamma^2,\label{4.43}
        \end{align}

    where $C_2$ is a constant independent of $\tau$. Here, we use the estimates for the truncation terms $R^{n+1}_\phi,R^{n+1}_\psi,R^{n+1}_\mu,R^{n+1}_\Gamma,R^{n+1}_\Phi,R^{n+1}_\Psi$.

    For the term $(A_3$, using the estimates for $H^n,\nabla H^n$, and $R_{\Phi}^n$, we have
        \begin{align}
            &-((F^{\prime}(\phi(t^{n}))-F^{\prime}(\phi^{n})),{e^{n+1}_\phi-e^{n}_\phi})_\Omega\notag\\
            =&-{\tau}(H^n,e^{n+1}_\Phi)_\Omega+{\tau}(H^n,{R^{n+1}_\Phi})_\Omega\notag\\
            =&-{\tau}(H^n,\Delta u^{n+1})_\Omega+{\tau}(H^n,{R^{n+1}_\Phi})_\Omega\notag\\
            =&{\tau}(\nabla H^n,\nabla u^{n+1})_\Omega+{\tau}(H^n,{R^{n+1}_\Phi})_\Omega\notag\\
            \le& {\tau}\|\nabla H^n\|\|\nabla u^{n+1}\|+{\tau}||H^n||||R^{n+1}_\Phi||\notag\\
            \leq& C_3\tau(||e_{\phi}^n||+||\nabla e_{\phi}^n||)\|\nabla u^{n+1}\|+C_4\tau ||e_{\phi}^n||||R^{n+1}_\Phi|| \notag\\
            \leq& {2C_3^2\tau M_1}\|\nabla e^n_\phi\|^2 + \frac{\tau}{4M_1}\|\nabla u^{n+1}\|^2+C_5\tau||e_{\phi}^n||^2+C_6\tau^3,\label{4.44}
        \end{align}

   where $C_i(i=3,4,5,6)$ are constants independent of $\tau$ and $C_5=2C_3^2+C_4/2$.
    Similarly, we can obtain
   \begin{align}
            &-((G^{\prime}(\psi(t^{n}))-G^{\prime}(\psi^{n})),{e^{n+1}_\psi-e^{n}_\psi})_\Gamma\notag\\
            =&-{\tau}(\tilde{H}^n,e^{n+1}_\Psi)_\Gamma+{\tau}(\tilde{H}^n,{R^{n+1}_\Psi})_\Gamma\notag\\
            =&-{\tau}(\tilde{H}^n,\Delta_\Gamma v^{n+1})_\Gamma+{\tau}(\tilde{H}^n,{R^{n+1}_\Psi})_\Gamma\notag\\
            =&{\tau}(\nabla_\Gamma\tilde{H}^n,\nabla_\Gamma v^{n+1})_\Gamma+{\tau}(\tilde{H}^n,{R^{n+1}_\Psi})_\Gamma\notag\\
            \le& {\tau}\|\nabla_\Gamma \tilde{H}^n\|_\Gamma\|\nabla_\Gamma v^{n+1}\|_\Gamma+{\tau}||\tilde{H}^n||_{\Gamma}||R^{n+1}_\Psi||_\Gamma\notag\\
            \leq&C_7\tau(||e_{\psi}^n||_{\Gamma}+||\nabla_{\Gamma}e_{\psi}^n||_{\Gamma})||\nabla_{\Gamma}v^{n+1}||_{\Gamma}+C_8\tau||e_{\psi}^n||_{\Gamma}||R_{\Psi}^{n+1}||_{\Gamma}
           \notag \\
            \leq& 2C_7^2\tau M_2\|\nabla_\Gamma e^n_\psi\|_\Gamma^2 + \frac{\tau}{4M_2}\|\nabla_\Gamma v^{n+1}\|^2_\Gamma+C_9\tau||e_{\psi}^n||_{\Gamma}^2+C_{10}\tau^3,\label{4.45}
        \end{align}

    where $C_i(i=7,8,9,10)$ are constants independent of $\tau$ and $C_9=2C_7^2+C_8/2$. Here we use the estimates for $\tilde{H}^n, \nabla_\Gamma \tilde{H}^n$ and $R_{\Psi}^{n+1}$.
    
 For the term $A_4$, we have 
     \begin{align}
   &-\tau(\nabla u^{n+1},\nabla e_{\phi}^{n+1})_{\Omega}-\tau(\nabla_{\Gamma} v^{n+1},\nabla_{\Gamma}e_{\psi}^{n+1})_{\Gamma}\notag\\
   \leq&\tau||\nabla u^{n+1}||||\nabla e_{\phi}^{n+1}||+\tau||\nabla_{\Gamma} v^{n+1}||_{\Gamma}||\nabla_{\Gamma}e_{\psi}^{n+1}||_{\Gamma}\notag\\
   \leq&M_1\tau||\nabla e_{\phi}^{n+1}||^2+\frac{\tau}{4M_1}||\nabla u^{n+1}||^2+M_2\tau||\nabla_{\Gamma}e_{\psi}^{n+1}||_{\Gamma}^2+\frac{\tau}{4M_2}||\nabla_{\Gamma} v^{n+1}||_{\Gamma}^2.
     \label{A4s}\end{align}
 
 We estimate the term $A_5$ as follows

     \begin{align}
 &-\tau(R_{\Phi}^{n+1},e_{\phi}^{n+1})_{\Omega}-\tau(R_{\Psi}^{n+1},e_{\psi}^{n+1})_{\Gamma} \notag\\
 \leq&\tau||R_{\Phi}^{n+1}||||e_{\phi}^{n+1}||+\tau||R_{\Psi}^{n+1}||_{\Gamma}||e_{\psi}^{n+1}||_{\Gamma}
 \leq C_{11}\tau^3+\frac{\tau}{2}||e_{\phi}^{n+1}||^2+    \frac{\tau}{2}||e_{\psi}^{n+1}||_{\Gamma}^2,\label{A5s}\end{align}
 where $C_{11}$ is a constant independent of $\tau$.
 
    By combining \eqref{4.37}, \eqref{4.42}, \eqref{4.43}, \eqref{4.44}, \eqref{4.45}, \eqref{A4s} and \eqref{A5s}, we simplify to obtain:
    
        \begin{align}
            &\frac{\beta_1}{2M_1}(\|\nabla u^{n+1}\|^2-\|\nabla u^{n}\|^2+\|\nabla u^{n+1}-\nabla u^{n}\|^2) +\frac{\tau}{4M_1}\|\nabla u^{n+1}\|^2\notag\\
    &+\frac12(\|\nabla e^{n+1}_\phi\|^2-\|\nabla e_\phi^{n}\|^2+\|\nabla e_\phi^{n+1}-\nabla e_\phi^{n}\|)+{s_1}\|e_\phi^{n+1}-e_\phi^n\|^2\notag\\
            &+\frac{1}{2} (||e_{\phi}^{n+1}||^2-||e_{\phi}^{n}||^2+
    ||e_{\phi}^{n+1}-e_{\phi}^{n}||^2) \notag\\&\frac{\beta_2}{2M_2}(\|\nabla_\Gamma v^{n+1}\|_\Gamma^2-\|\nabla_\Gamma v^{n}\|_\Gamma^2+\|\nabla_\Gamma v^{n+1}-\nabla_\Gamma v^{n}\|_\Gamma^2)+\frac{\tau}{4M_2}\|\nabla_\Gamma v^{n+1}\|_\Gamma^2\notag\\
            &+\frac12(\|\nabla_\Gamma e^{n+1}_\psi\|_\Omega^2-\|\nabla_\Gamma e_\psi^{n}\|_\Gamma^2+\|\nabla_\Gamma e_\psi^{n+1}-\nabla_\Gamma e_\psi^{n}\|_\Gamma^2)+{s_2}\|e_\psi^{n+1}-e_\psi^n\|^2_\Gamma\notag\\
     &+\frac{1}{2} (||e_{\psi}^{n+1}||_{\Gamma}^2-||e_{\psi}^{n}||_{\Gamma}^2+
    ||e_{\psi}^{n+1}-e_{\psi}^{n}||_{\Gamma}^2)\notag\\       
            \lesssim &\tau^3 + \tau (\|\nabla e_{\phi}^{n+1}\|^2+\|\nabla e_{\phi}^{n}\|^2+||e_{\phi}^{n+1}||^2+||e_{\phi}^{n}||^2+s_1\|e_{\phi}^{n+1}-e_{\phi}^{n}\|^2)\notag\\
        &+\tau(\|\nabla_\Gamma e_{\psi}^{n+1}\|^2_\Gamma+\|\nabla_\Gamma e_{\psi}^{n}\|^2_\Gamma+\| e_{\psi}^{n+1}\|^2_\Gamma+\| e_{\psi}^{n}\|^2_\Gamma+s_2\|e_{\psi}^{n+1}-e_{\psi}^{n}\|^2_\Gamma).\label{4.46}
        \end{align}

    Summing \eqref{4.46} together for \(n = 0 \) to \(m (m \le M)\), we have
  
        \begin{align}
            &\frac{1}{2}(\|\nabla e^{m+1}_\phi\|^2+\| e^{m+1}_\phi\|^2+\|\nabla_\Gamma e^{m+1}_\psi\|_\Gamma^2+\| e^{m+1}_\psi\|_\Gamma^2)+\frac{\beta_1}{2M_1}\|\nabla u^{m+1}\|^2+\frac{\beta_2}{2M_2}\|\nabla_\Gamma v^{m+1}\|_\Gamma^2\notag\\
            &+\sum_{n=0}^{m}\bigg(\frac{\beta_1}{2M_1}\|\nabla (u^{n+1}- u^{n})\|^2 +\frac{\tau}{4M_1}\|\nabla u^{n+1}\|^2+\frac12\|\nabla (e_\phi^{n+1}- e_\phi^{n})\|^2+\notag\\
            &+{(s_1+\frac12)}\|e_\phi^{n+1}-e_\phi^n\|^2+{(s_2+\frac12)}\|e_\psi^{n+1}-e_\psi^n\|^2_\Gamma\notag\\
            &+\frac{\beta_2}{2M_2}\|\nabla_\Gamma (v^{n+1}- v^{n})\|_\Gamma^2+\frac{\tau}{4M_2}\|\nabla_\Gamma v^{n+1}\|_\Gamma^2+\frac12\|\nabla_\Gamma (e_\psi^{n+1}- e_\psi^{n})\|_\Gamma^2\bigg)\notag\\
            \le& \tilde{C}(m+1)\tau^3 +\tilde{C}\tau\sum_{n=0}^{m}\bigg(\|\nabla e_{\phi}^{n+1}\|^2+\| e_{\phi}^{n+1}\|^2+\|e_{\phi}^{n+1}-e_{\phi}^{n}\|^2\notag\\
            &+\|\nabla_\Gamma e_{\psi}^{n+1}\|^2_\Gamma+\| e_{\psi}^{n+1}\|^2_\Gamma+\|e_{\psi}^{n+1}-e_{\psi}^{n}\|^2_\Gamma\bigg),\label{4.47}
        \end{align}
        
Where $\tilde{C}$ is a constant independent of $\tau$ and $||e_{\phi}^0||=||e_{\psi}^0||_{\Gamma}=||\nabla e_{\phi}^0||=||\nabla e_{\psi}^0||_{\Gamma}=||\nabla u^0||=||\nabla v^0||_{\Gamma}=0$.

Define that 
\begin{align}
    I_m=&\frac{1}{2}\big(\|\nabla e^{m+1}_\phi\|^2+\| e^{m+1}_\phi\|^2+\|\nabla_\Gamma e^{m+1}_\psi\|_\Gamma^2+\| e^{m+1}_\psi\|_\Gamma^2\big)\nonumber\\
&+{(s_1+\frac12)}\|e_\phi^{m+1}-e_\phi^m\|^2+{(s_2+\frac12)}\|e_\psi^{m+1}-e_\psi^m\|^2_\Gamma,
\end{align}
 and
 \begin{align}
  S_m= \sum_{n=0}^m \bigg( \frac12\|\nabla (e_\phi^{n+1}- e_\phi^{n})\|^2+\frac12\|\nabla_\Gamma (e_\psi^{n+1}- e_\psi^{n})\|_\Gamma^2\bigg).
 \end{align}
Then dropping the positive terms from the left side of \eqref{4.47}, we have
\begin{align}
    I_m+S_m\lesssim \tau^2+\tau\sum_{n=0}^mI_n.
\end{align}
    According to the discrete Gr\"{o}nwall's inequality, there exist constants $\tilde{c}_0$ and $C_0$, such that
    \begin{align}
I_m+S_m\leq\tilde{c}_0\tau^2,
    \end{align}
    where $\tilde{c}_0$ is independent of  $\tau$ and  $\tau \le C_0$.
Therefore  we obtain the error estimates for $e^{m+1}_\phi$ and $e^{m+1}_\psi$.
\end{proof} 

\section{Numerical experiments}
\label{sec5}
In this section, we will verify energy stability and the temporal accuracy of the scheme  \eqref{3.1}-\eqref{3.8} by testing some  numerical experiments.
 Here we use the second-order central finite difference method to discretize the space.
%The definitions for calculating energy and mass are as follows:
%\begin{align}
%E(\phi^{n},\psi^n)=&\int_\Omega(\frac{|\nabla \phi^n|^2}{2} +  F(\phi^n) +\frac{\beta_1}{2M_1}|\nabla p^n|^2)d\bx\notag\\
  %  &+ \int_\Gamma(\frac{|\nabla_\Gamma \psi^n|^2}{2} + G(\psi^n)+\frac{\beta_2}{2M_2}|\nabla_\Gamma q^n|^2)dS,\\ 
%M(\phi^n,\psi^n) =& \int_\Omega \phi^n d\bx + \int_\Gamma \psi^n dS.
%\end{align}
  For simplicity, if not explicit specified, the surface potential $G(\psi)=(\psi^2-1)^2/(4\delta^2)$ is chosen, the 2D square domain $\Omega\cup\Gamma=[0,1]\times[0,1]$ and $\Omega=(0,1)\times(0,1)$  are selected, and
the experimental parameters are set by default as,
\begin{align}
    \label{equ:parameters}
    \begin{aligned}
        &M_1=M_2=0.001,\; \tau=10^{-4}, \;h=1/100,\\
        &\beta_1=\beta_2=\beta=0,\; \varepsilon=\delta=2h,\; s_1=\frac{2}{\varepsilon^2},\;
         s_2=\frac{2}{\delta^2}.
    \end{aligned}
\end{align}
\subsection{Temporal accuracy test}

Firstly, we performed a convergence test of the numerical scheme to verify the  error analysis. The spatial step size is set to $h = 1/50= 0.02$ , and the time step $\tau$ is chosen as $0.01, 0.005, 0.0025, 0.00125$, $ 0.000625,$ and  $0.0003125$. %Other parameters $\beta_1, \beta_2, \varepsilon, \delta, M_1, M_2, s_1,$ and $s_2$ are set to $0, 0, 0.04, 0.04, 0.001, 0.001, 1250$, and $1250$, respectively.
The initial condition is specified as zero in the interior domain $\Omega = (0,1) \times (0,1)$ and one on the boundary $\Gamma = \partial\Omega$.  The complete spatial domain $\Omega\cup\Gamma$, is the closed unit square $[0,1] \times [0,1]$.

Then we select $\tau = 1 \times 10^{-6}$ as the reference solution, and the error is carried out between the reference solution and the numerical solution with
different time step sizes at $T=1$.% We plotted the $L^2$ errors of the numerical solution against the reference solution at $T = 1$ for different time steps, as shown 
 The figure \ref{5.1}  indicates that the convergence rate of the numerical scheme is asymptotic of first order in time, which is consistent with the error analysis in Section \ref{sec4}.

\begin{figure}[h]
    \centering
    \includegraphics[width=0.7\textwidth]{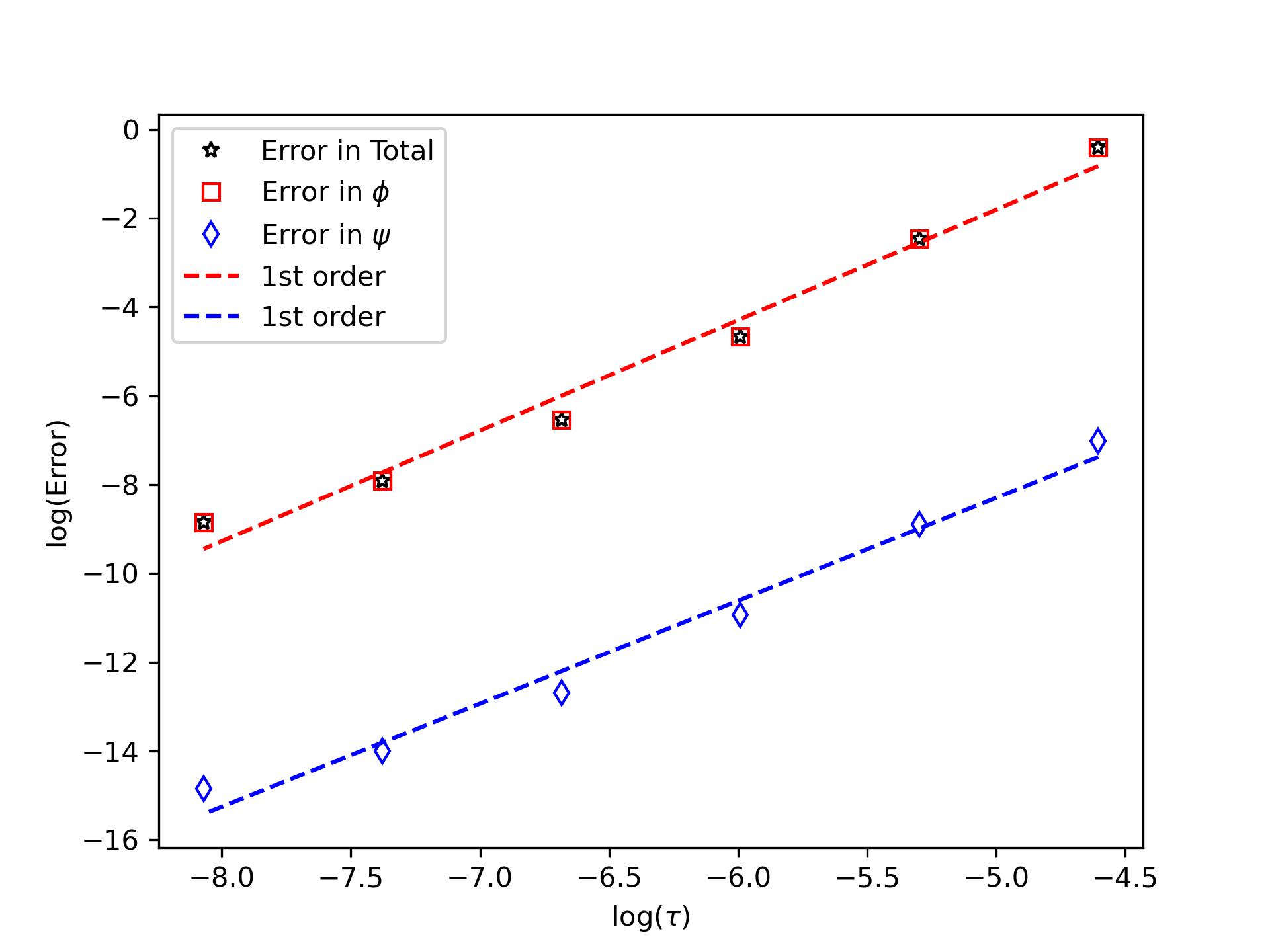} 
    \caption{The $L^2$ numerical errors for $\phi$ and $\psi$ at $T = 1.0$.} 
    \label{5.1} 
\end{figure}

\subsection{The effect of the hyperbolic term}

In this section, we consider the effect of the hyperbolic term on the system by changing the value of the parameter $\beta$. % compare the variation of different $\beta_1$ and $\beta_2$ values at the same time step, and based on the analysis, we draw conclusions. We set $\Omega$ as a square region with a side length of 1, with a spatial step size of $h = 1/100 = 0.01$ and a time step $dt = 0.0005$. The other parameters $\varepsilon, \delta, M_1, M_2, s_1$, and $s_2$ are set to $2h, 2h, 0.01, 0.01, \frac{2}{\varepsilon^2}$, and $\frac{2}{\delta^2}$, respectively. 

\textbf{Case 1:} For the initial value, we set in Figure \ref{5.2},
\begin{align}
 \phi_0(x,y)=
    \begin{cases}
        0, \quad & (x,y) \text{ in} \; \Omega,\\
        1, \quad & (x,y)\text{ on} \; \Gamma.
    \end{cases}
    \label{equ:initial data of 8.1}
\end{align}

 In this case, the parameter $\beta$  is set to the values $1,$ $ 0.1,$ and $0$.
 Then We  obtain the time evolution of the numerical solutions for 
$\phi$, as shown in Figure \ref{5.4}, along with the corresponding energy and mass evolutions depicted in Figure \ref{5.3}.
We observe that the system reaches the steady state more slowly when the value of 
$\beta$ is larger. Meanwhile we find that the discrete energy is decreasing fast by reducing the value of $\beta$. We also see that the mass conservation in the bulk and on the boundary is holding during the computation.

\begin{figure}[!htbp]
    \centering
    \includegraphics[width=0.3\textwidth]{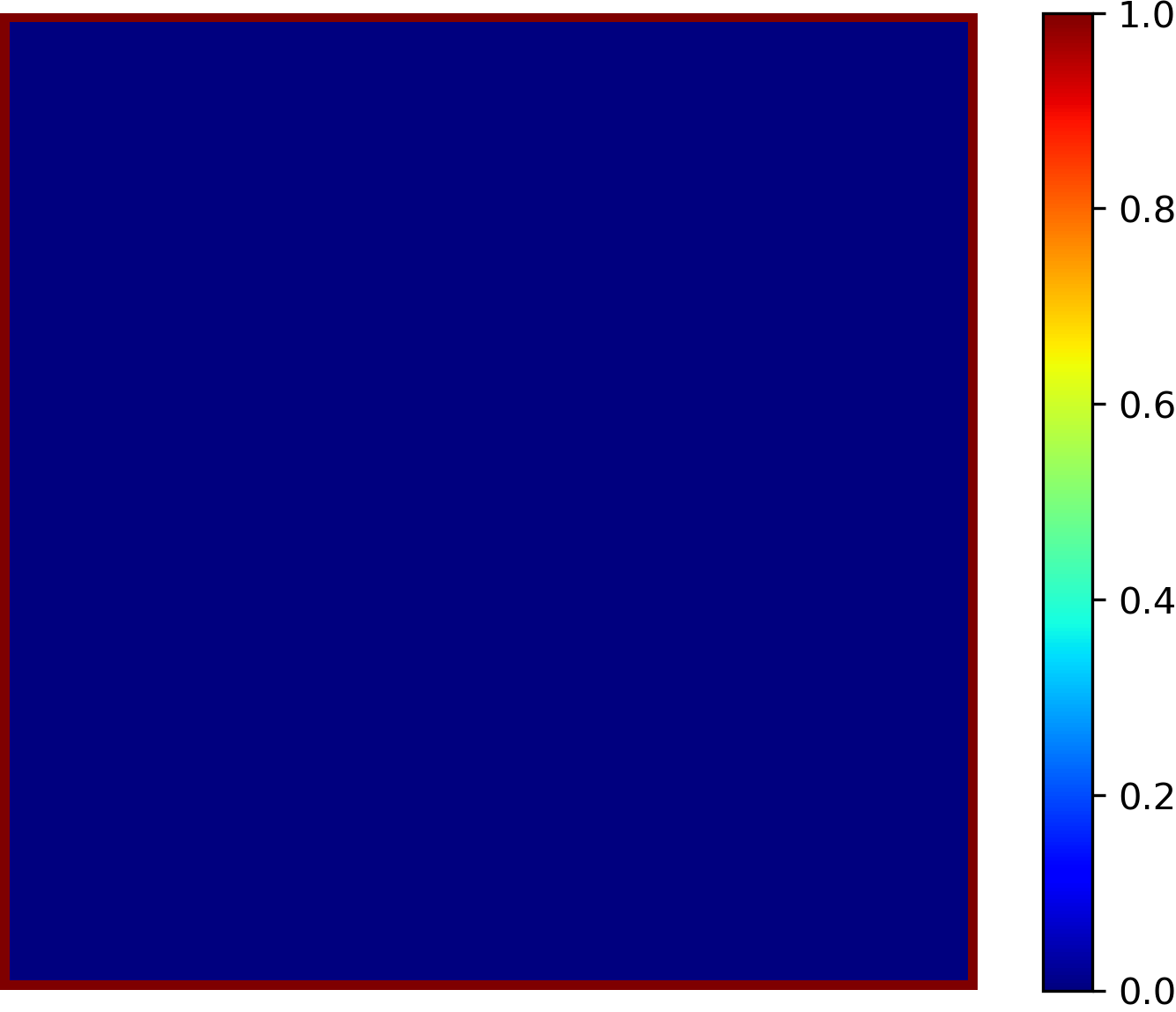}
    \caption{The initial data of Case 1.}

    \label{5.2}

\end{figure}

\begin{figure}[!htbp]
    \centering
    \scalebox{0.8}{
        \begin{minipage}{\textwidth}
            \begin{subfigure}[b]{0.49\textwidth}
                \includegraphics[width=\textwidth]{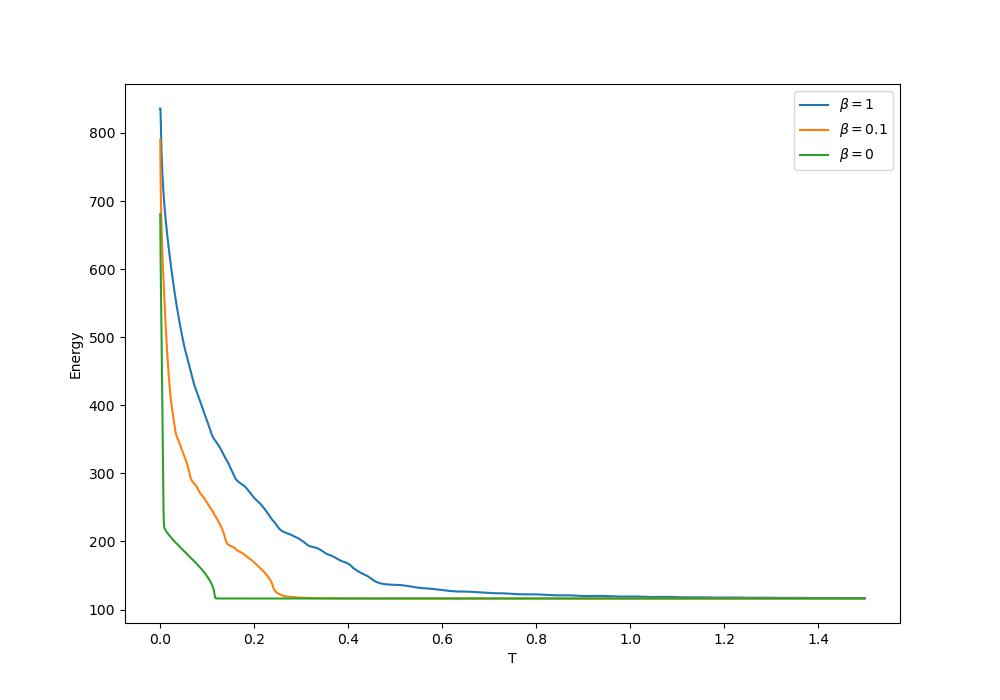}
                \caption{Energy curves with different $\beta$.}
            \end{subfigure}
            \begin{subfigure}[b]{0.49\textwidth}
                \includegraphics[width=\textwidth]{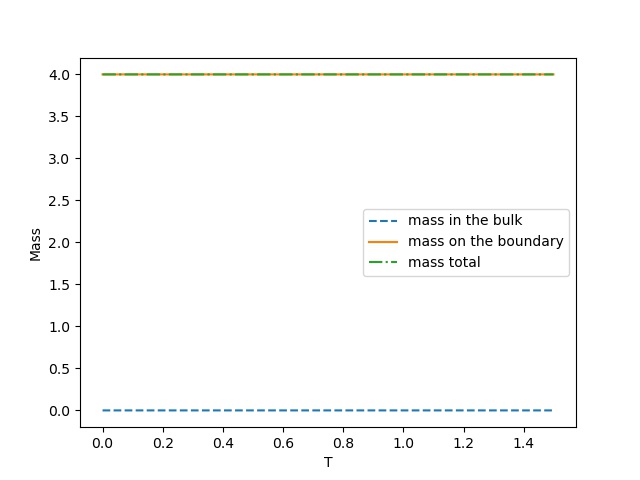}
                \caption{The mass with $\beta = 1$.}
            \end{subfigure}
            \vspace{0.5em}
            \begin{subfigure}[b]{0.49\textwidth}
                \includegraphics[width=\textwidth]{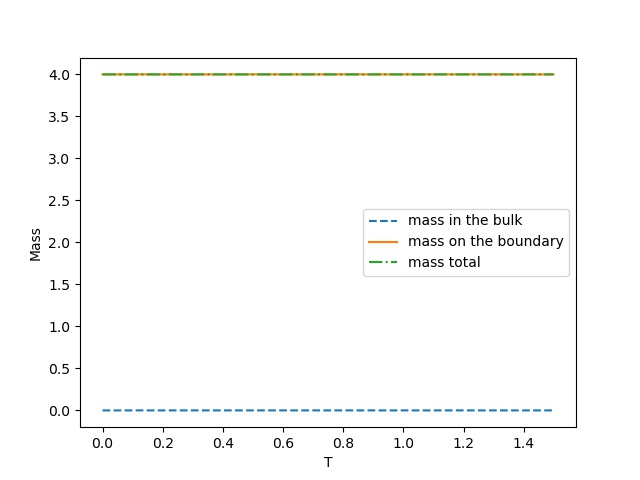}
                \caption{The mass with $\beta = 0.1$.}
            \end{subfigure}
            \begin{subfigure}[b]{0.49\textwidth}
                \includegraphics[width=\textwidth]{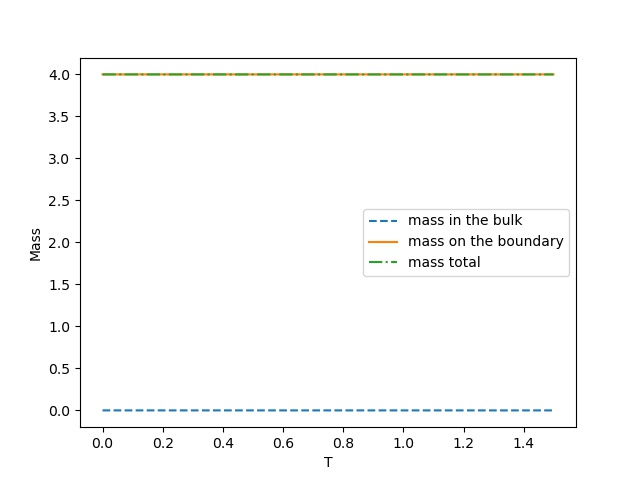}
                \caption{The mass with $\beta = 0$.}
            \end{subfigure}
        \end{minipage}
    }
    \caption{The energy evolution and the mass evolutions of Case 1.}

    \label{5.3}

\end{figure}

\begin{figure}[!htbp]
    \centering
    \begin{subfigure}[t]{0.3\textwidth}      \includegraphics[scale=0.4]{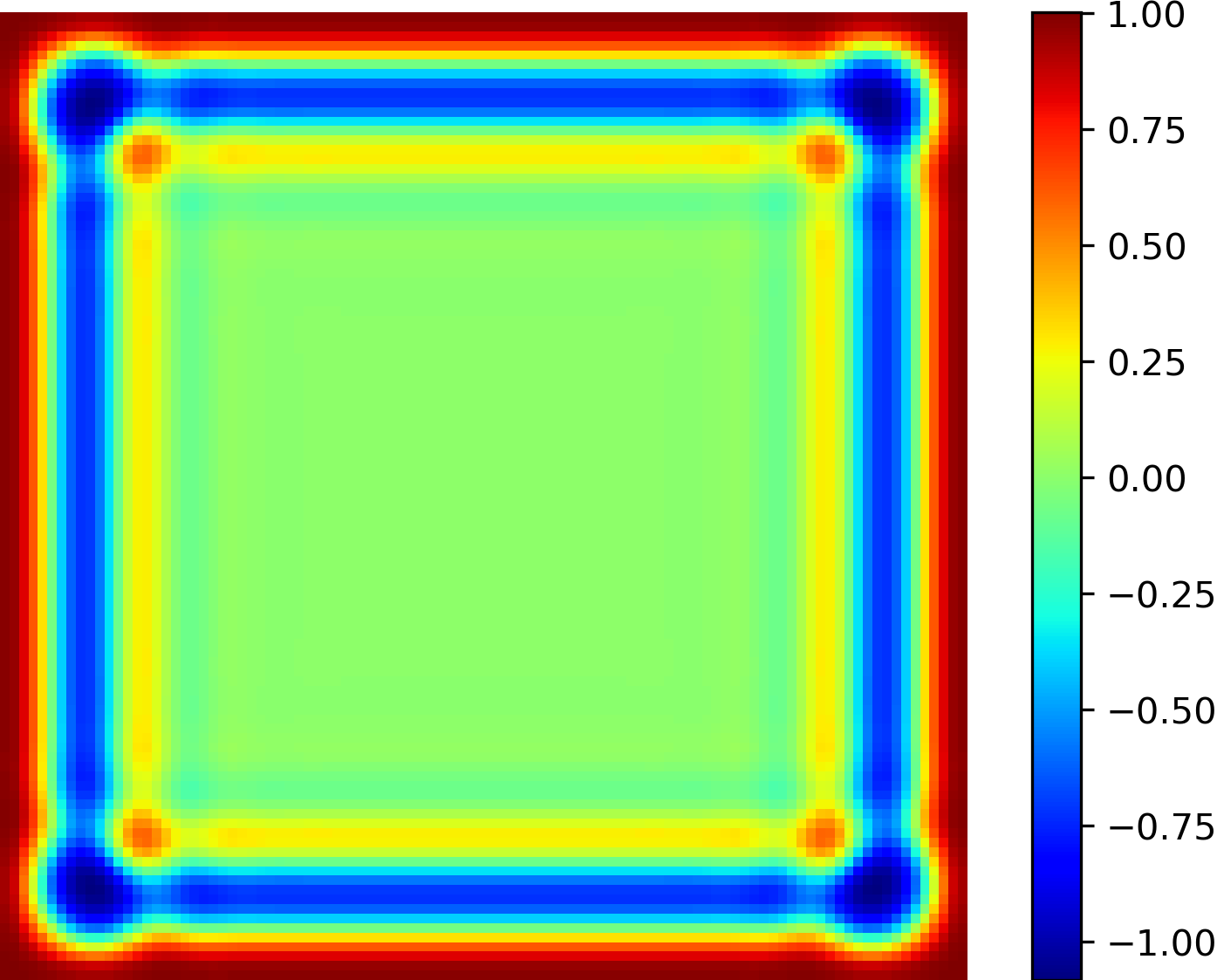}
        %\caption{t = 0.015, $\beta = 1$}
    \end{subfigure}
    \hfill
    \begin{subfigure}[t]{0.3\textwidth}  \includegraphics[scale=0.4]{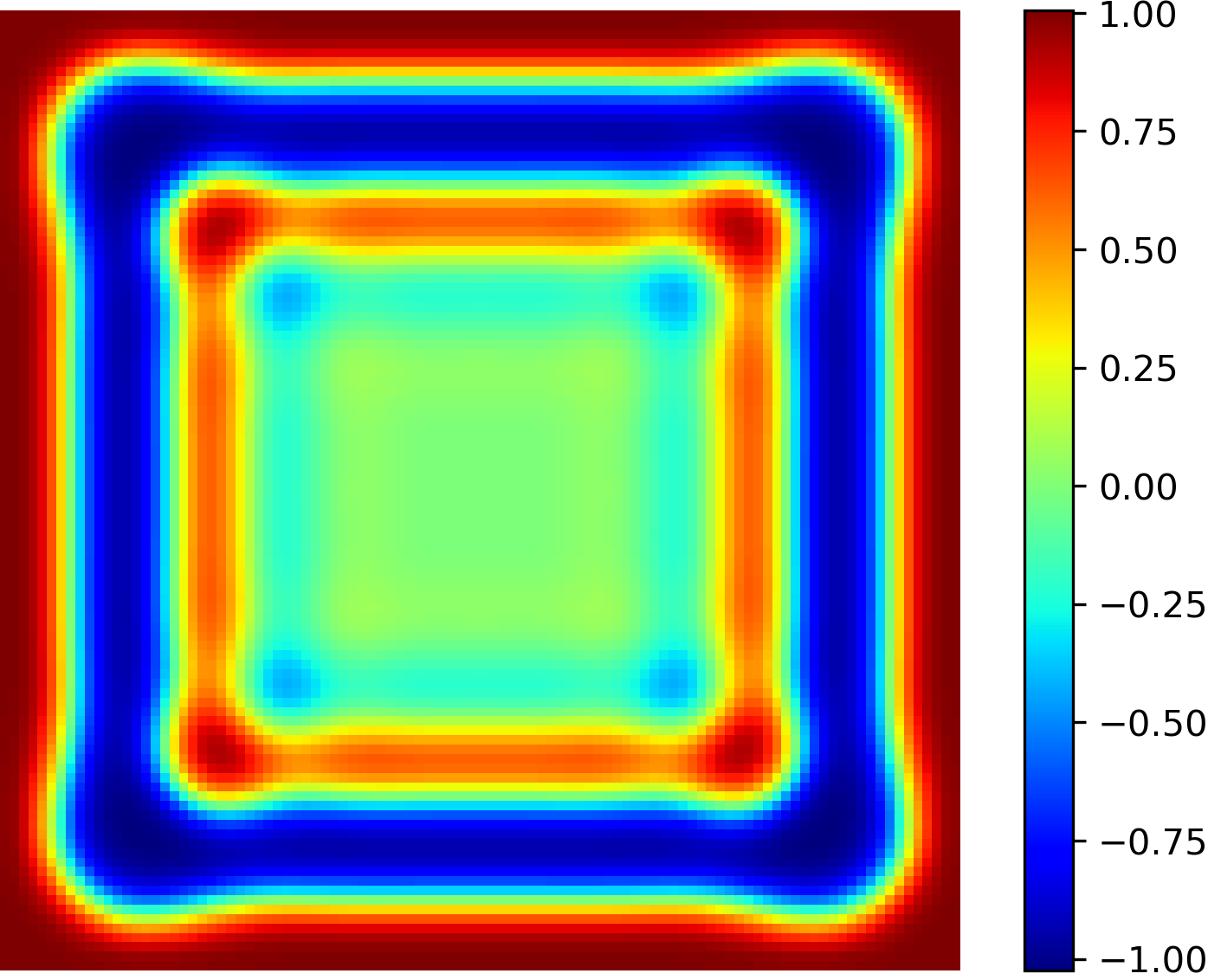}
        %\caption{t = 0.015, $\beta = 0.1$}
    \end{subfigure}
    \hfill
    \begin{subfigure}[t]{0.3\textwidth}
\includegraphics[scale=0.4]{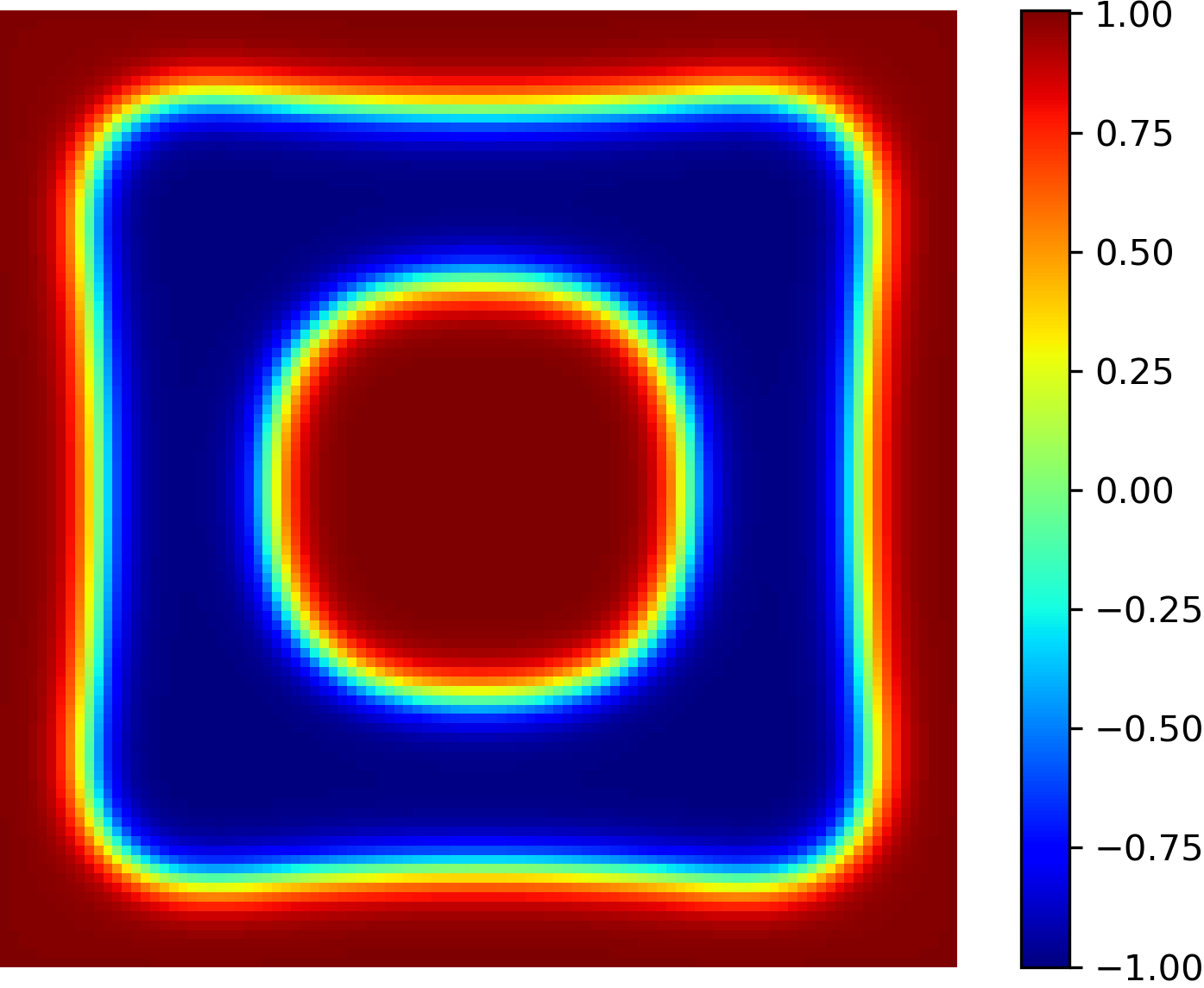}
        %\caption{t = 0.015, $\beta = 0$}
    \end{subfigure}
\begin{subfigure}[t]{0.3\textwidth}
\includegraphics[scale=0.4]{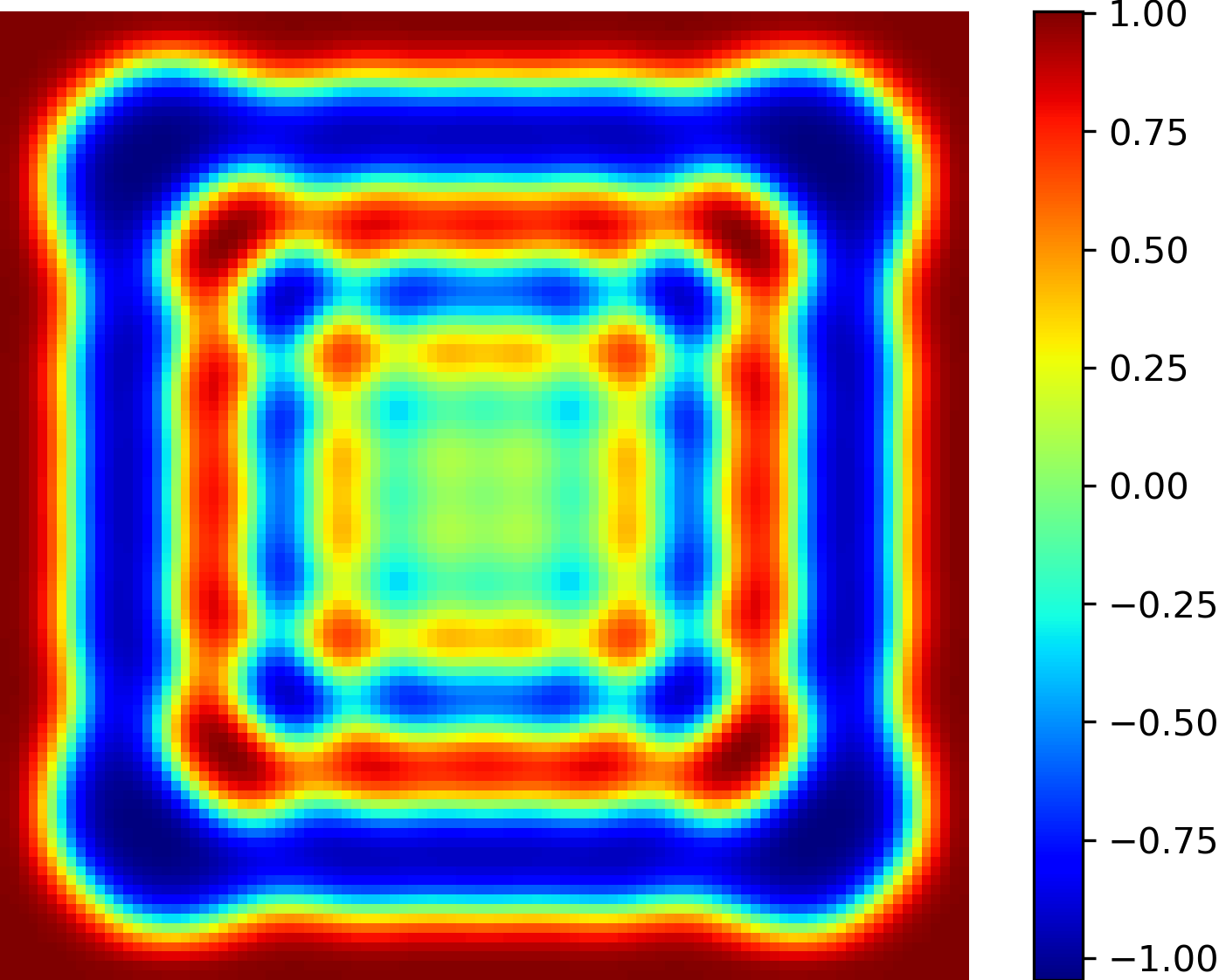}
        %\caption{t = 0.045, $\beta = 1$}
    \end{subfigure}
    \hfill
    \begin{subfigure}[t]{0.3\textwidth}
    \includegraphics[scale=0.4]{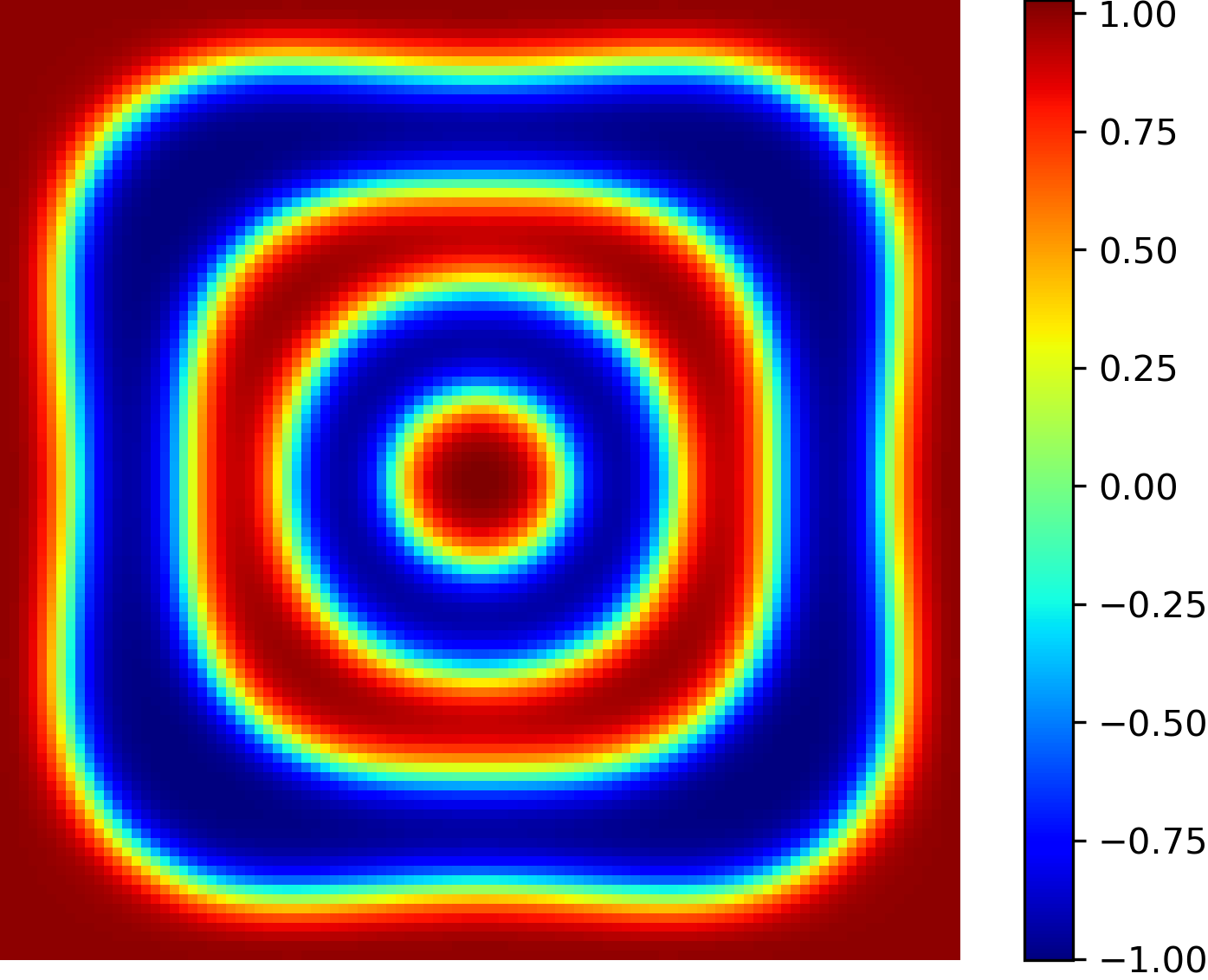}
        %\caption{t = 0.045, $\beta = 0.1$}
    \end{subfigure}
    \hfill
    \begin{subfigure}[t]{0.3\textwidth}
\includegraphics[scale=0.4]{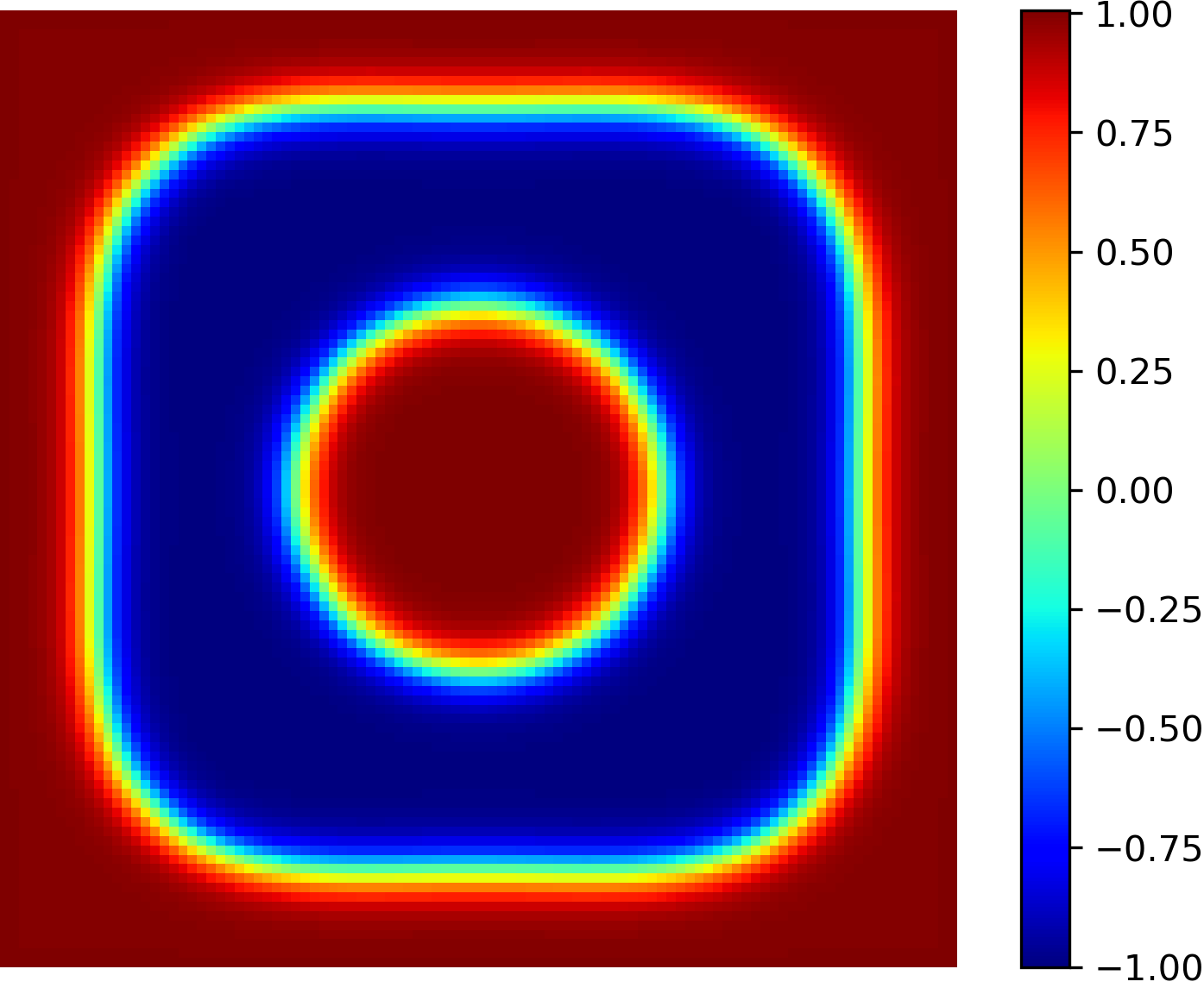}
        %\caption{t = 0.045, $\beta = 0$}
    \end{subfigure}
\begin{subfigure}[t]{0.3\textwidth}
\includegraphics[scale=0.4]{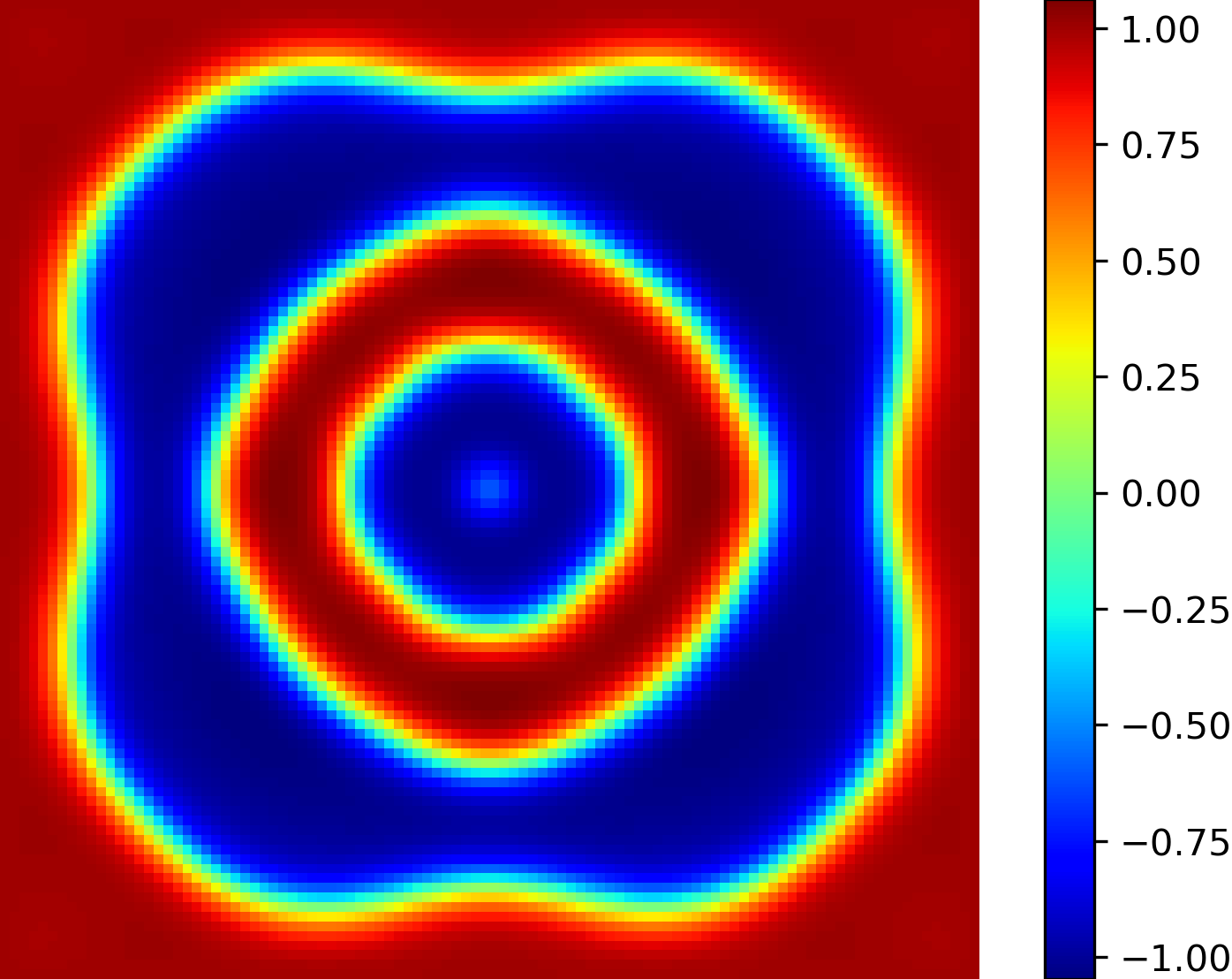}
        %\caption{t = 0.15, $\beta = 1$}
    \end{subfigure}
    \hfill
    \begin{subfigure}[t]{0.3\textwidth}
    \includegraphics[scale=0.4]{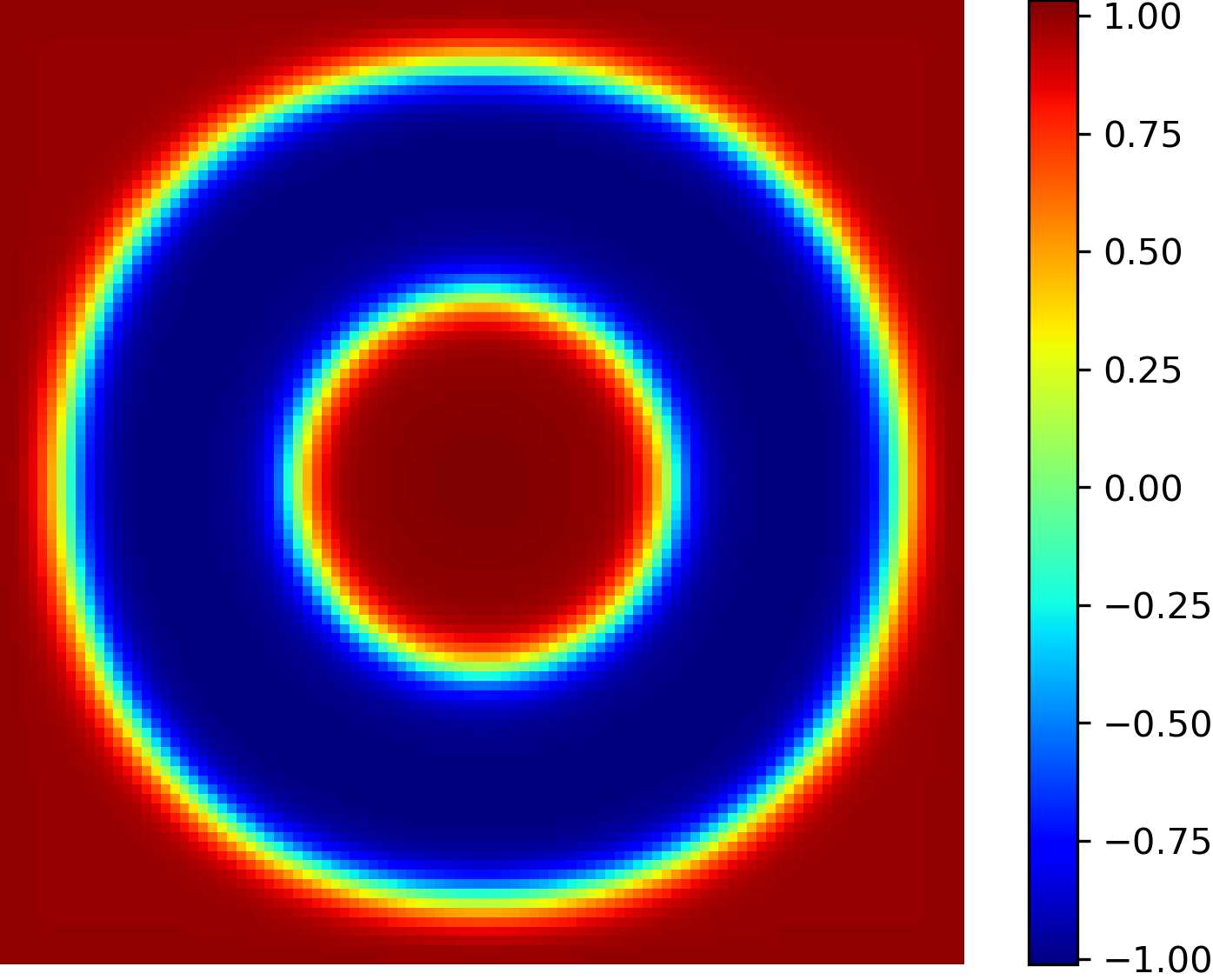}
        %\caption{t = 0.15, $\beta = 0.1$}
    \end{subfigure}
    \hfill
    \begin{subfigure}[t]{0.3\textwidth}
\includegraphics[scale=0.4]{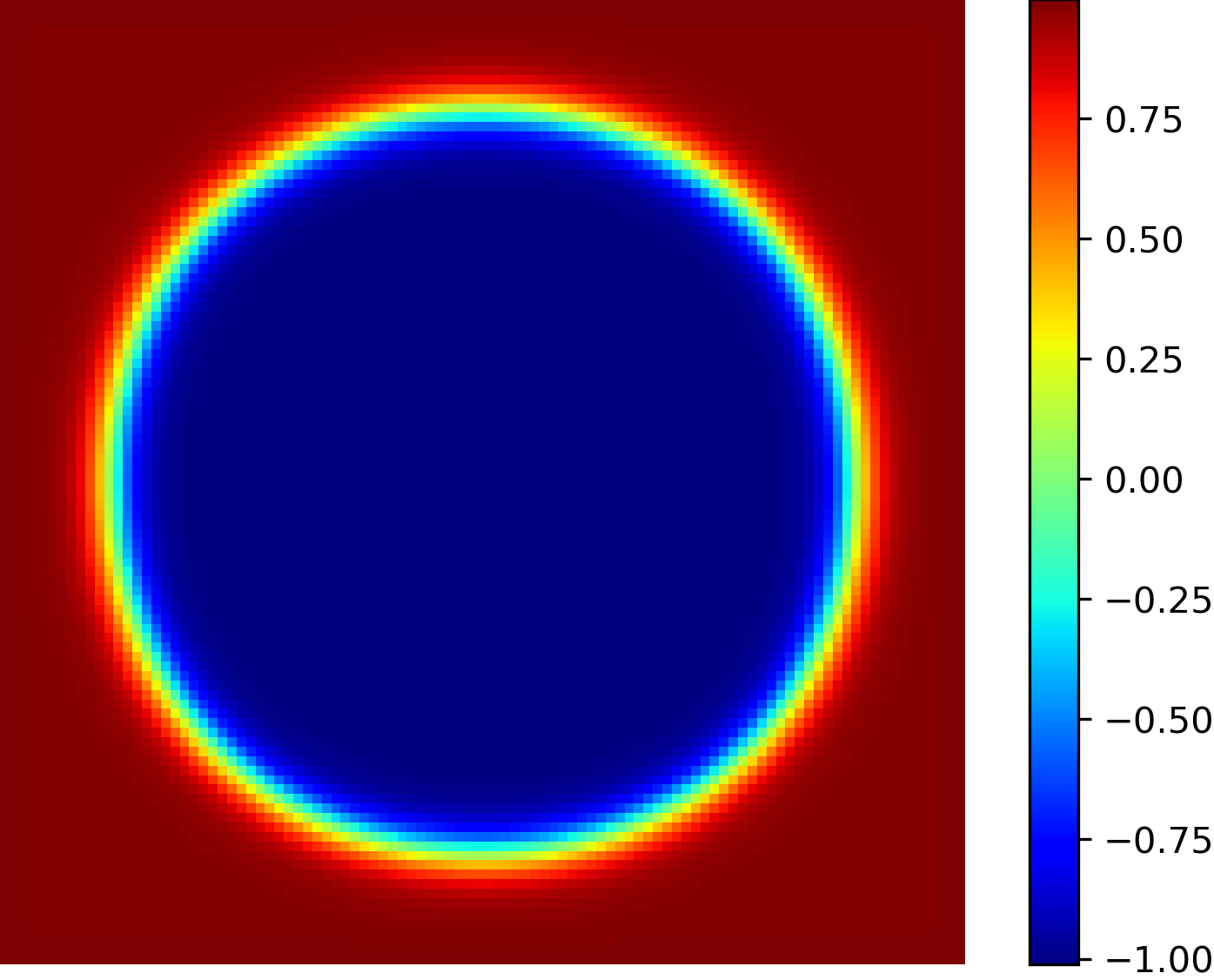}
        %\caption{t = 0.15, $\beta = 0$}
    \end{subfigure}
\begin{subfigure}[t]{0.3\textwidth}
\includegraphics[scale=0.4]{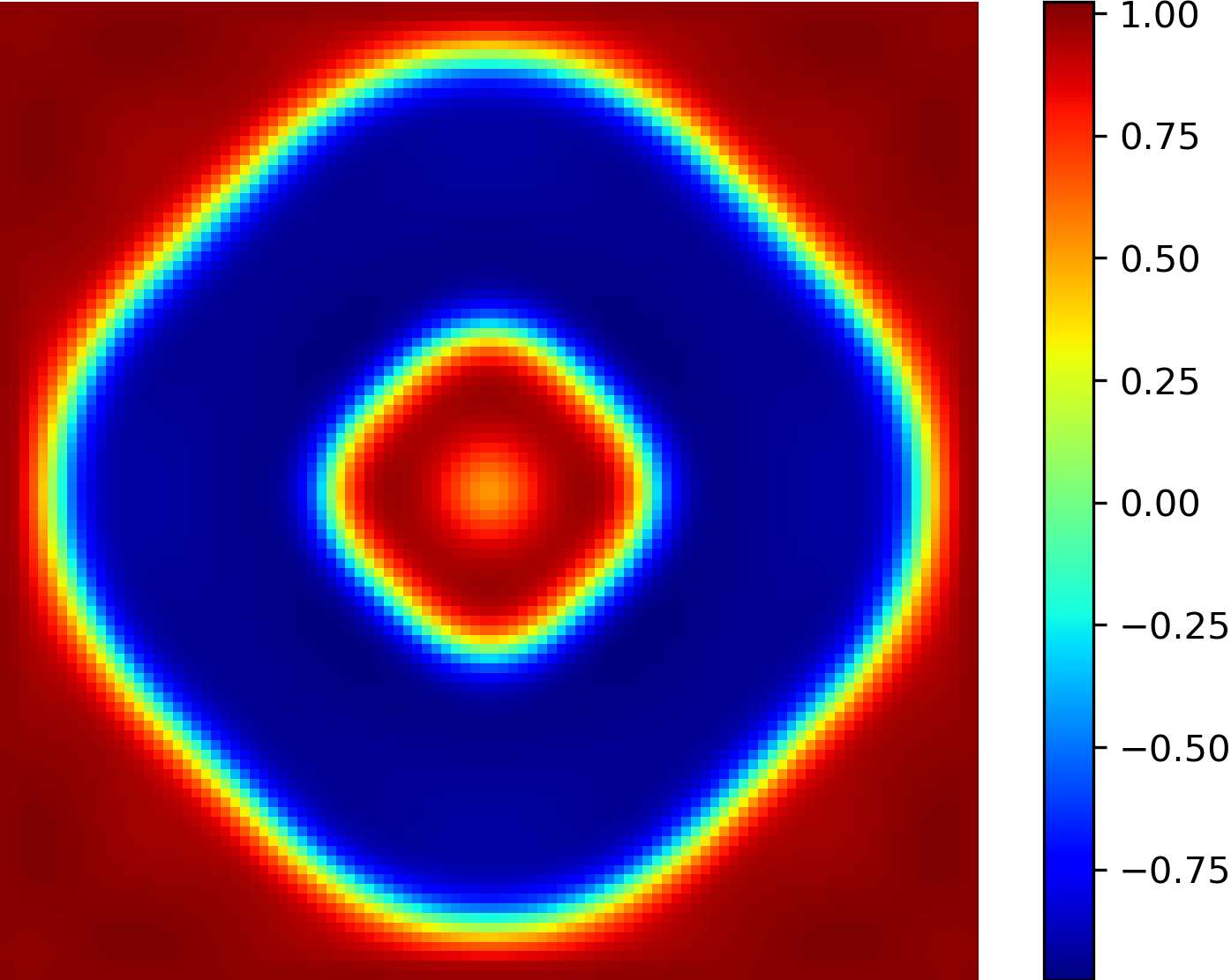}
        %\caption{t = 0.3, $\beta = 1$}
    \end{subfigure}
    \hfill
    \begin{subfigure}[t]{0.3\textwidth}
\includegraphics[scale=0.4]{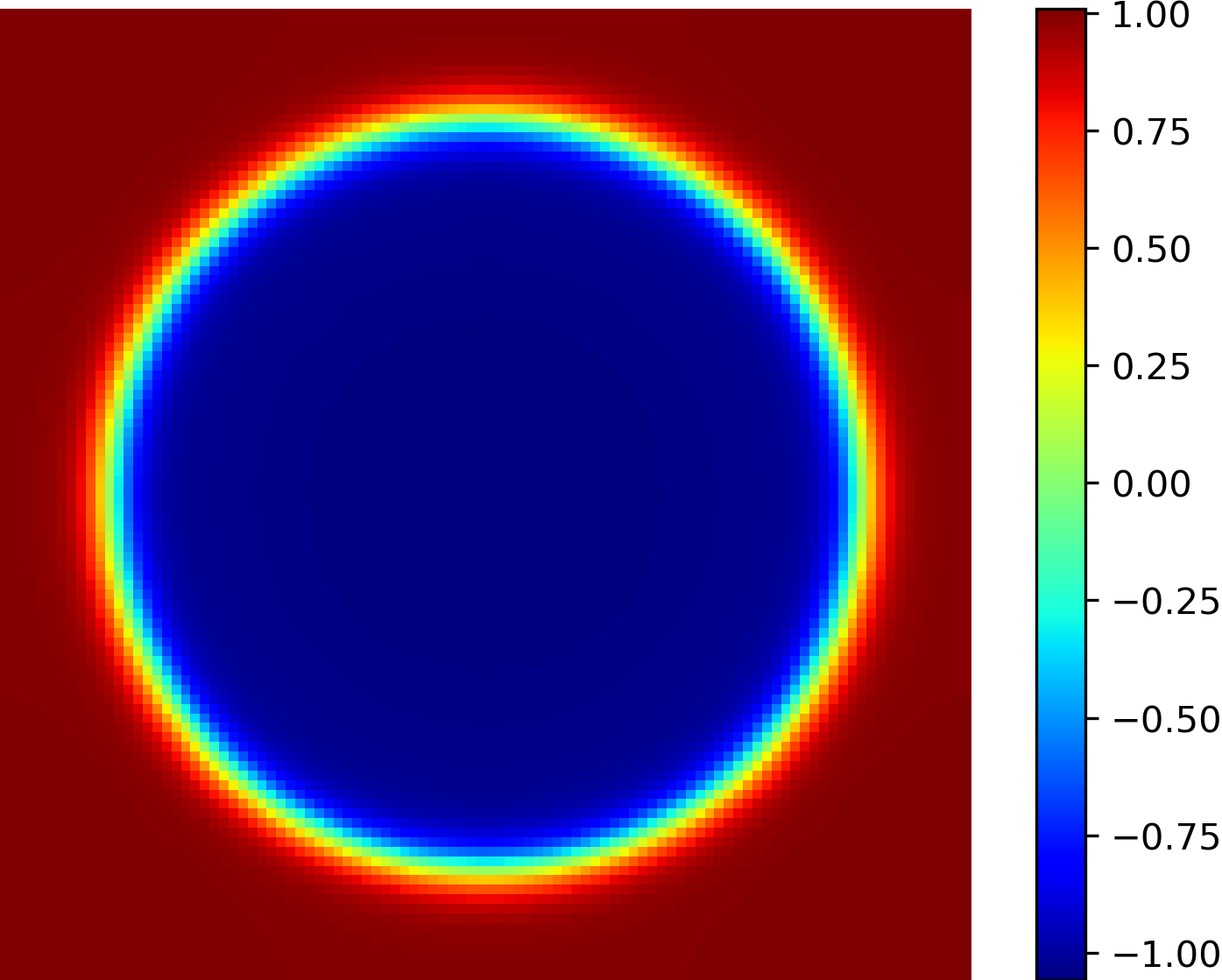}
        %\caption{t = 0.3, $\beta = 0.1$}
    \end{subfigure}
    \hfill
    \begin{subfigure}[t]{0.3\textwidth}
\includegraphics[scale=0.4]{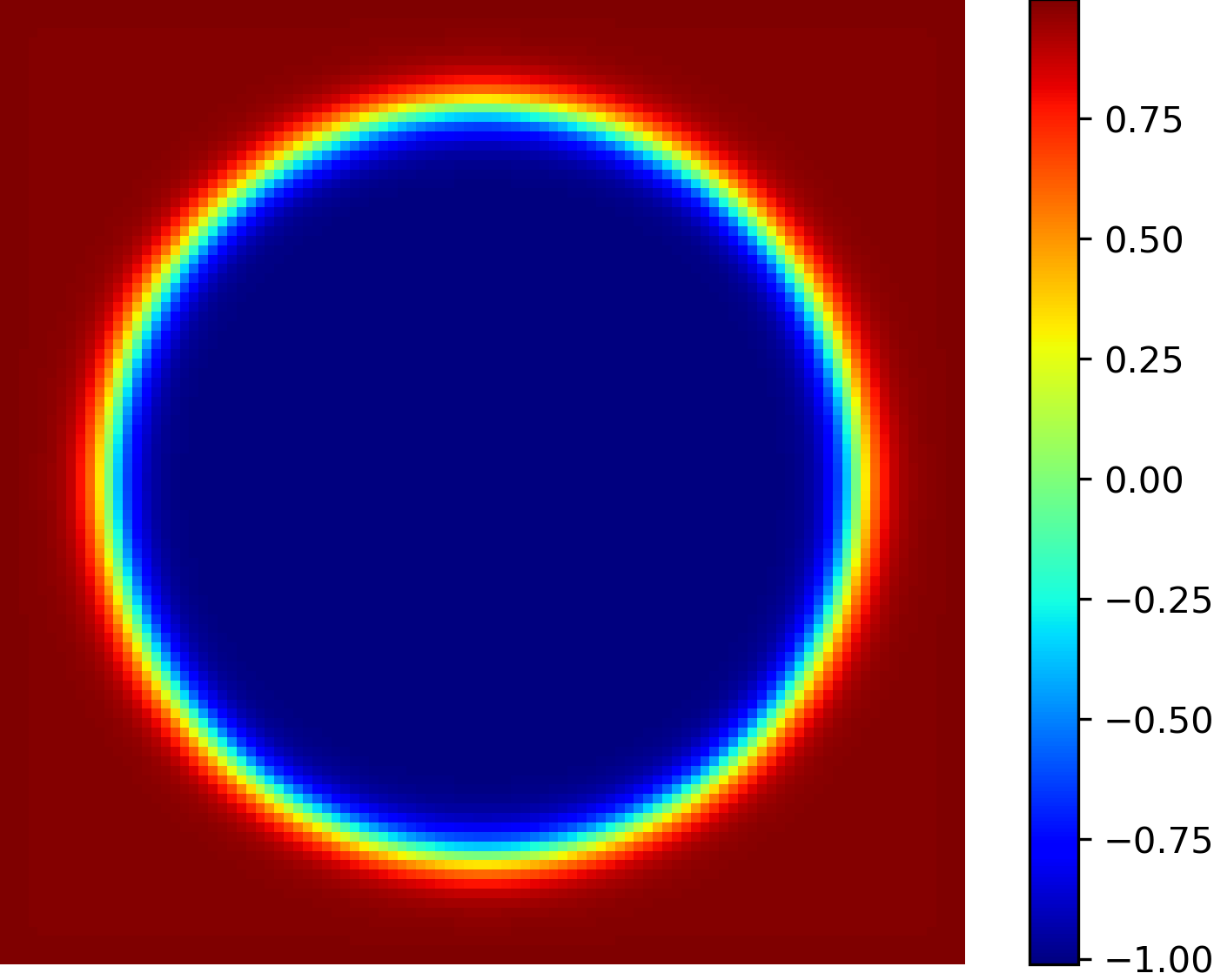}
        %\caption{t = 0.3, $\beta = 0$}
    \end{subfigure}
\begin{subfigure}[t]{0.3\textwidth}
\includegraphics[scale=0.4]{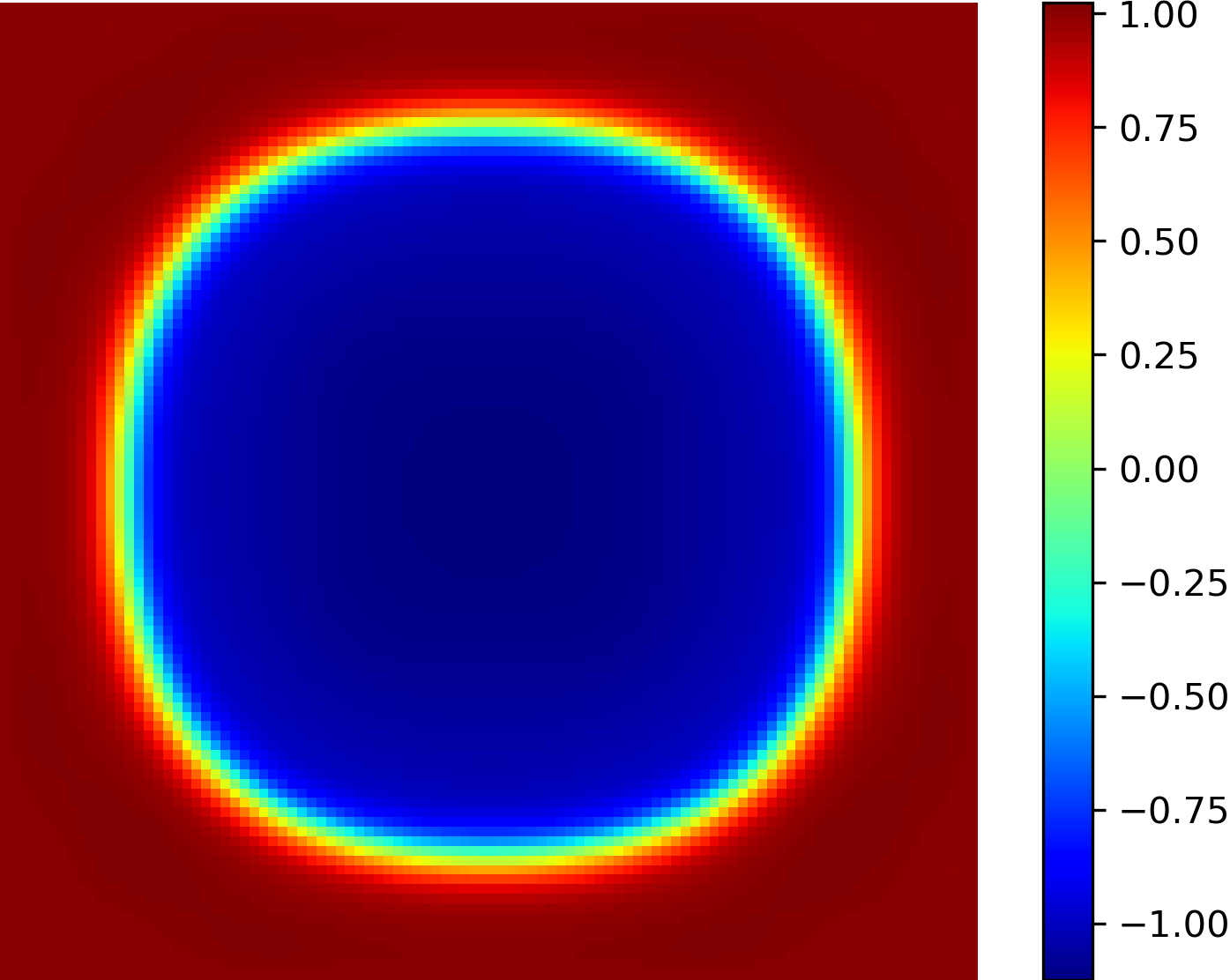}
       % \caption{t = 0.6, $\beta = 1$}
    \end{subfigure}
    \hfill
    \begin{subfigure}[t]{0.3\textwidth}
\includegraphics[scale=0.4]{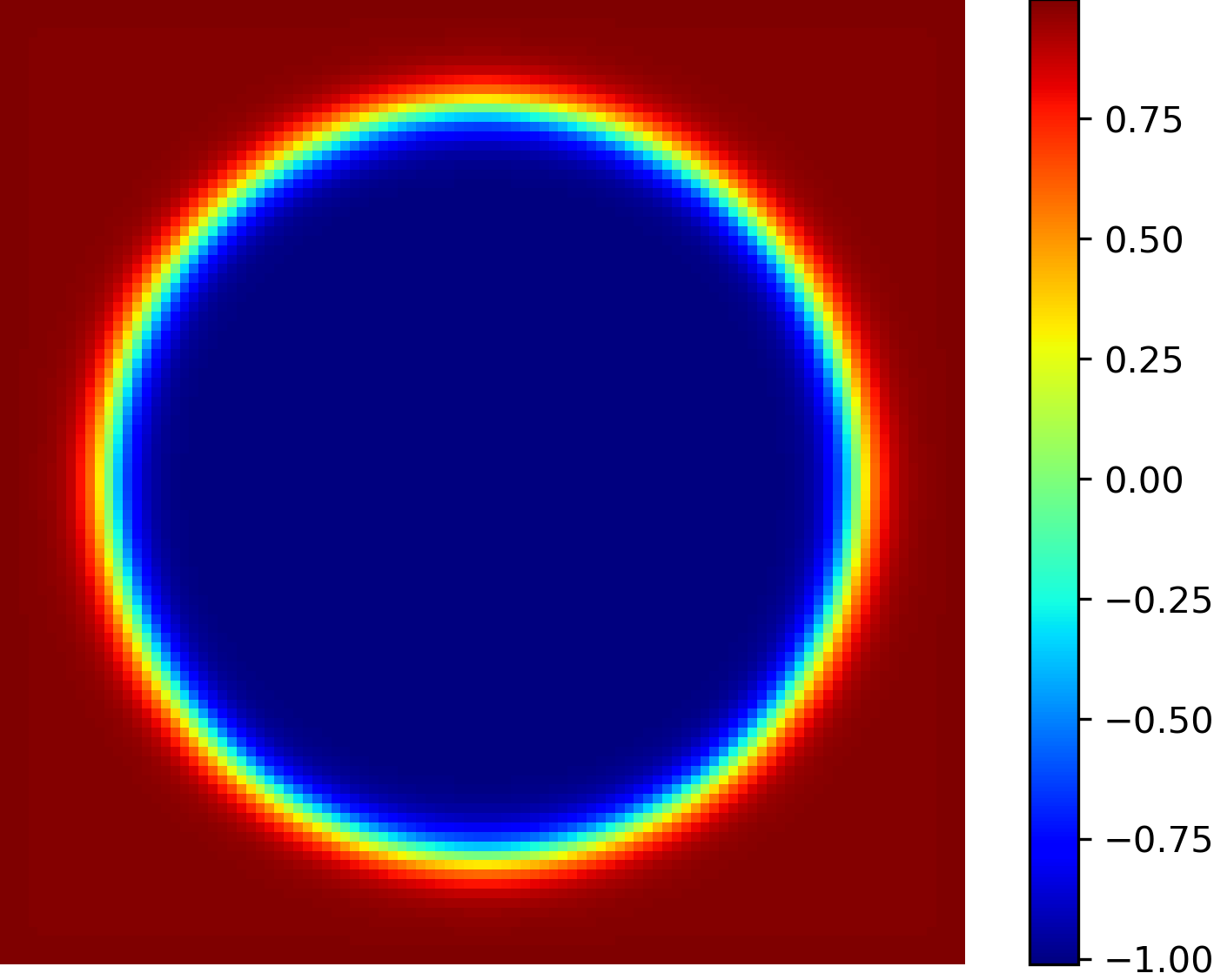}
      %  \caption{t = 0.6, $\beta = 0.1$}
    \end{subfigure}
    \hfill
    \begin{subfigure}[t]{0.3\textwidth}
\includegraphics[scale=0.4]{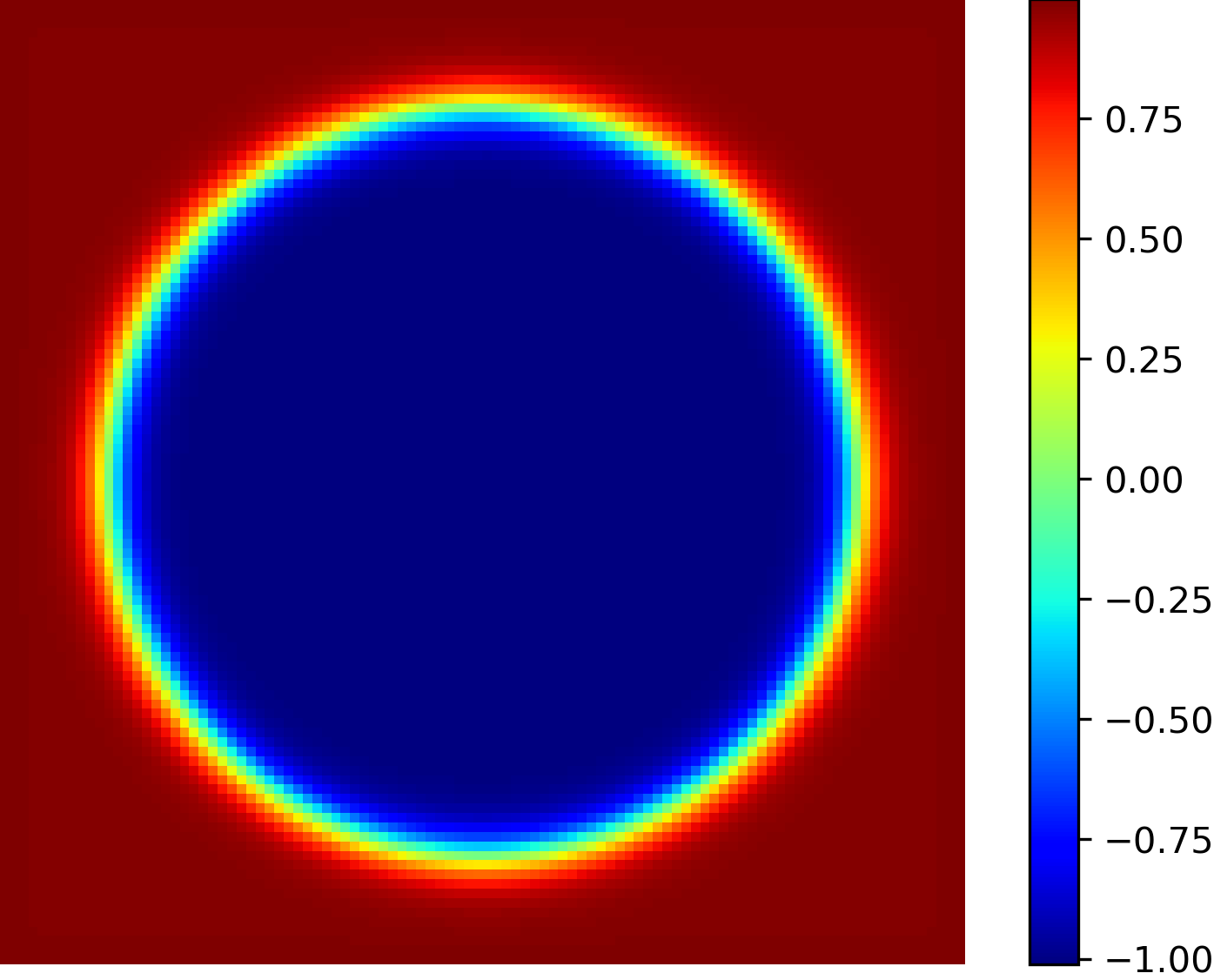}
       % \caption{t = 0.6, $\beta = 0$}
    \end{subfigure}
\caption{
    Case 1: Snapshots of the numerical approximation are taken at $T= 0.015$, $0.045$, $0.15$, $0.3$, and $0.6$ with different $\beta$. Left: $\beta=1$; Middle: $\beta=0.1$; Right: $\beta=0$.}
    \label{5.4}
\end{figure}

\textbf{Case 2:} In Figure \ref{5.5}, we set the random value as,
\begin{equation}
    \phi_0(x,y) = \left\{
        \begin{array}{ll}
        \text{rand}[-0.1,0.1], & (x,y)\text{ in\, } \Omega , \\[8pt]
        \text{rand}[0.4,0.6], &(x,y)\text{ on\, } \Gamma.    
        \end{array}
        \right.
\end{equation}
By varying the values of 
$\beta$, we obtain the numerical solutions for 
$\phi$, as illustrated in Figure \ref{5.7}, while the corresponding energy and mass evolutions are presented in Figure \ref{5.6}.
Similarly, the Figure \ref{5.7} illustrates that the phase of the system is coarsening more slowly when $\beta$ becomes lager. The Figure \ref{5.6} indicates that the discrete energy of the system is decreasing slowly by enlarging the value of $\beta$. We also observe that the mass is conserved  in  the bulk and on the boundary respectively.
\begin{figure}[!htbp]
    \centering
    \includegraphics[width=0.4\textwidth]{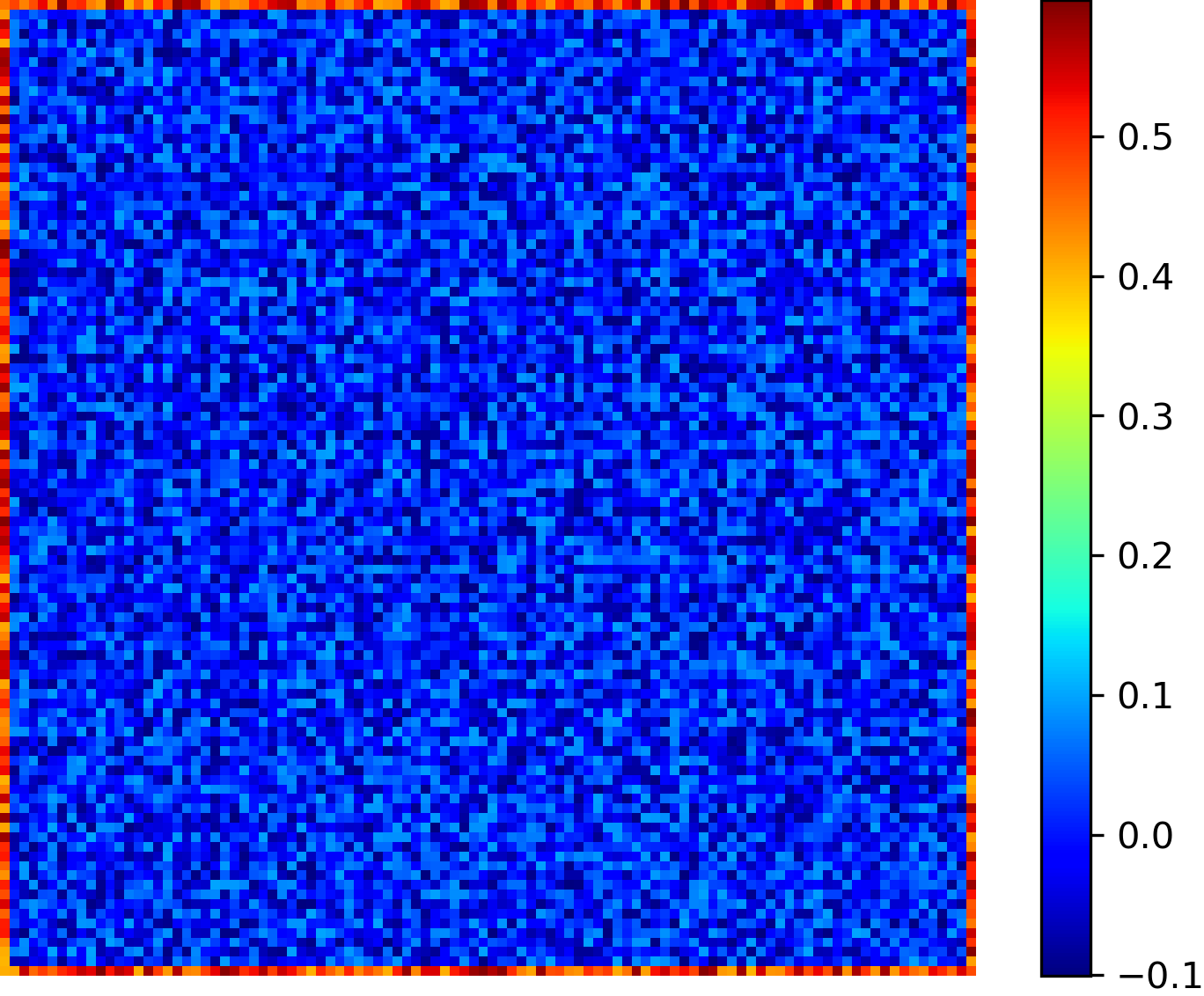}
    \caption{The initial data of Case 2.}

    \label{5.5}

\end{figure}

\begin{figure}[!htbp]
    \centering
    \scalebox{0.85}{
        \begin{minipage}{\textwidth}
            \begin{subfigure}[b]{0.49\textwidth}
                \includegraphics[width=\textwidth]{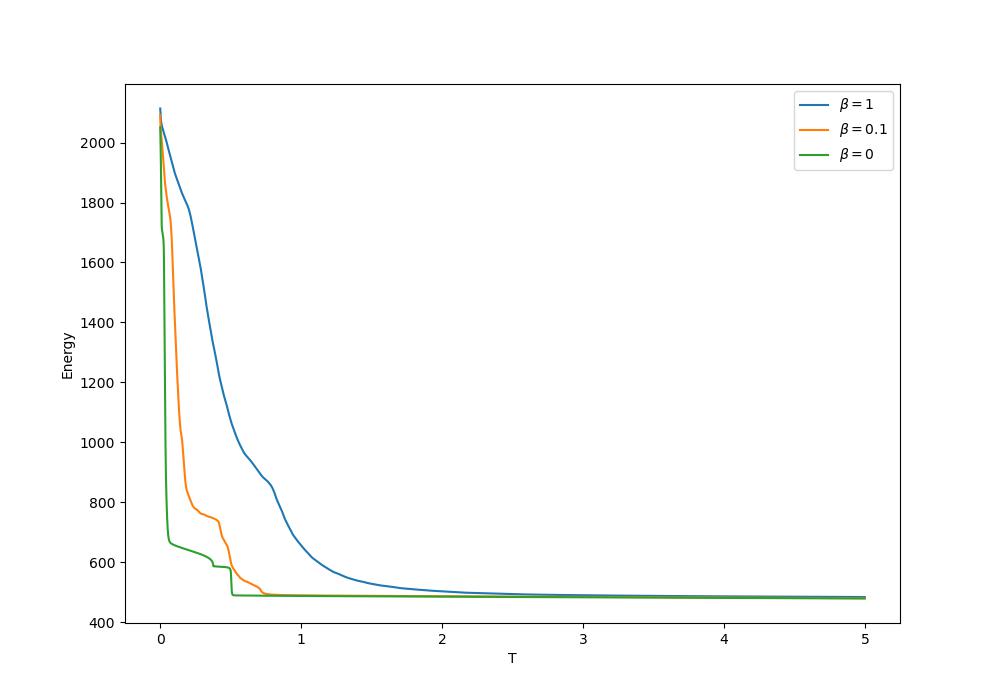}
                \caption{Energy curves with different $\beta$.}
            \end{subfigure}
            \begin{subfigure}[b]{0.49\textwidth}
                \includegraphics[width=\textwidth]{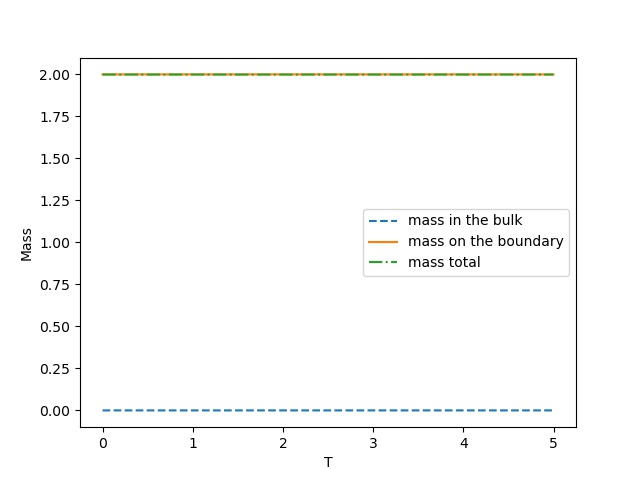}
                \caption{The mass with $\beta = 1$.}
            \end{subfigure}
            \begin{subfigure}[b]{0.49\textwidth}
                \includegraphics[width=\textwidth]{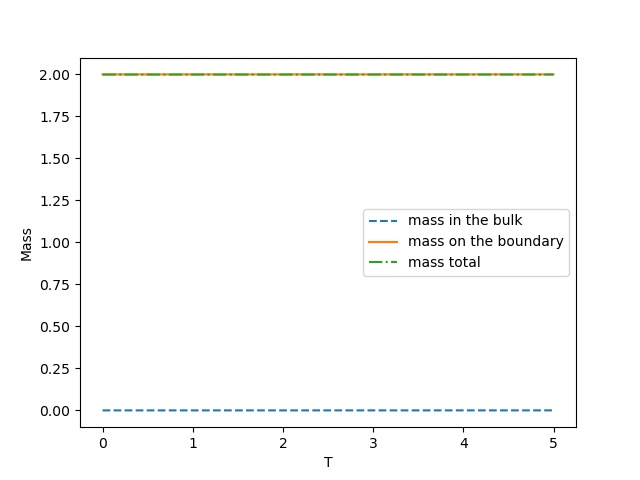}
                \caption{The mass with $\beta = 0.1$.}
            \end{subfigure}
            \begin{subfigure}[b]{0.49\textwidth}
                \includegraphics[width=\textwidth]{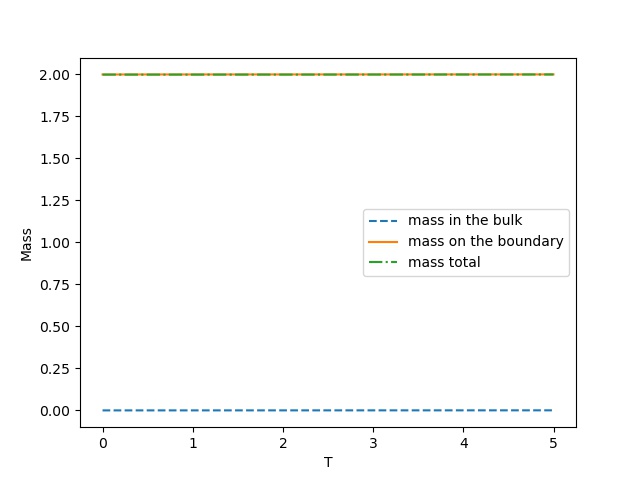}
                \caption{The mass with $\beta = 0$.}
            \end{subfigure}
        \end{minipage}
    }
    \caption{The energy evolution and the mass evolutions of Case 2.}

    \label{5.6}

\end{figure}

\begin{figure}[!htbp]
    \centering
    \begin{subfigure}[t]{0.3\textwidth}
        \includegraphics[scale=0.4]{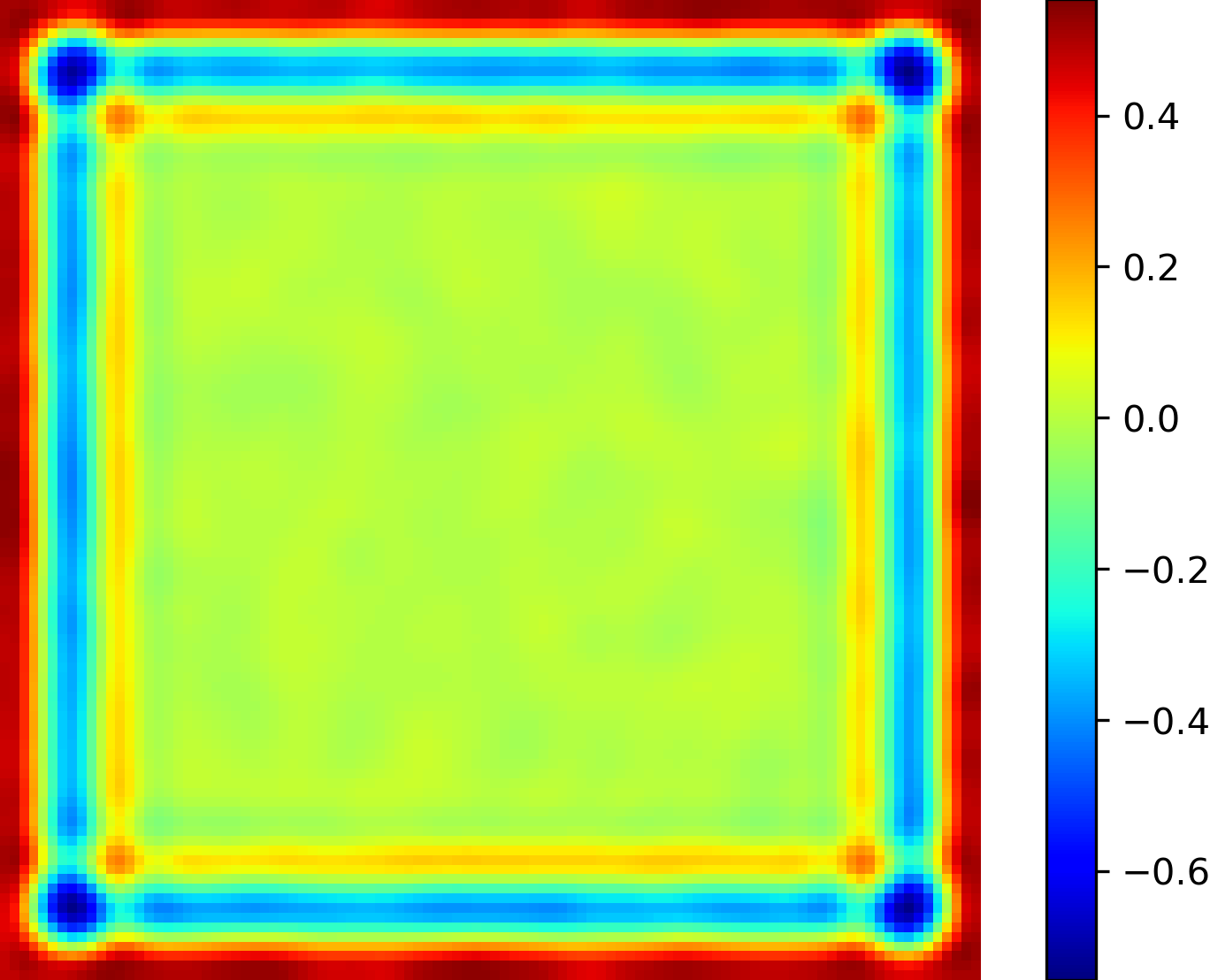}
        %\caption{t = 0.015, $\beta = 1$}
    \end{subfigure}
    \hfill
    \begin{subfigure}[t]{0.3\textwidth}
        \includegraphics[scale=0.4]{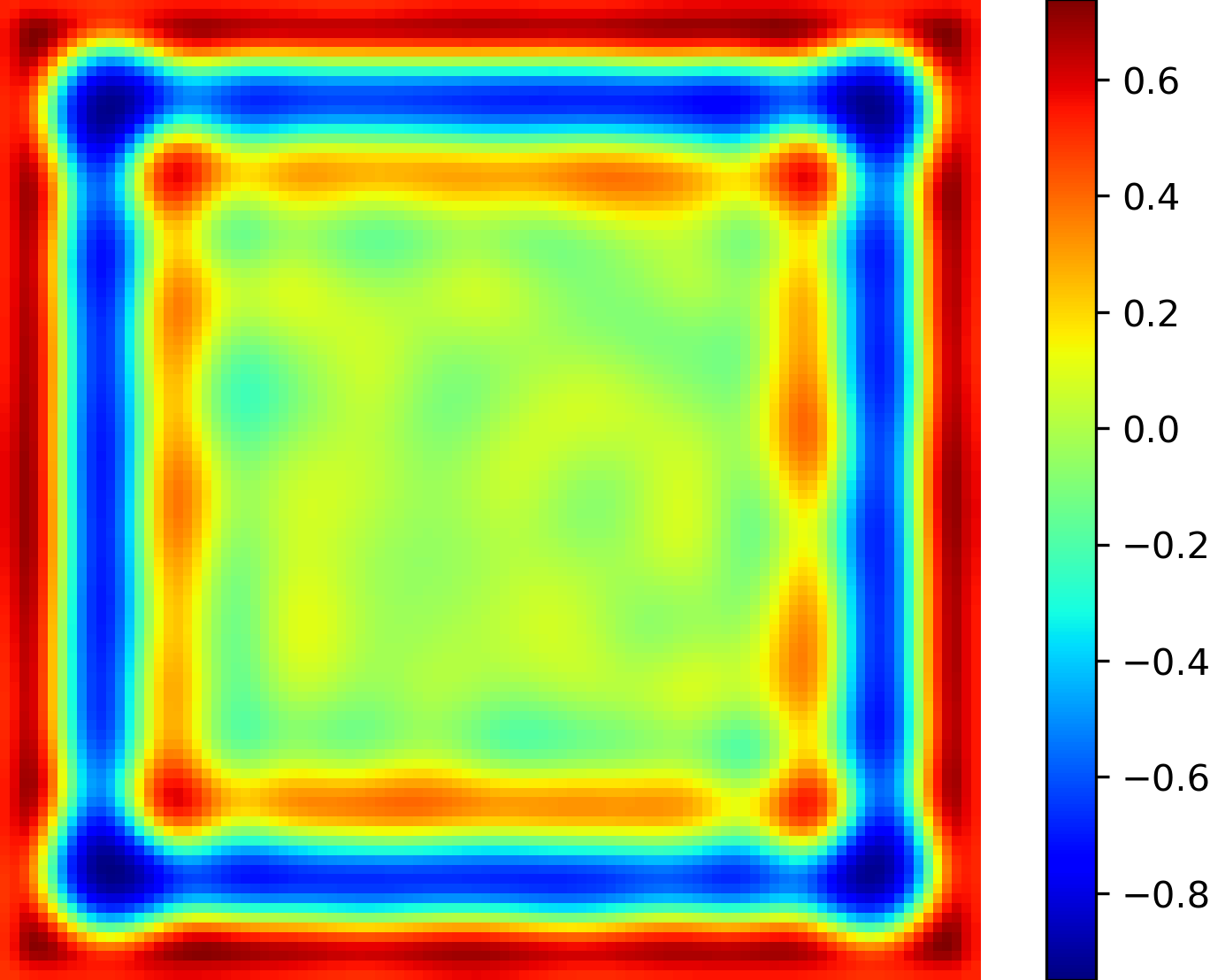}
        %\caption{t = 0.015, $\beta = 0.1$}
    \end{subfigure}
    \hfill
    \begin{subfigure}[t]{0.3\textwidth}
        \includegraphics[scale=0.4]{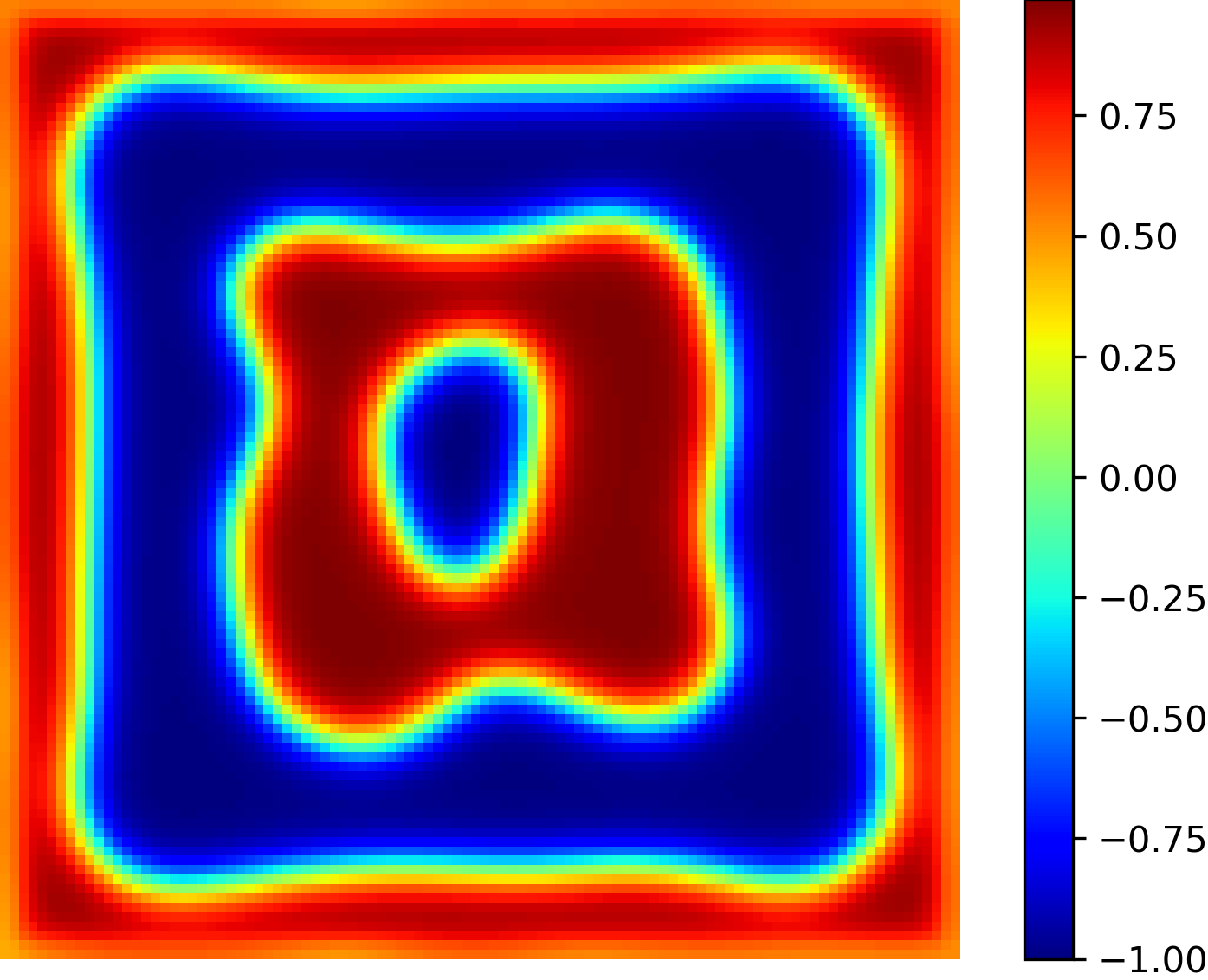}
        %\caption{t = 0.015, $\beta = 0$}
    \end{subfigure}

    \begin{subfigure}[t]{0.3\textwidth}
        \includegraphics[scale=0.4]{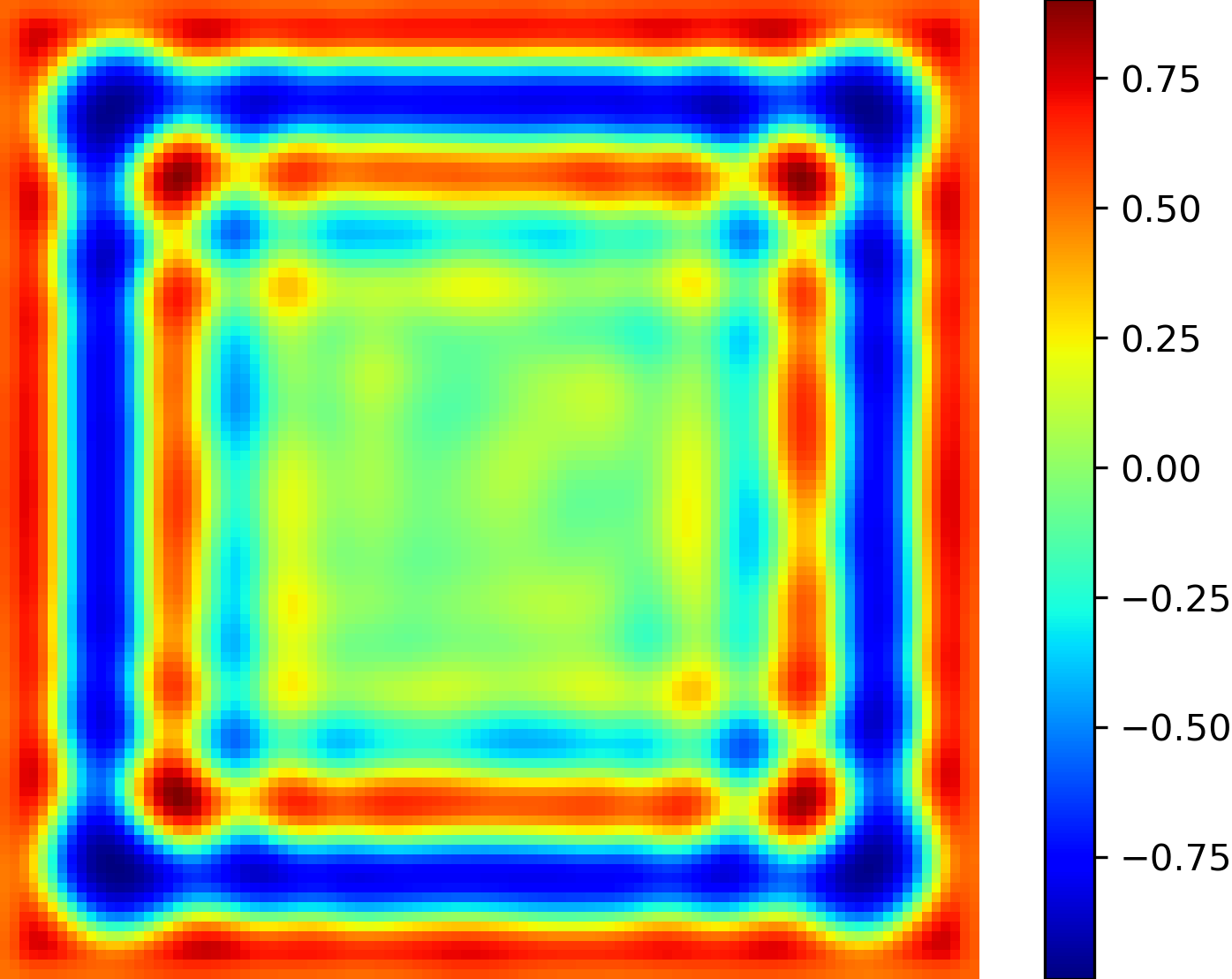}
        %\caption{t = 0.045, $\beta = 1$}
    \end{subfigure}
    \hfill
    \begin{subfigure}[t]{0.3\textwidth}
        \includegraphics[scale=0.4]{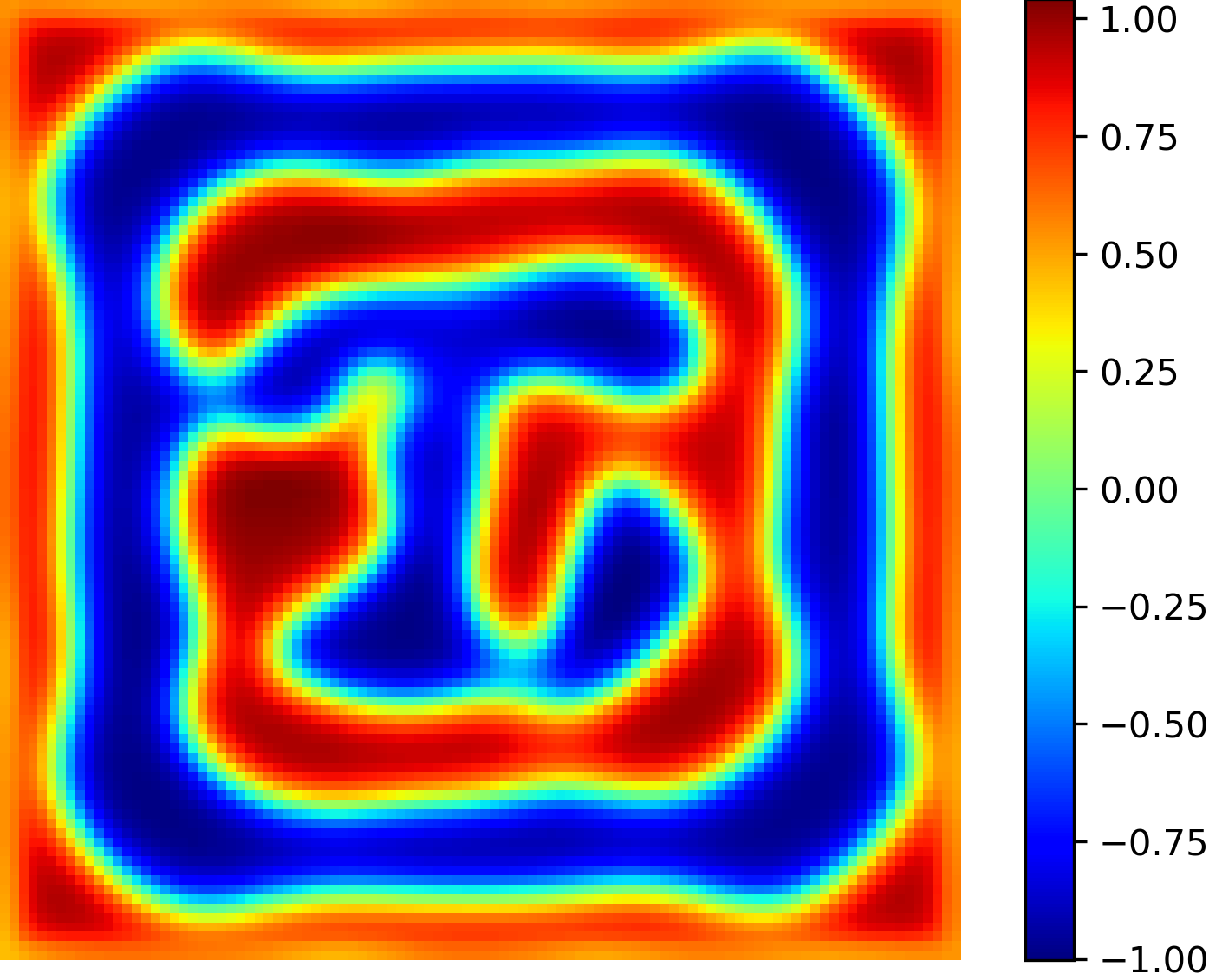}
        %\caption{t = 0.045, $\beta = 0.1$}
    \end{subfigure}
    \hfill
    \begin{subfigure}[t]{0.3\textwidth}
        \includegraphics[scale=0.4]{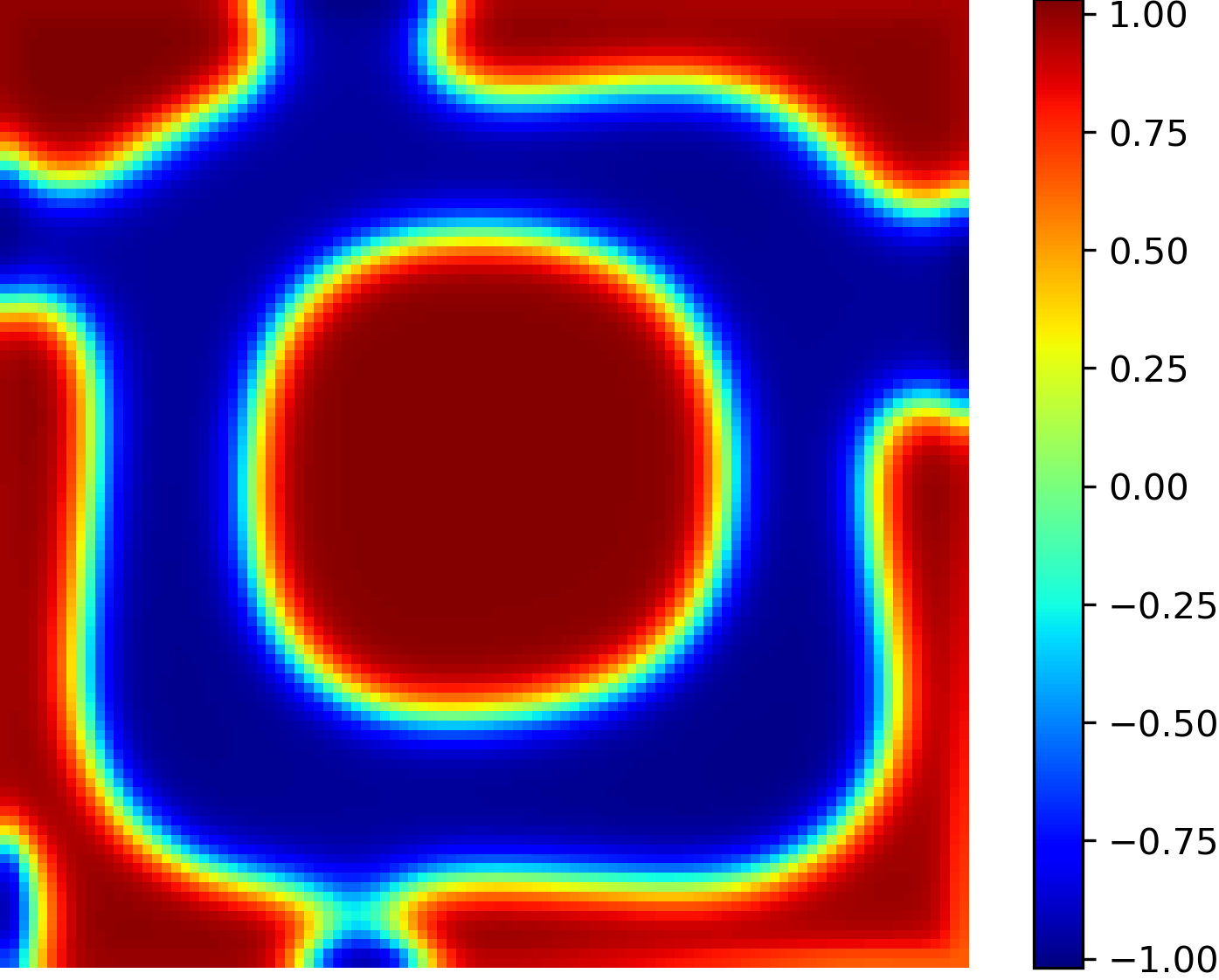}
        %\caption{t = 0.045, $\beta = 0$}
    \end{subfigure}

    \begin{subfigure}[t]{0.3\textwidth}
        \includegraphics[scale=0.4]{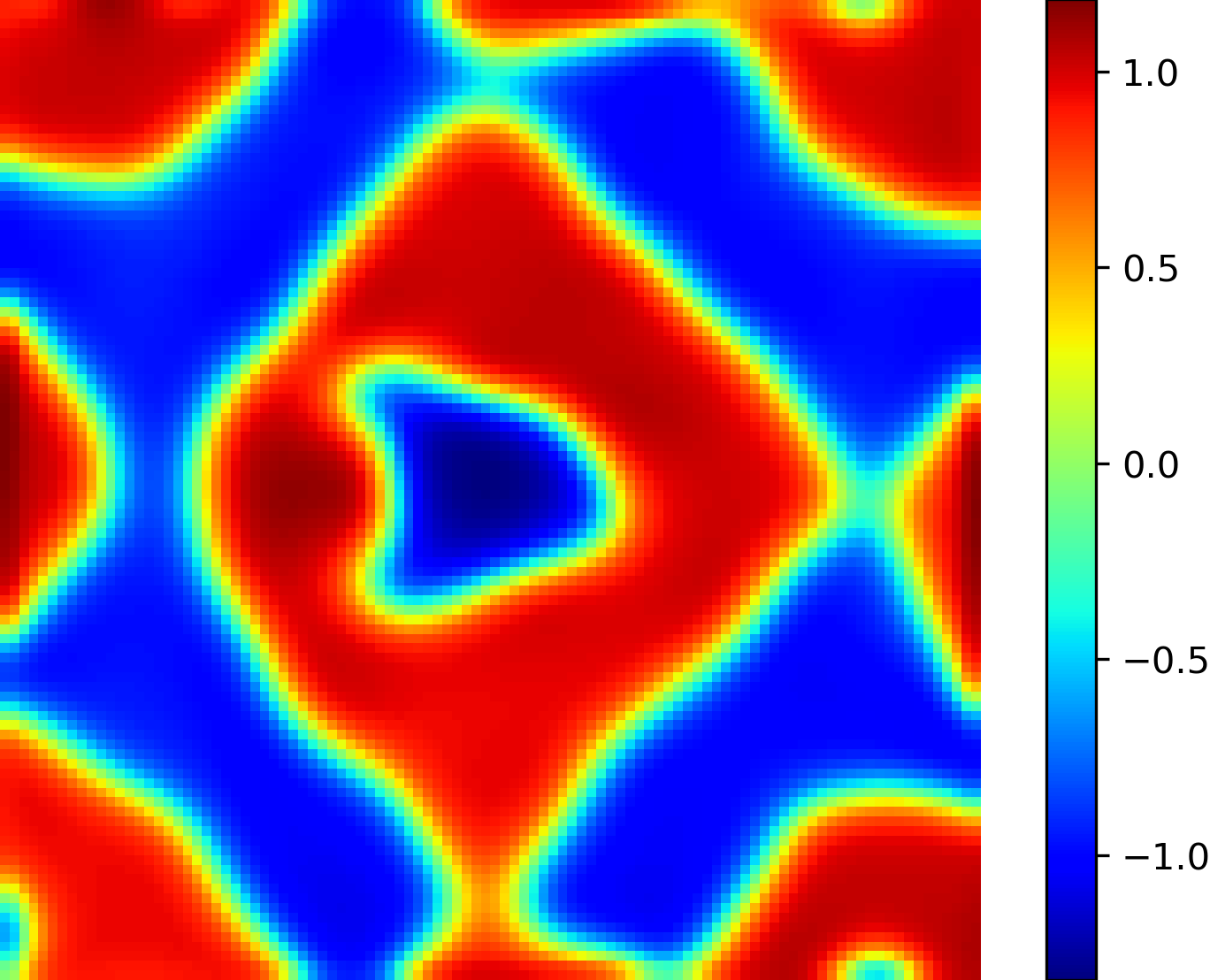}
        %\caption{t = 0.3, $\beta = 1$}
    \end{subfigure}
    \hfill
    \begin{subfigure}[t]{0.3\textwidth}
        \includegraphics[scale=0.4]{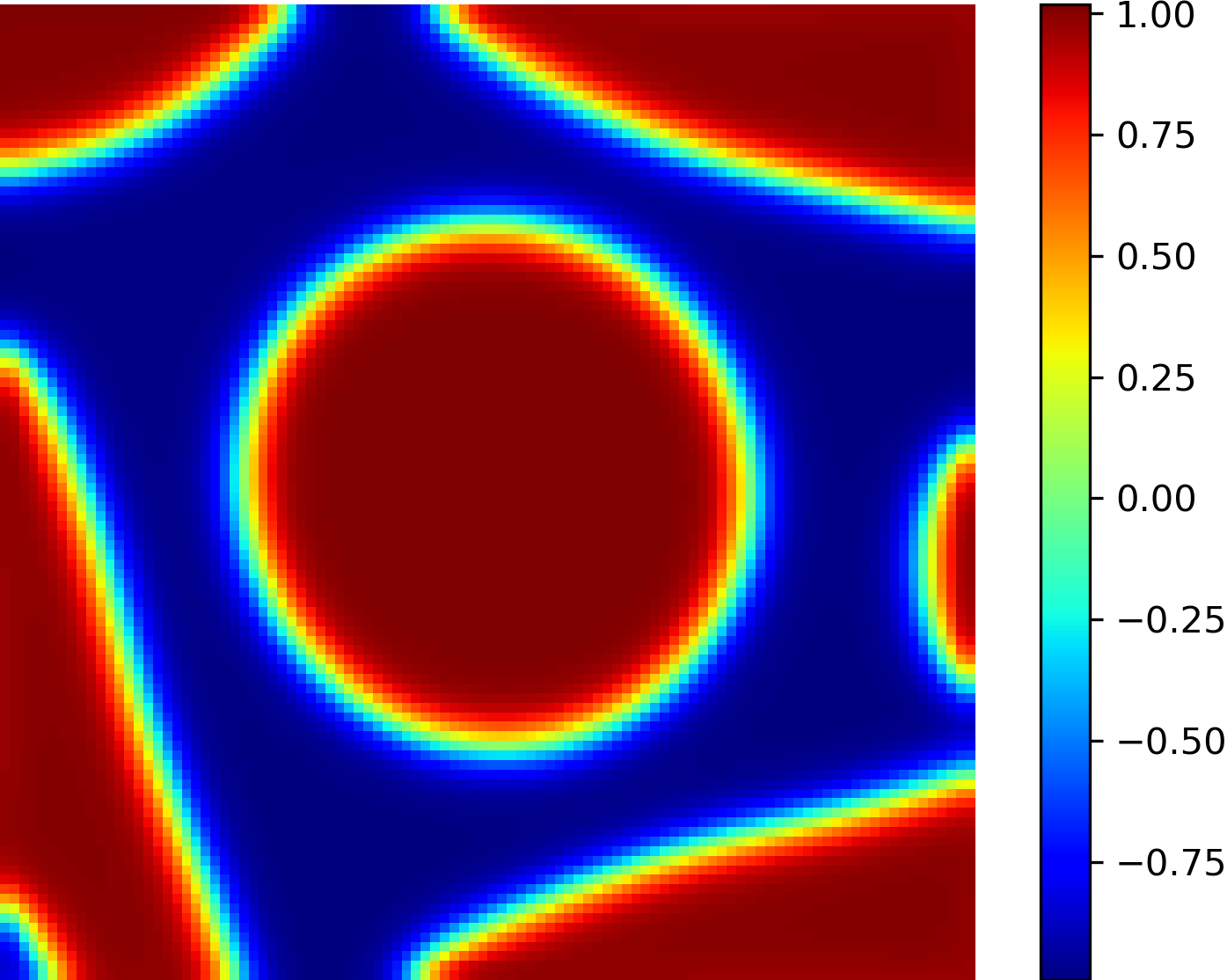}
        %\caption{t = 0.3, $\beta = 0.1$}
    \end{subfigure}
    \hfill
    \begin{subfigure}[t]{0.3\textwidth}
        \includegraphics[scale=0.4]{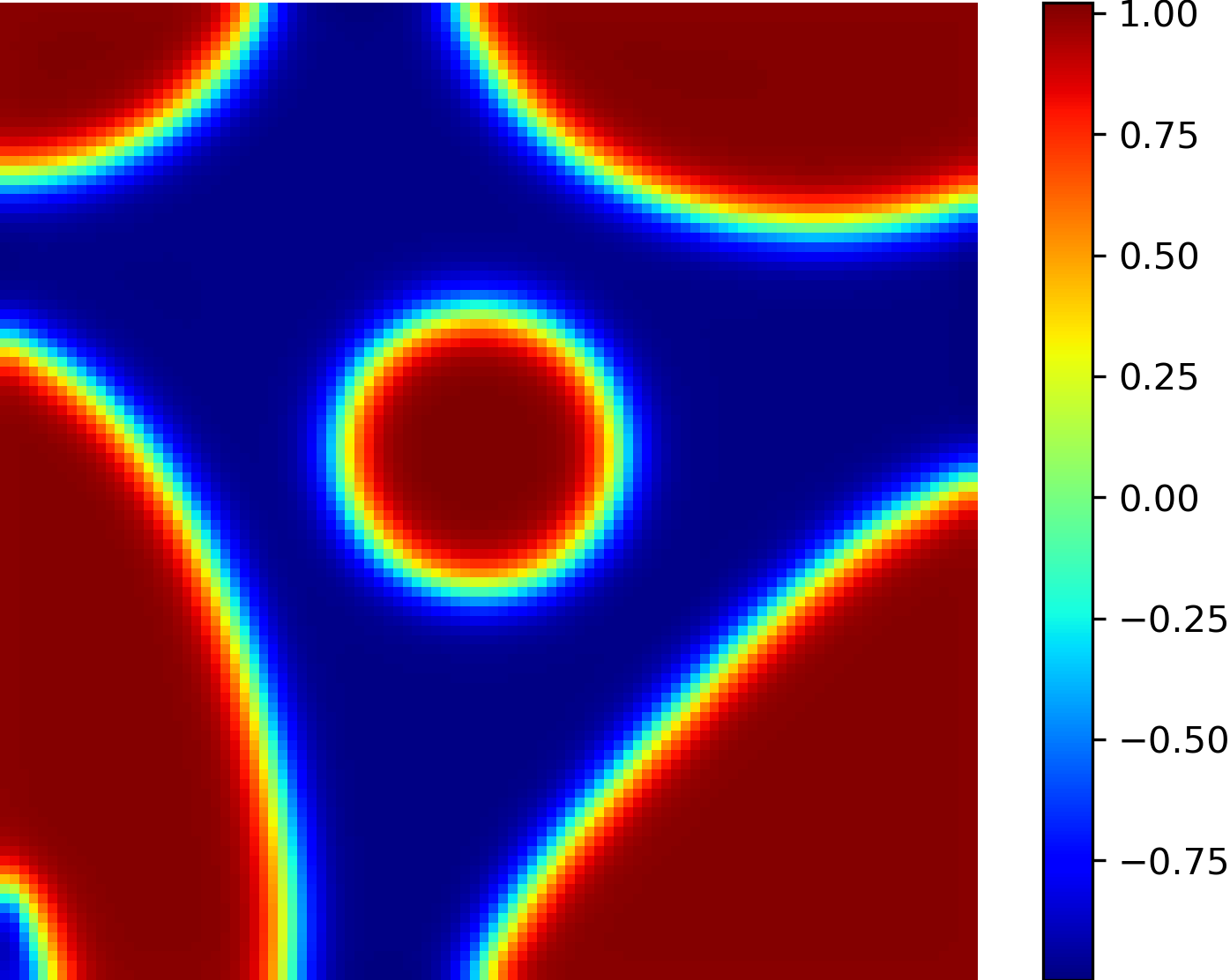}
        %\caption{t = 0.3, $\beta = 0$}
    \end{subfigure}

    \begin{subfigure}[t]{0.3\textwidth}
        \includegraphics[scale=0.4]{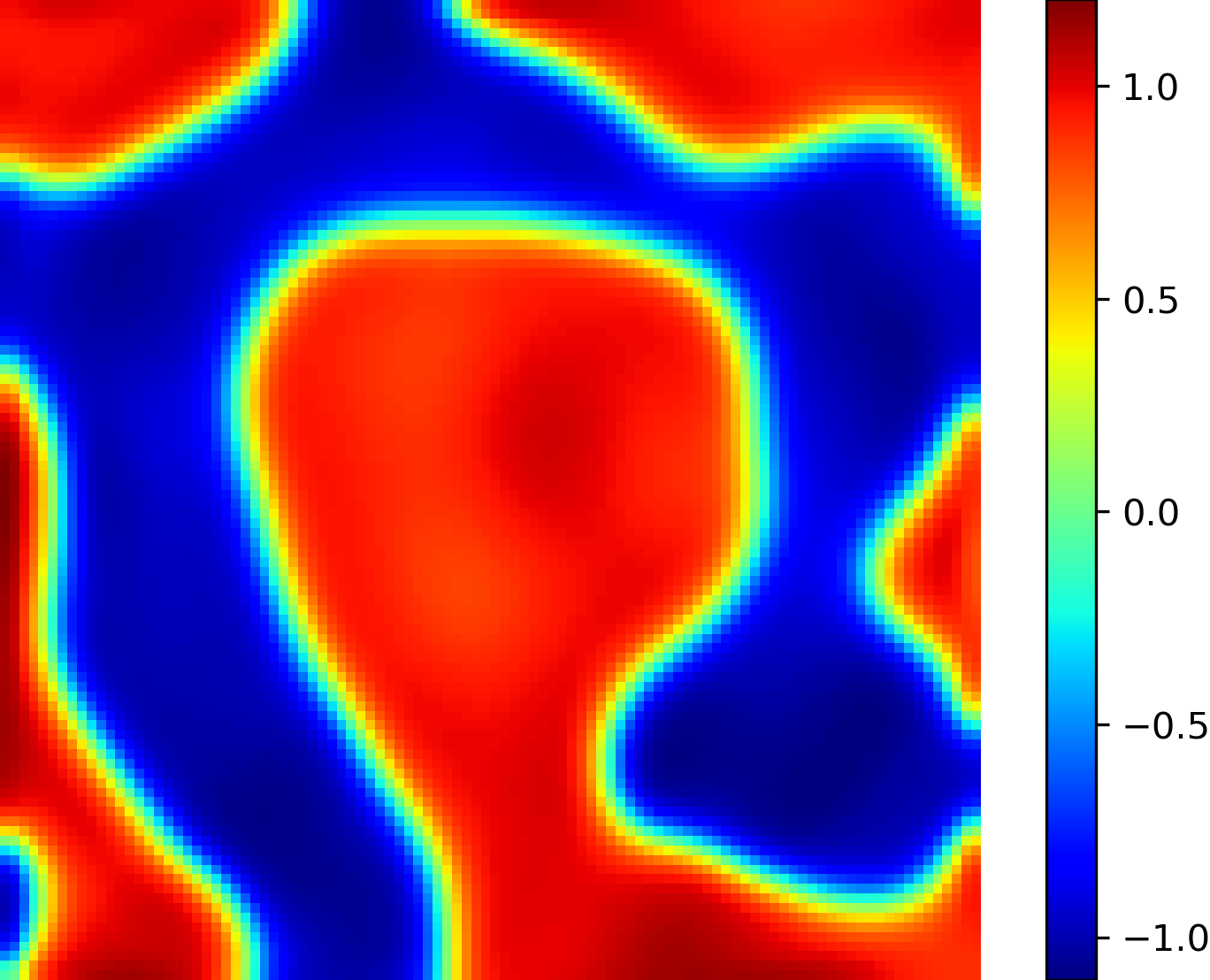}
        %\caption{t = 0.5, $\beta = 1$}
    \end{subfigure}
    \hfill
    \begin{subfigure}[t]{0.3\textwidth}
        \includegraphics[scale=0.4]{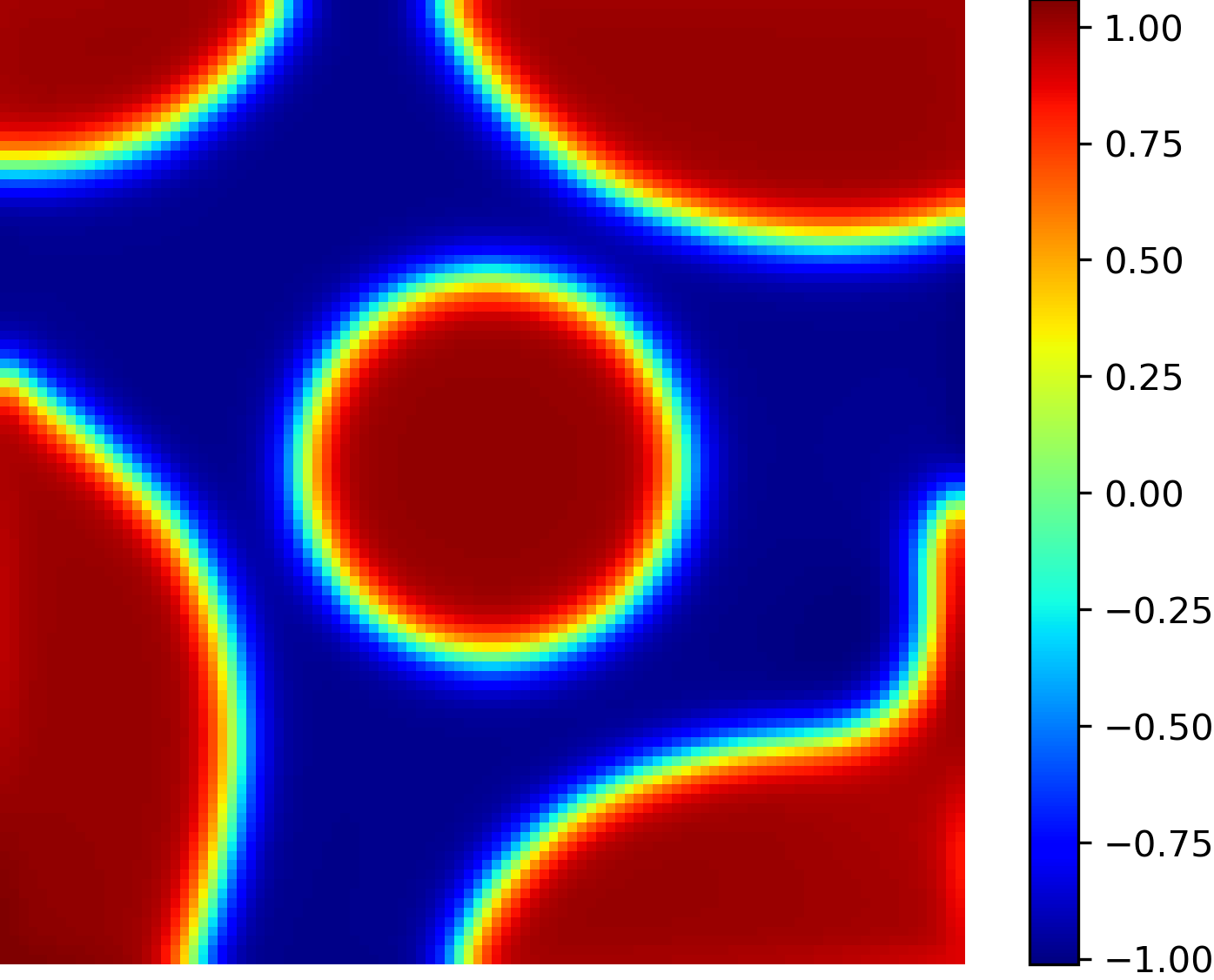}
        %\caption{t = 0.5, $\beta = 0.1$}
    \end{subfigure}
    \hfill
    \begin{subfigure}[t]{0.3\textwidth}
        \includegraphics[scale=0.4]{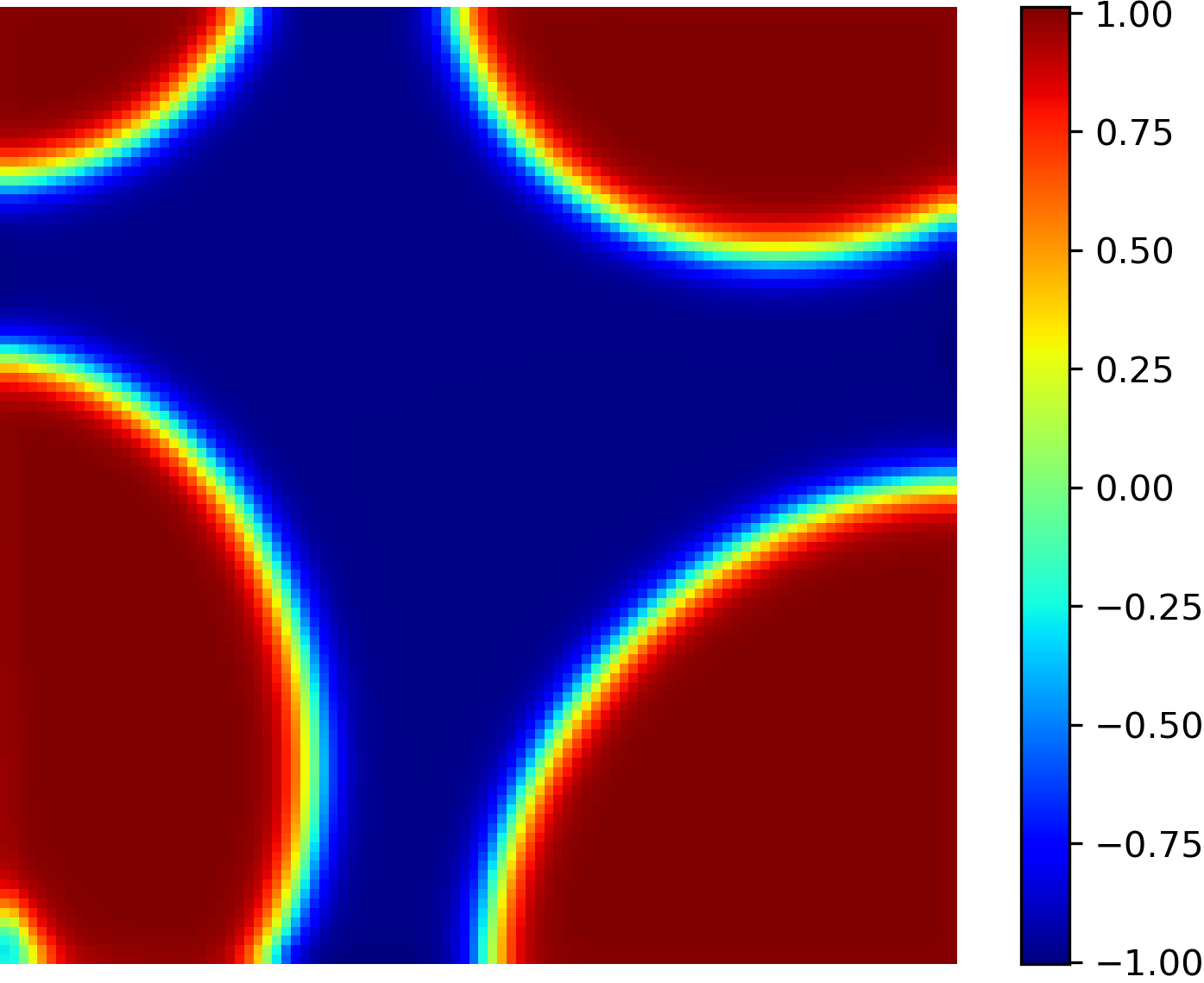}
        %\caption{t = 0.5, $\beta = 0$}
    \end{subfigure}

    \begin{subfigure}[t]{0.3\textwidth}
        \includegraphics[scale=0.4]{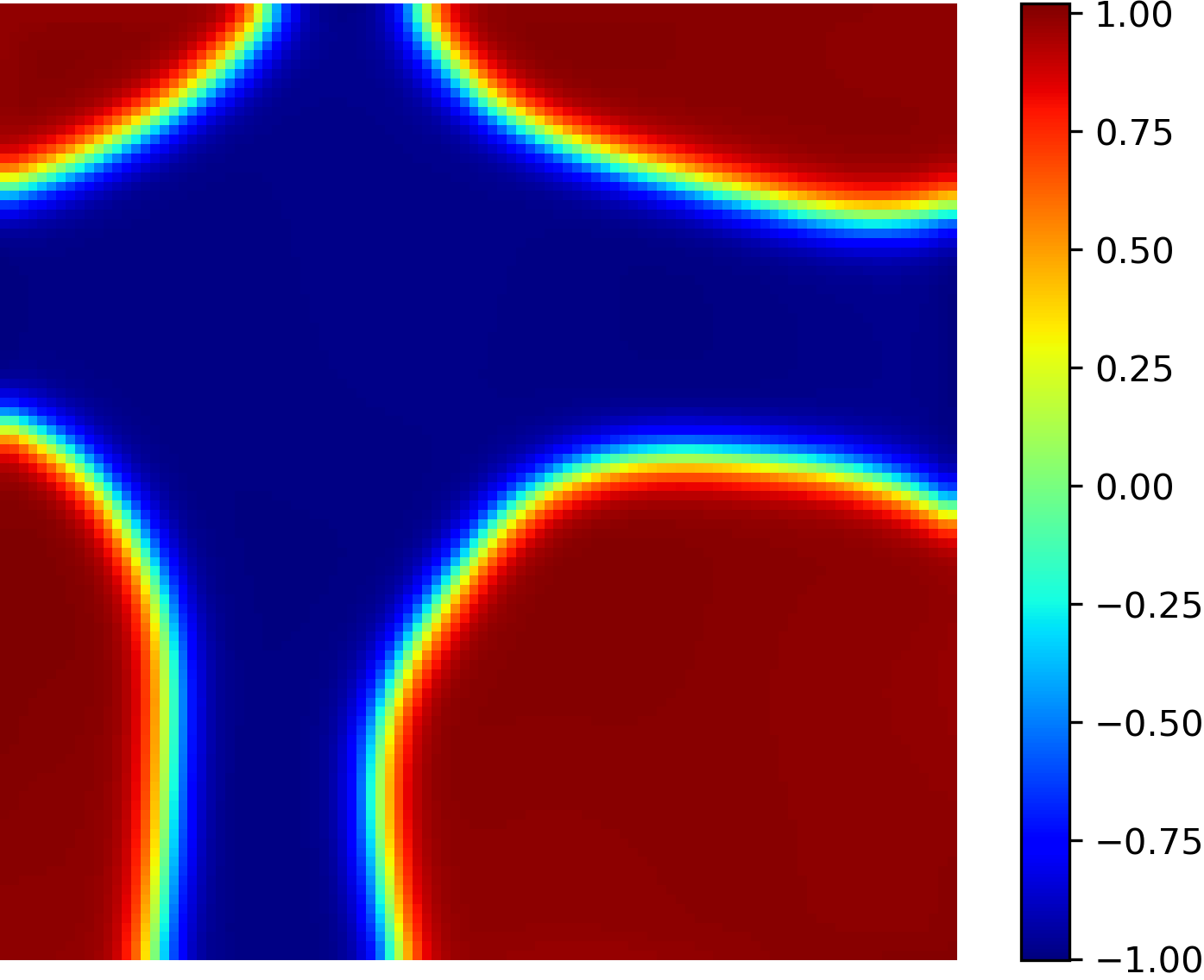}
        %\caption{t = 2.0, $\beta = 1$}
    \end{subfigure}
    \hfill
    \begin{subfigure}[t]{0.3\textwidth}
        \includegraphics[scale=0.4]{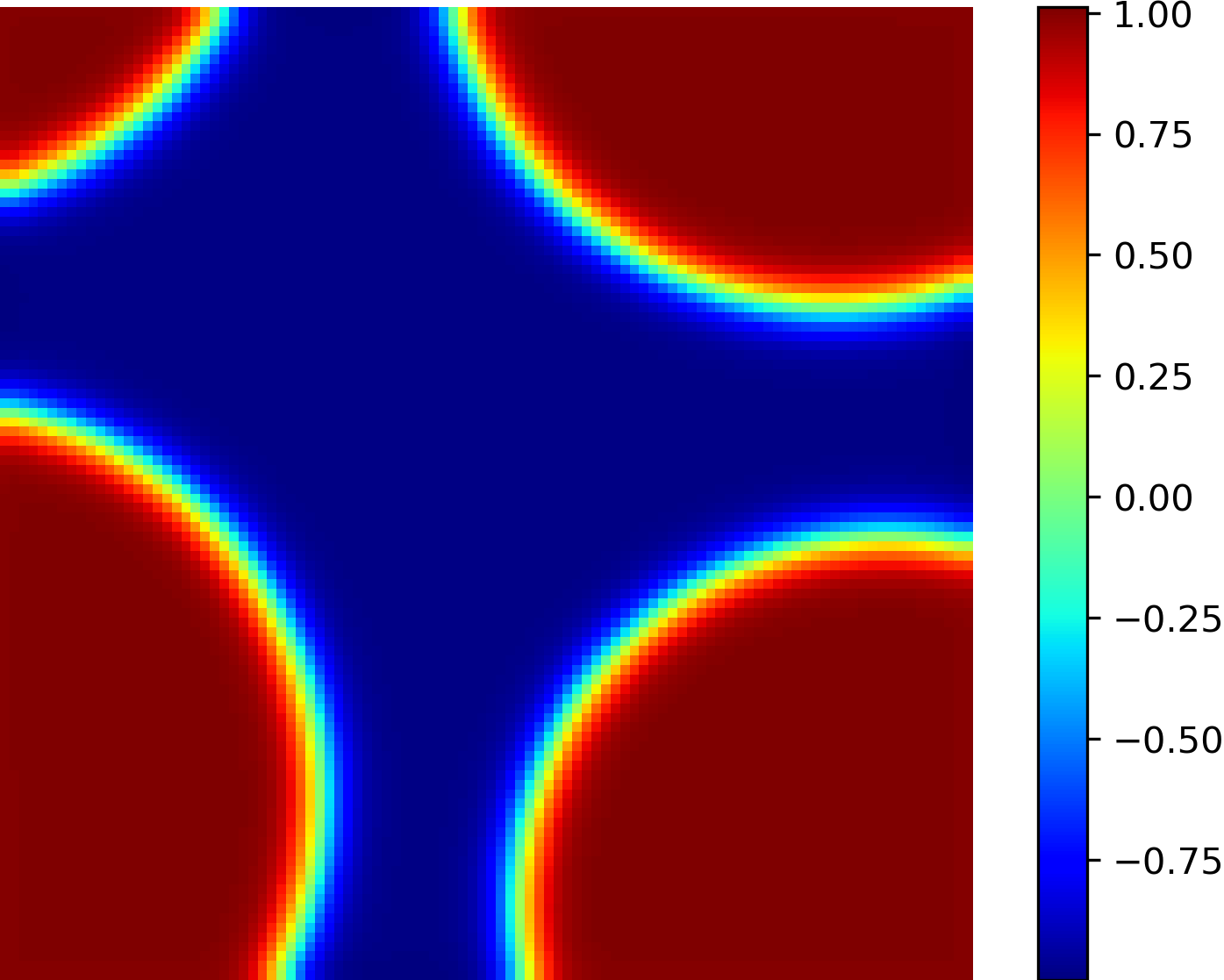}
        %\caption{t = 2.0, $\beta = 0.1$}
    \end{subfigure}
    \hfill
    \begin{subfigure}[t]{0.3\textwidth}
        \includegraphics[scale=0.4]{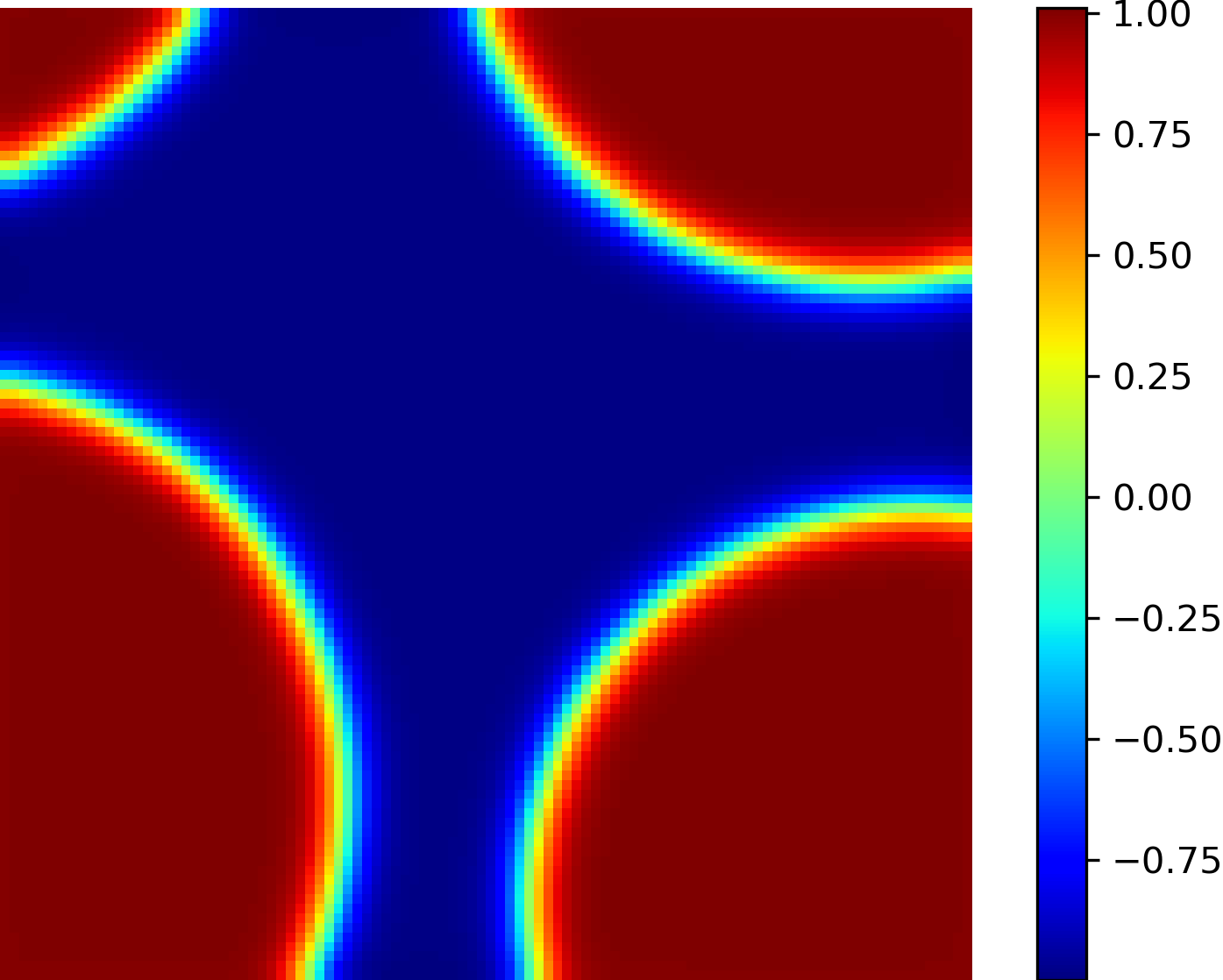}
        %\caption{t = 2.0, $\beta = 0$}
    \end{subfigure}

    \caption{Case 2: Snapshots of the numerical approximation are taken at $T= 0.015$, $0.045$, $0.3$, $0.5$, and $2.0$ with different $\beta$. Left: $\beta=1$; Middle: $\beta=0.1$; Right: $\beta=0$.}
    \label{5.7}
\end{figure}

\textbf{Case 3:} In Figure \ref{5.8}, the initial value is specified as follows,
\begin{equation}
    \phi_0(x,y) = sin(2\pi x)cos(2\pi y),\quad (x,y)\in \Omega\cup\Gamma.
\end{equation}
By reducing the value of $\beta$, we find that the system is reaching the steady state fast in Figure \ref{5.10}. Meanwhile the Figure \ref{5.9} indicates that the discrete energy is declining slowly by enlarging the value of $\beta$. The conservation of mass in the bulk and on the boundary is consistently maintained with different values of $\beta$.

\begin{figure}[!htbp]
    \centering
    \includegraphics[width=0.4\textwidth]{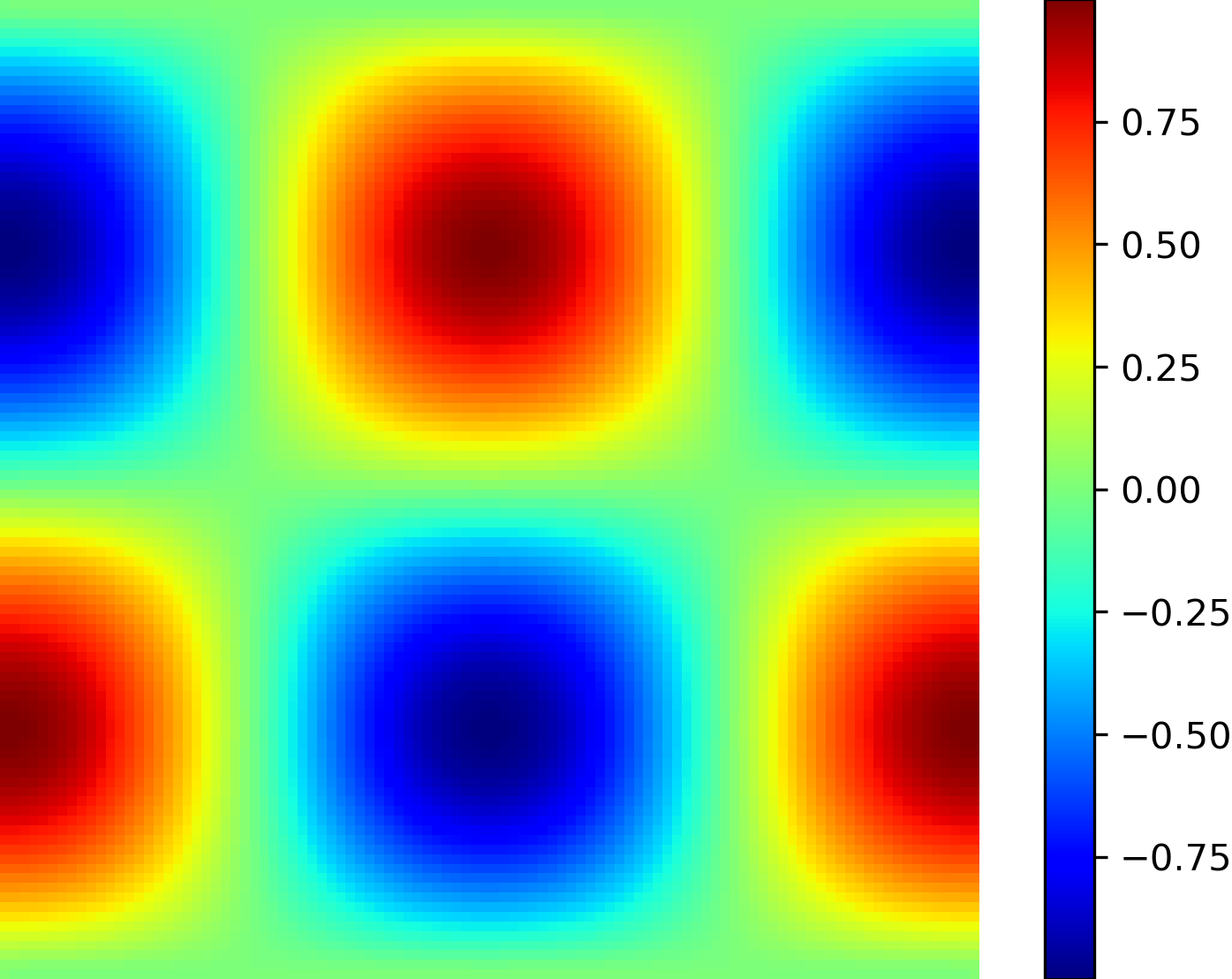}
    \caption{The initial data of Case 3.}

    \label{5.8}

\end{figure}

\begin{figure}[!htbp]
    \centering
    \scalebox{0.9}{
        \begin{minipage}{\textwidth}
            \begin{subfigure}[b]{0.49\textwidth}
                \includegraphics[width=\textwidth]{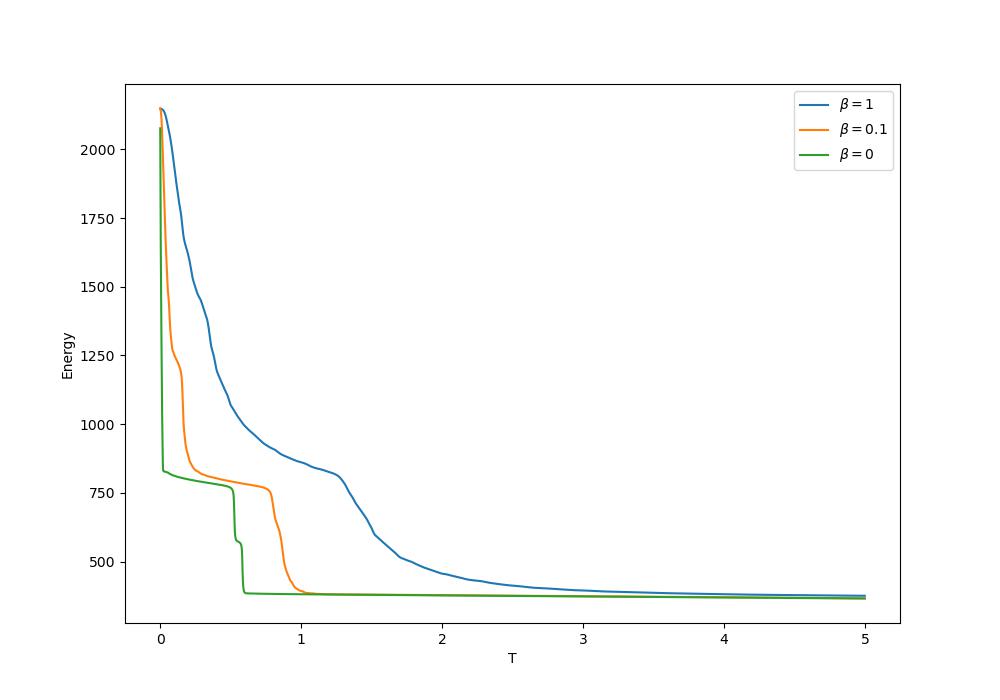}
                \caption{Energy curves with different $\beta$.}
            \end{subfigure}
            \hfill
            \begin{subfigure}[b]{0.49\textwidth}
                \includegraphics[width=\textwidth]{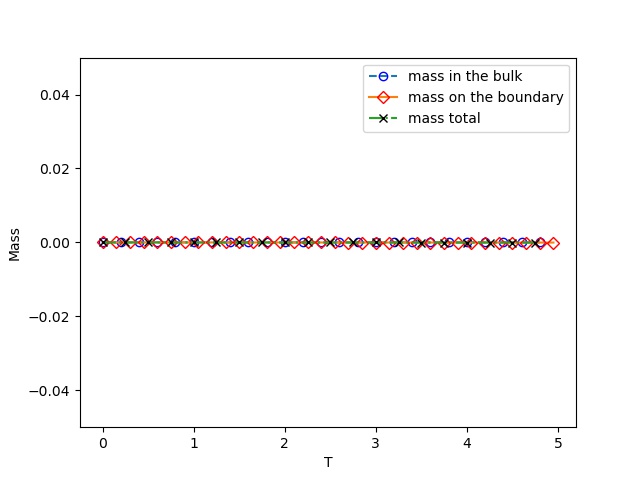}
                \caption{The mass with $\beta = 1$.}
            \end{subfigure}
            \hfill
            \begin{subfigure}[b]{0.49\textwidth}
                \includegraphics[width=\textwidth]{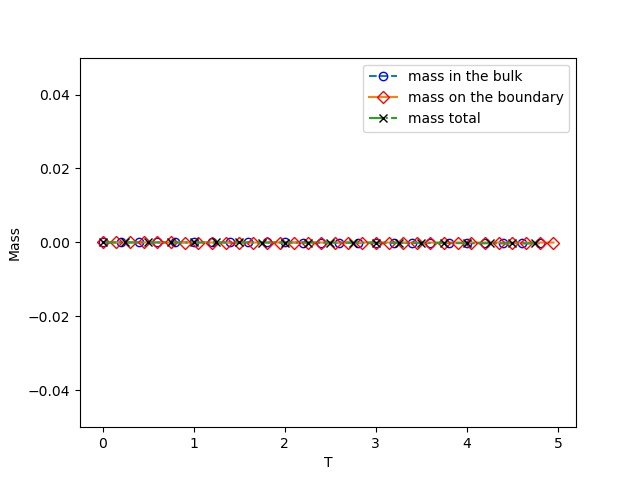}%{f3/2update4_01_beta_0.1_beta2_0.1_dt_0.0005_nu_0.002/mass_combined_beta_0.1_beta2_0.1.jpg}
                \caption{The mass with $\beta = 0.1$.}
            \end{subfigure}
            \hfill
            \begin{subfigure}[b]{0.49\textwidth}
                \includegraphics[width=\textwidth]{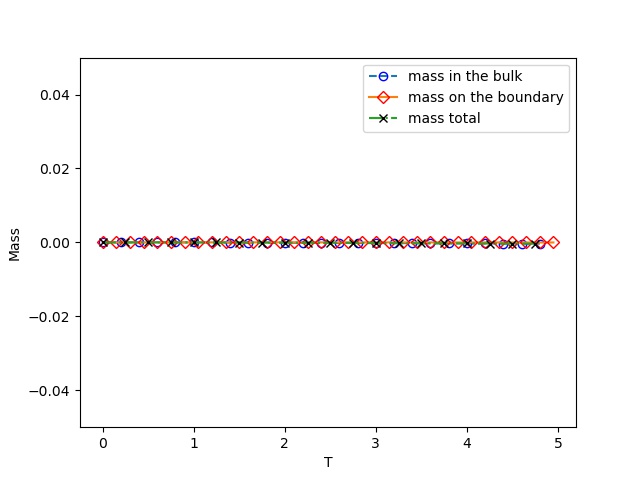}%{figure3/2update4_01_beta_0_beta2_0_dt_0.0005_nu_0.002/mass_combined_beta_0_beta2_0.jpg}
                \caption{The mass with $\beta = 0$.}
            \end{subfigure}
        \end{minipage}
    }
    \caption{The energy evolution and the mass evolutions of Case 3.}

    \label{5.9}

\end{figure}

\begin{figure}[!htbp]
    \centering
    \begin{subfigure}[t]{0.3\textwidth}
        \includegraphics[scale=0.4]{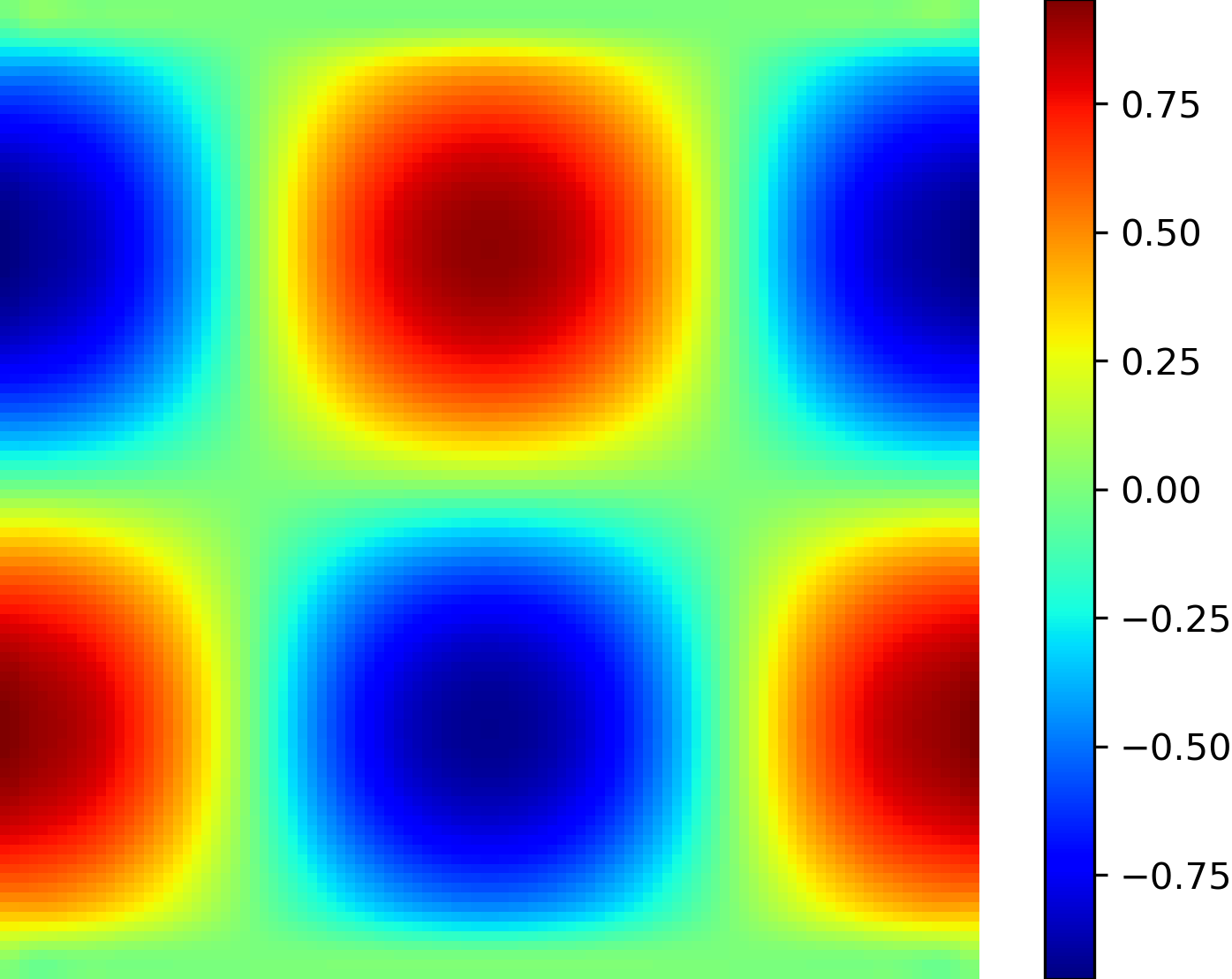}
        %\caption{t = 0.015, $\beta = 1$}
    \end{subfigure}
    \hfill
    \begin{subfigure}[t]{0.3\textwidth}
        \includegraphics[scale=0.4]{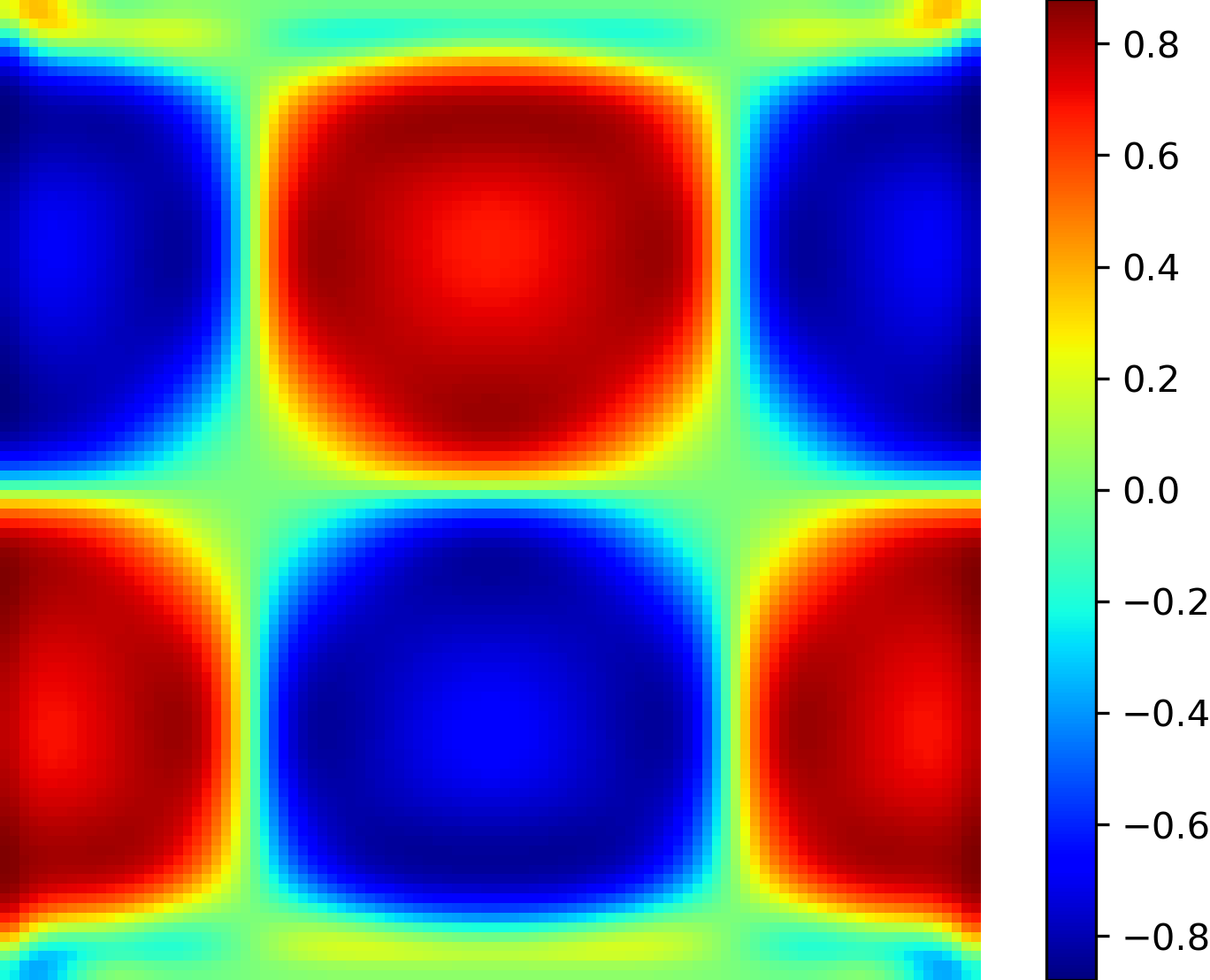}%{figure3/2update4_01_beta_0.1_beta2_0.1_dt_0.0005_nu_0.002/30_phi.jpg}
        %\caption{t = 0.015, $\beta = 0.1$}
    \end{subfigure}
    \hfill
    \begin{subfigure}[t]{0.3\textwidth}
        \includegraphics[scale=0.4]{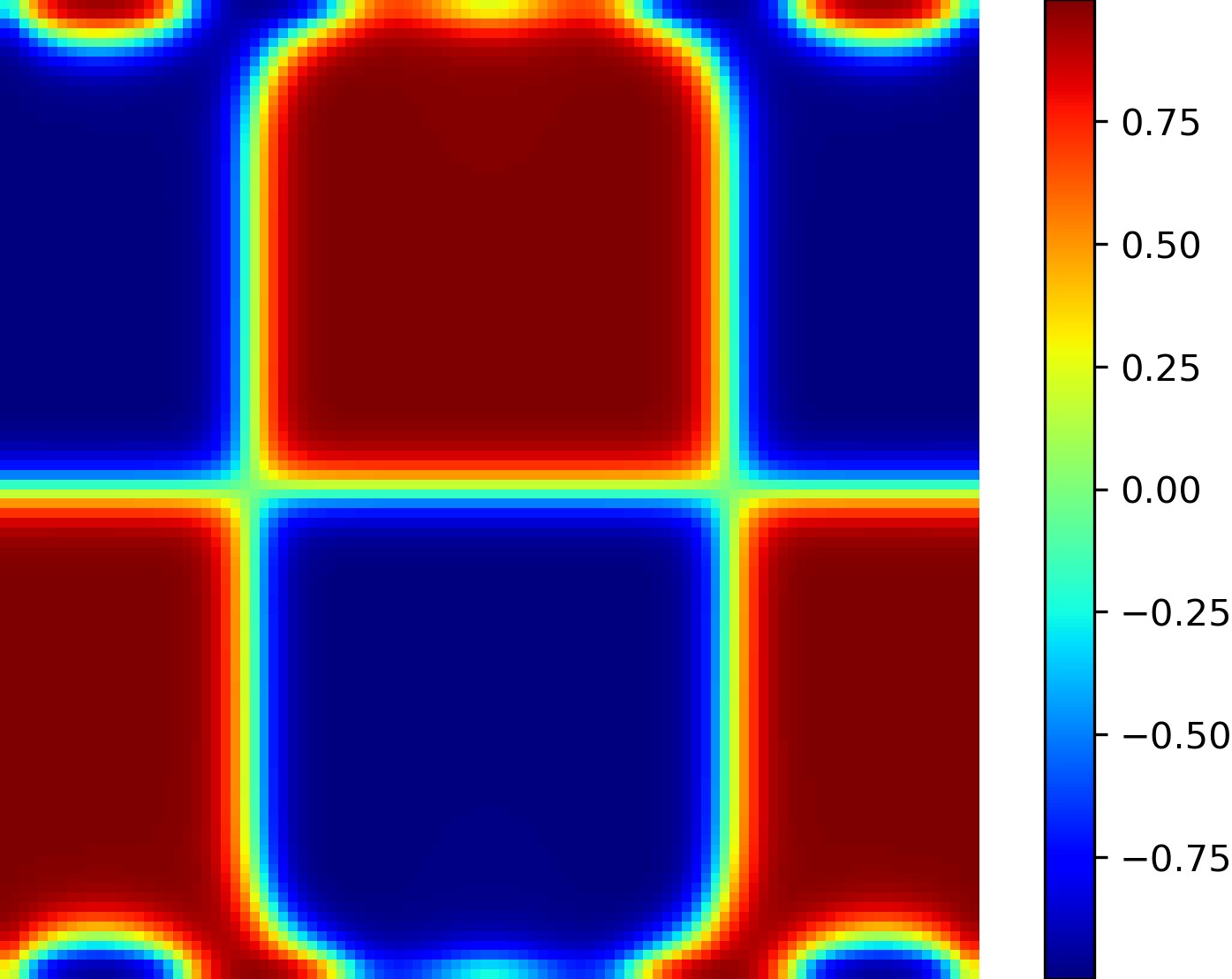}%{figure3/2update4_01_beta_0_beta2_0_dt_0.0005_nu_0.002/30_phi.jpg}
        %\caption{t = 0.015, $\beta = 0$}
    \end{subfigure}

    \begin{subfigure}[t]{0.3\textwidth}
        \includegraphics[scale=0.4]{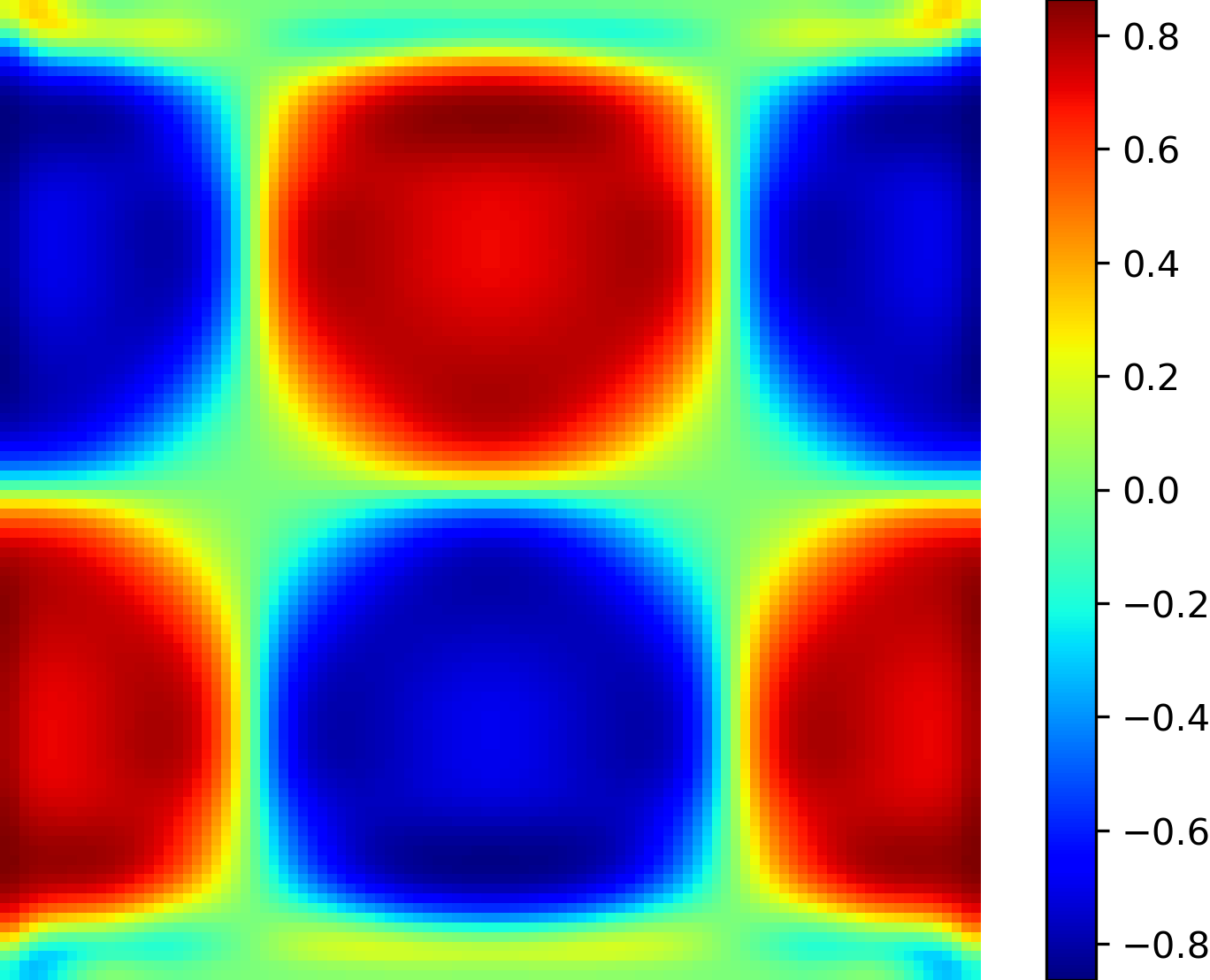}
        %\caption{t = 0.04, $\beta = 1$}
    \end{subfigure}
    \hfill
    \begin{subfigure}[t]{0.3\textwidth}
        \includegraphics[scale=0.4]{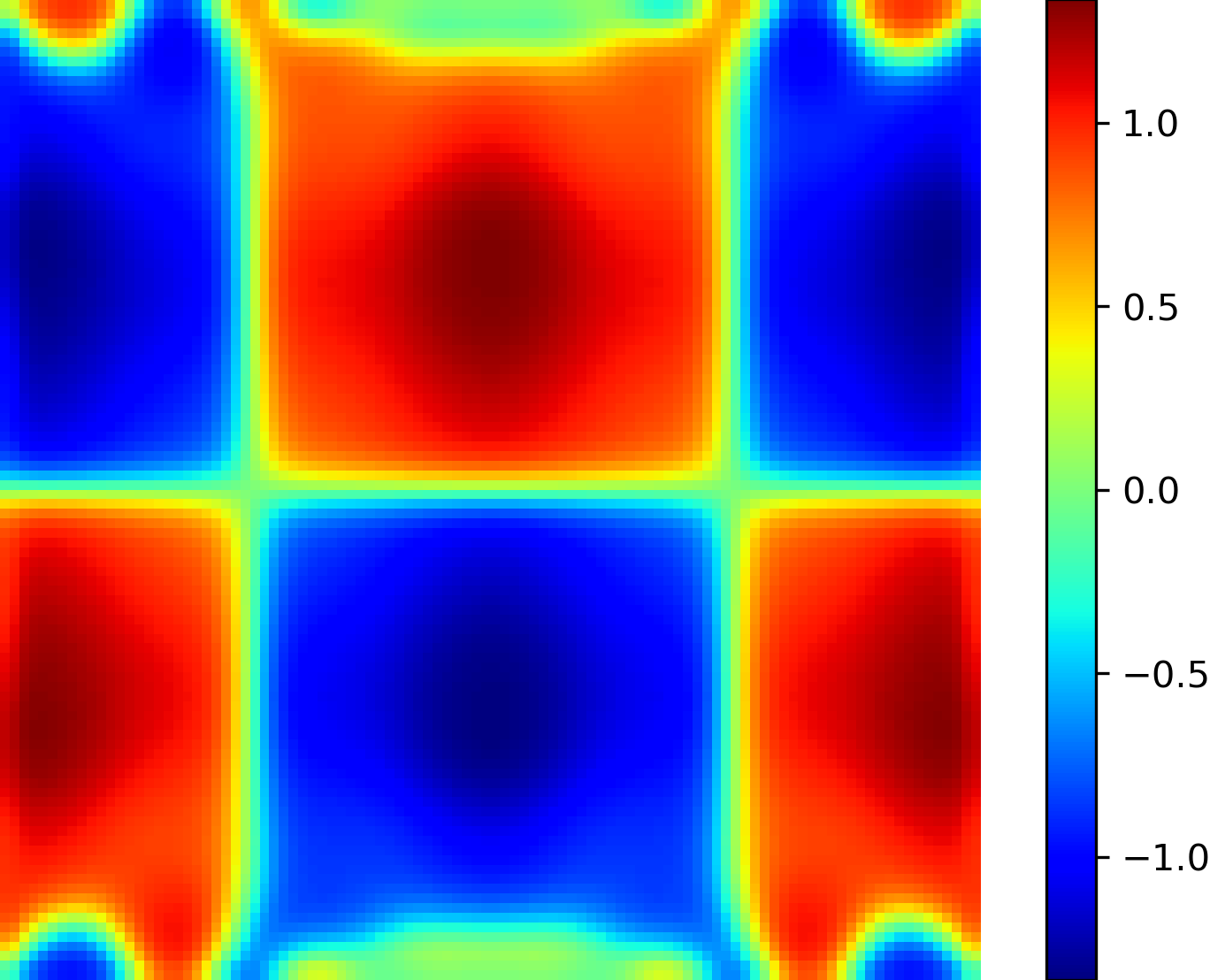}%{figure3/2update4_01_beta_0.1_beta2_0.1_dt_0.0005_nu_0.002/80_phi.jpg}
        %\caption{t = 0.04, $\beta = 0.1$}
    \end{subfigure}
    \hfill
    \begin{subfigure}[t]{0.3\textwidth}
        \includegraphics[scale=0.4]{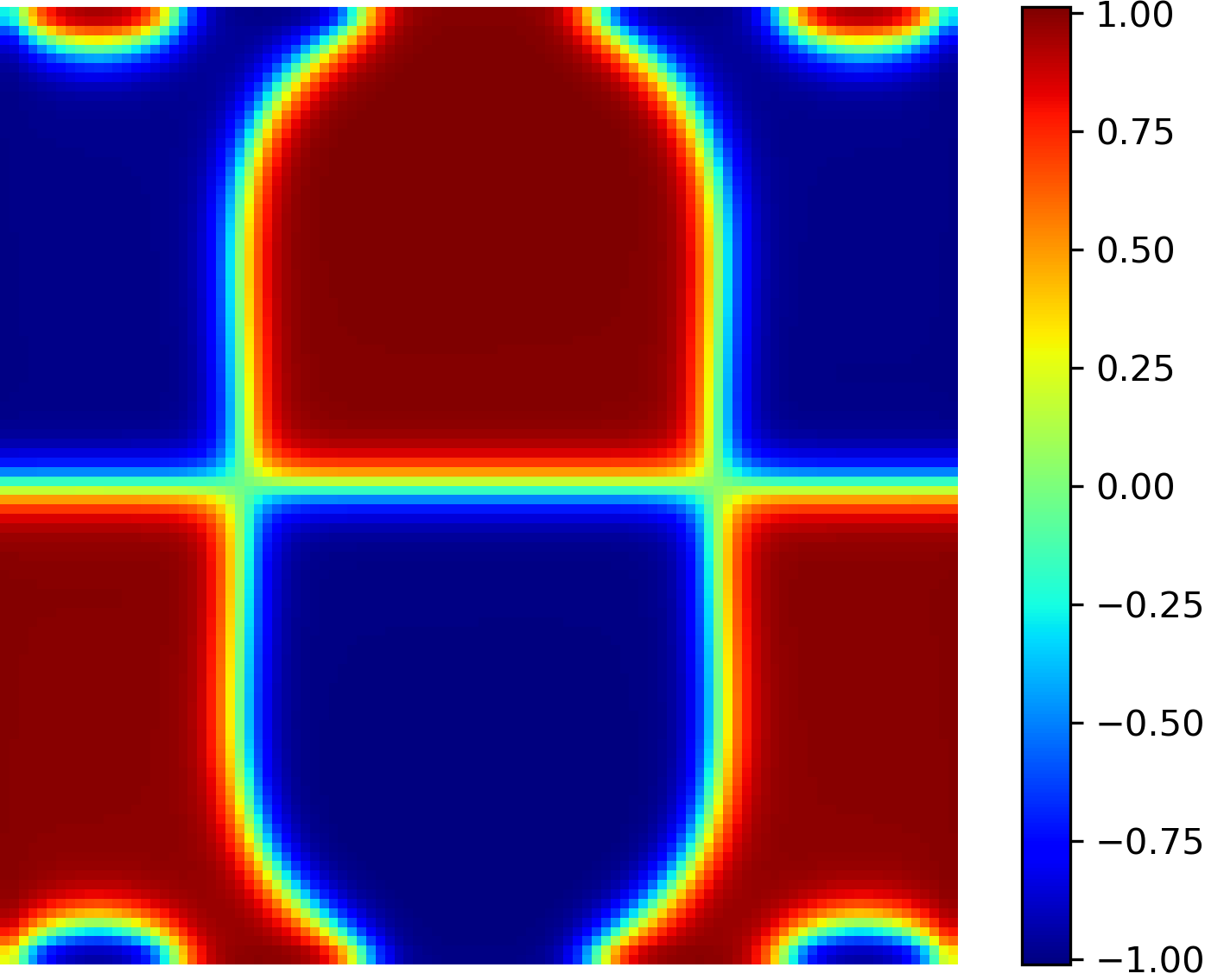}%{figure3/2update4_01_beta_0_beta2_0_dt_0.0005_nu_0.002/80_phi.jpg}
        %\caption{t = 0.04, $\beta = 0$}
    \end{subfigure}

    \begin{subfigure}[t]{0.3\textwidth}
        \includegraphics[scale=0.4]{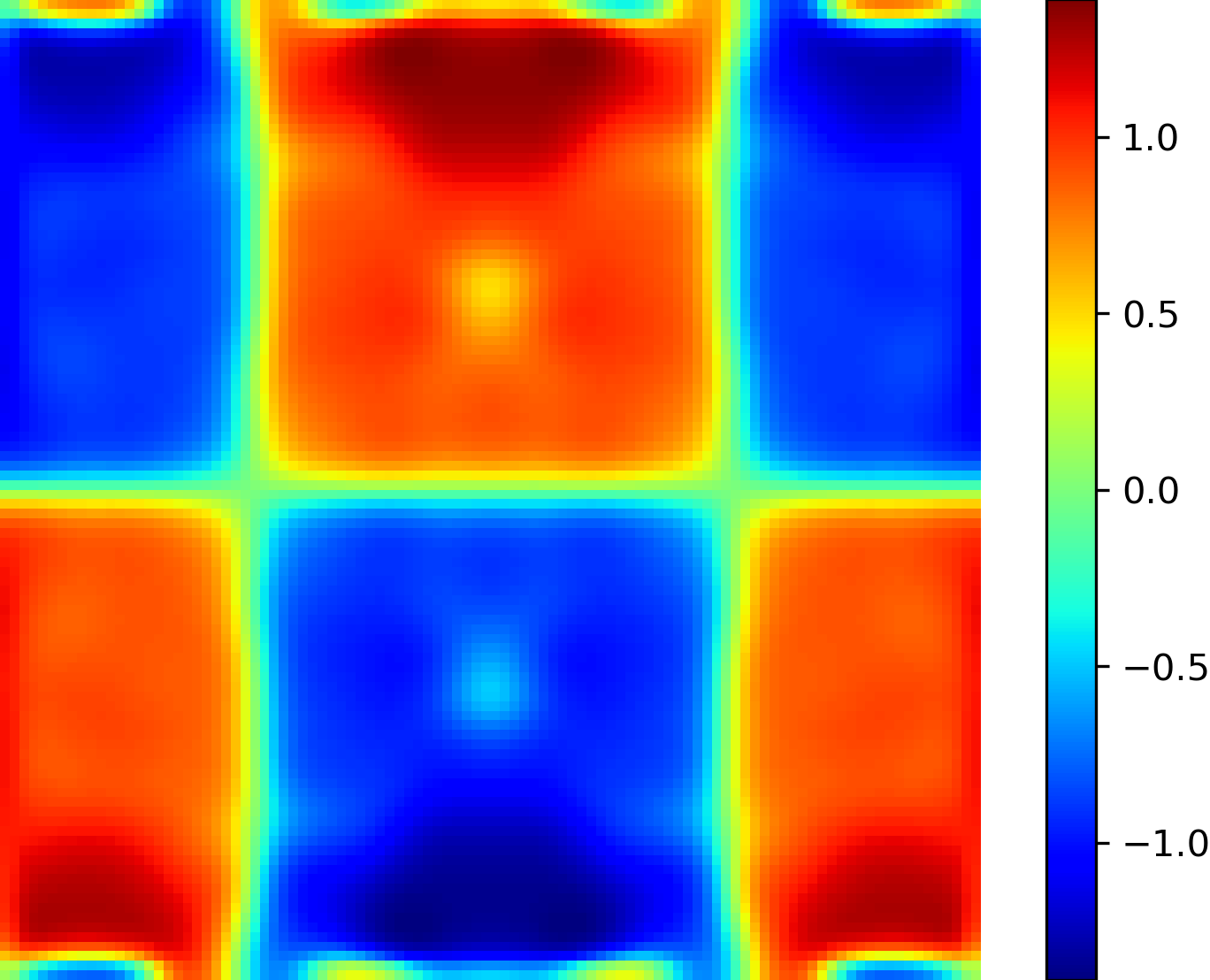}
        %\caption{t = 0.20, $\beta = 1$}
    \end{subfigure}
    \hfill
    \begin{subfigure}[t]{0.3\textwidth}
        \includegraphics[scale=0.4]{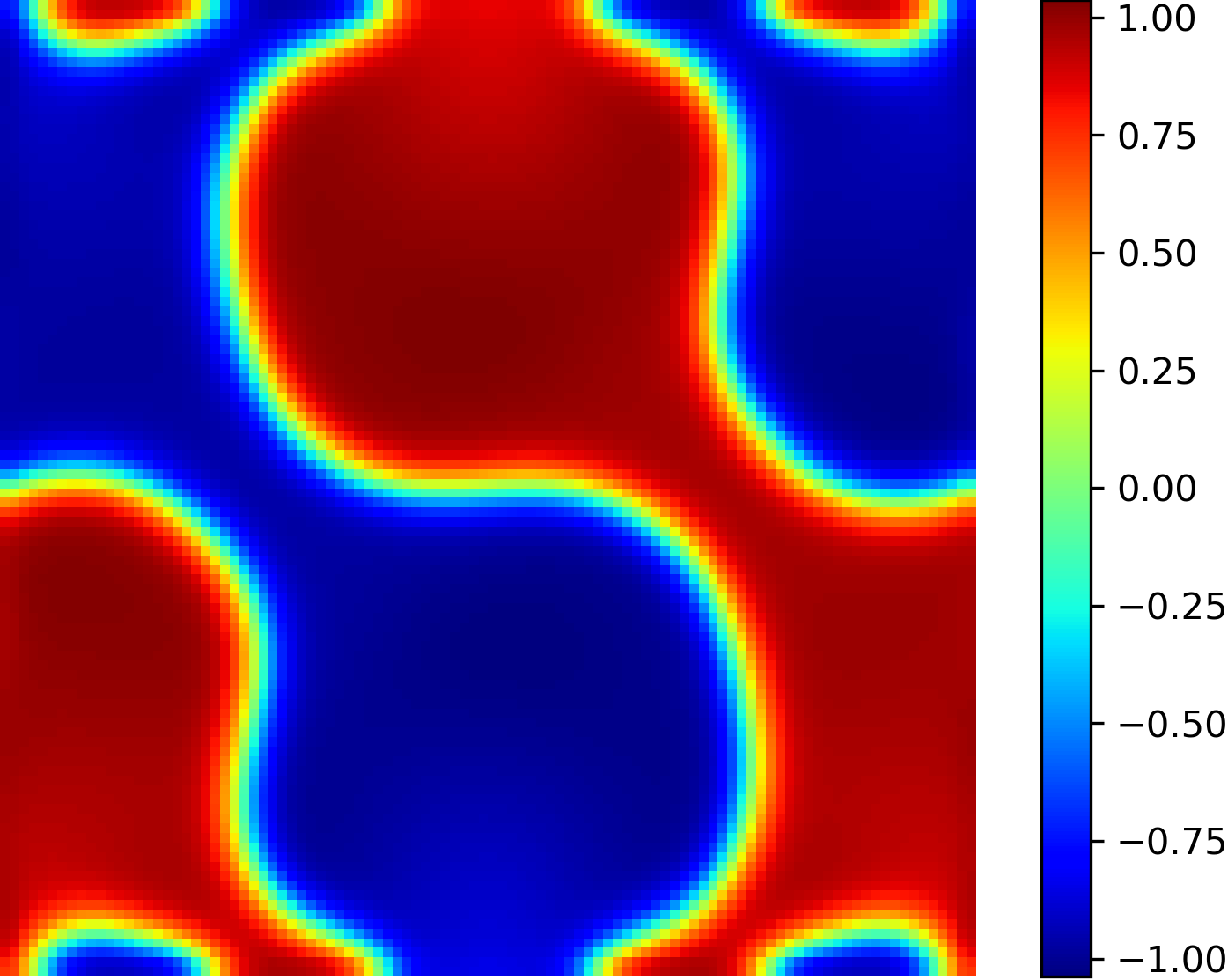}%{figure3/2update4_01_beta_0.1_beta2_0.1_dt_0.0005_nu_0.002/400_phi.jpg}
        %\caption{t = 0.20, $\beta = 0.1$}
    \end{subfigure}
    \hfill
    \begin{subfigure}[t]{0.3\textwidth}
        \includegraphics[scale=0.4]{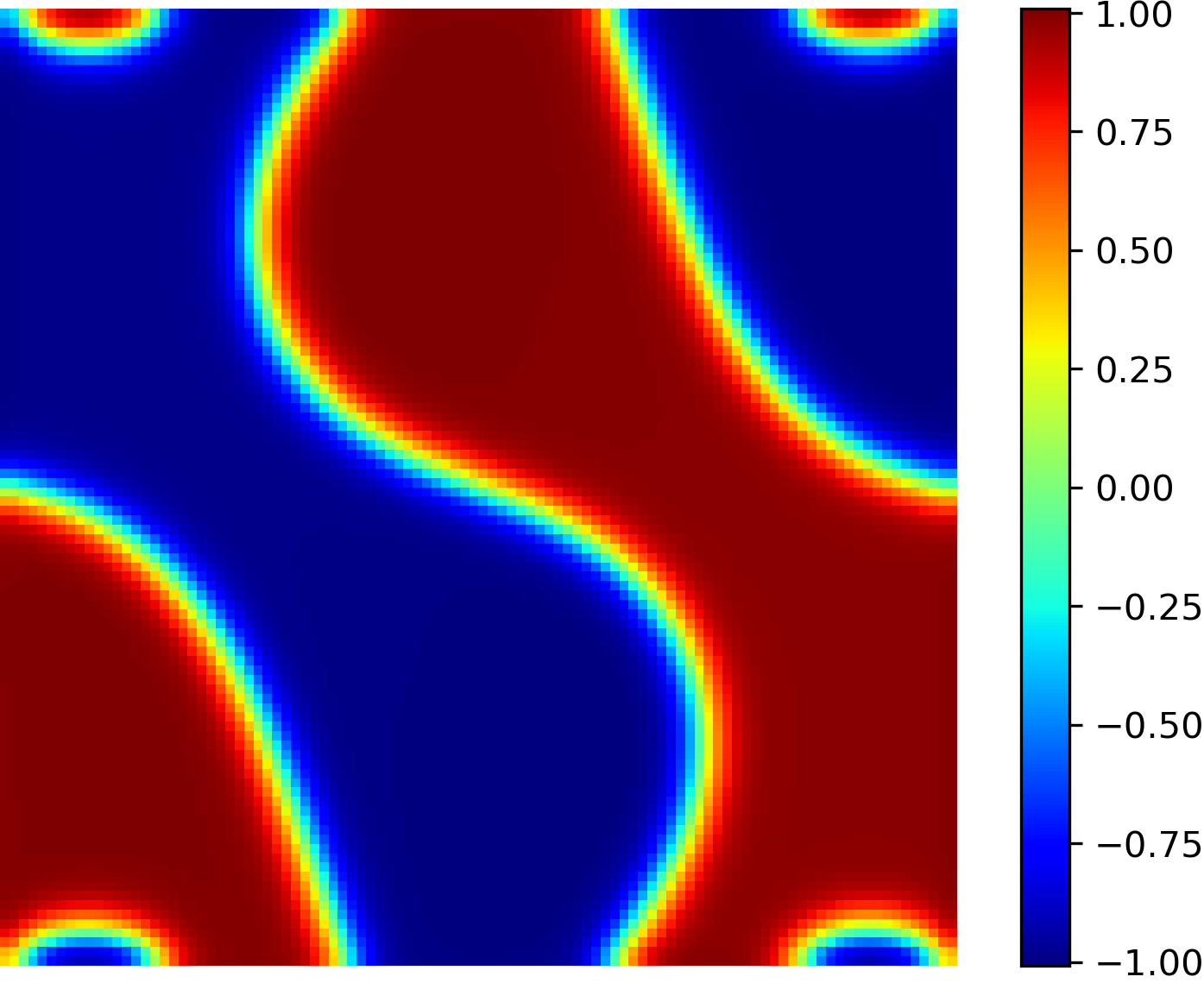}%{f3/2update4_01_beta_0_beta2_0_dt_0.0005_nu_0.002/400_phi.jpg}
        %\caption{t = 0.20, $\beta = 0$}
    \end{subfigure}

    \begin{subfigure}[t]{0.3\textwidth}
        \includegraphics[scale=0.4]{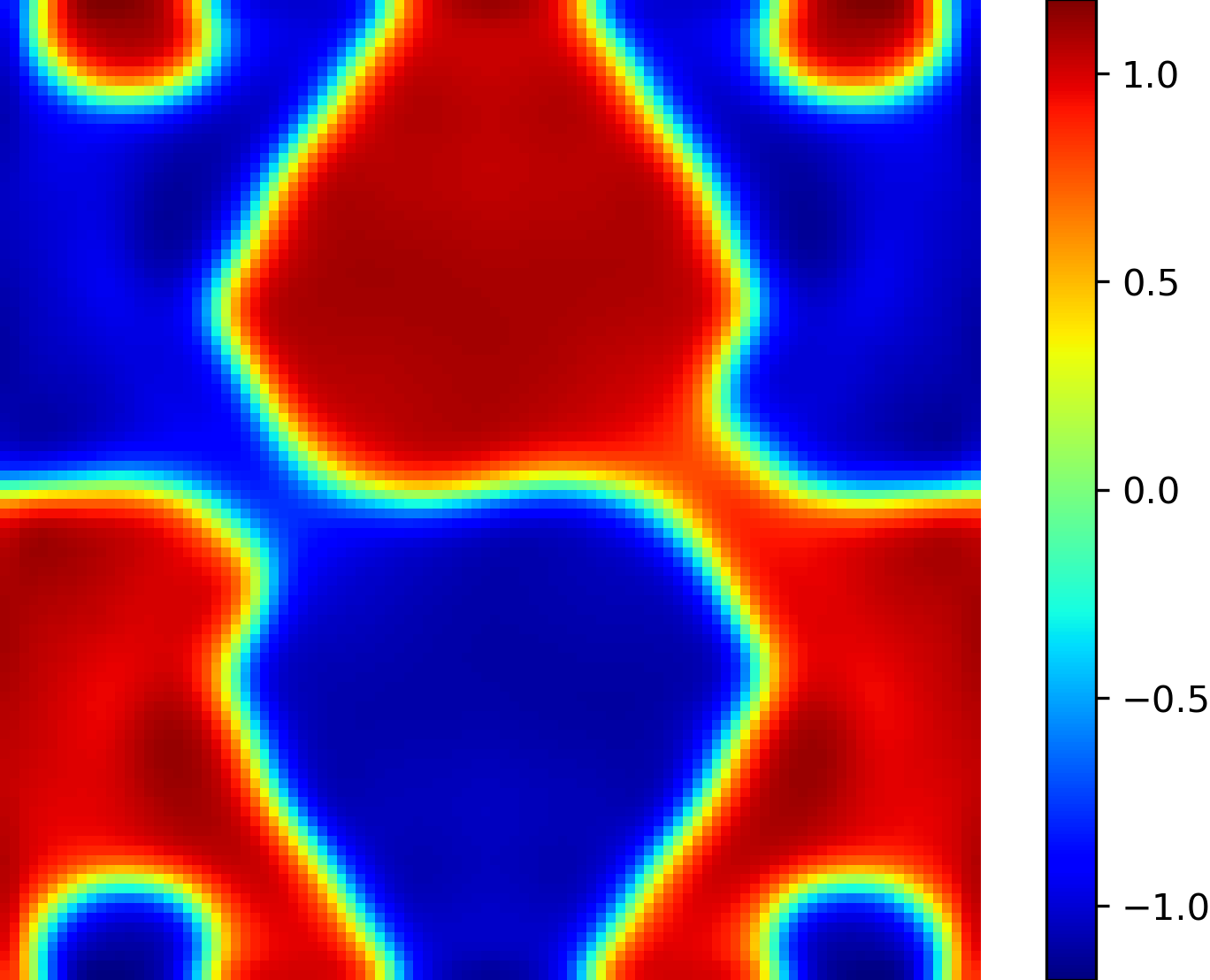}
        %\caption{t = 0.5, $\beta = 1$}
    \end{subfigure}
    \hfill
    \begin{subfigure}[t]{0.3\textwidth}
        \includegraphics[scale=0.4]{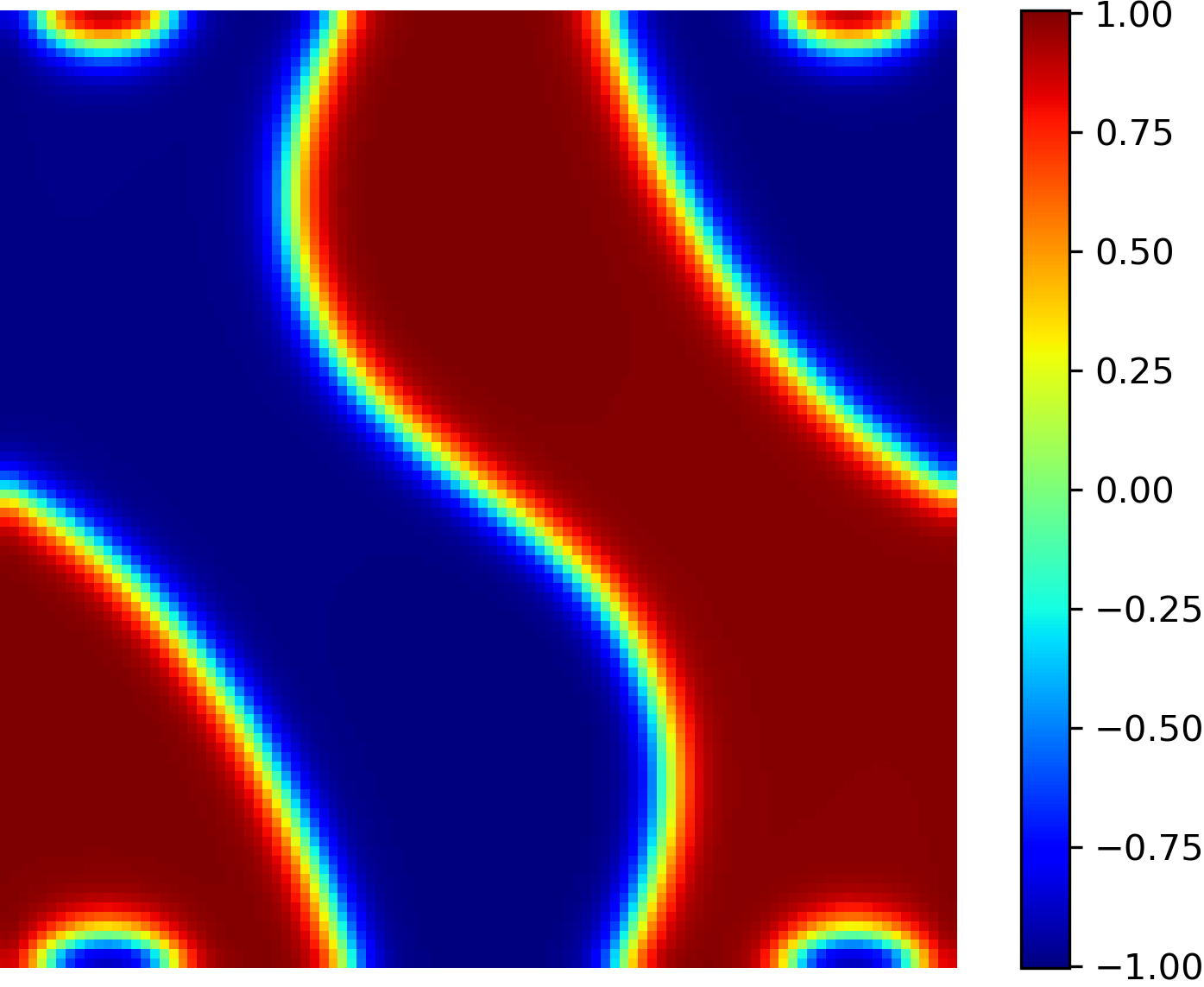}%{figure3/2update4_01_beta_0.1_beta2_0.1_dt_0.0005_nu_0.002/1000_phi.jpg}
        %\caption{t = 0.5, $\beta = 0.1$}
    \end{subfigure}
    \hfill
    \begin{subfigure}[t]{0.3\textwidth}
        \includegraphics[scale=0.4]{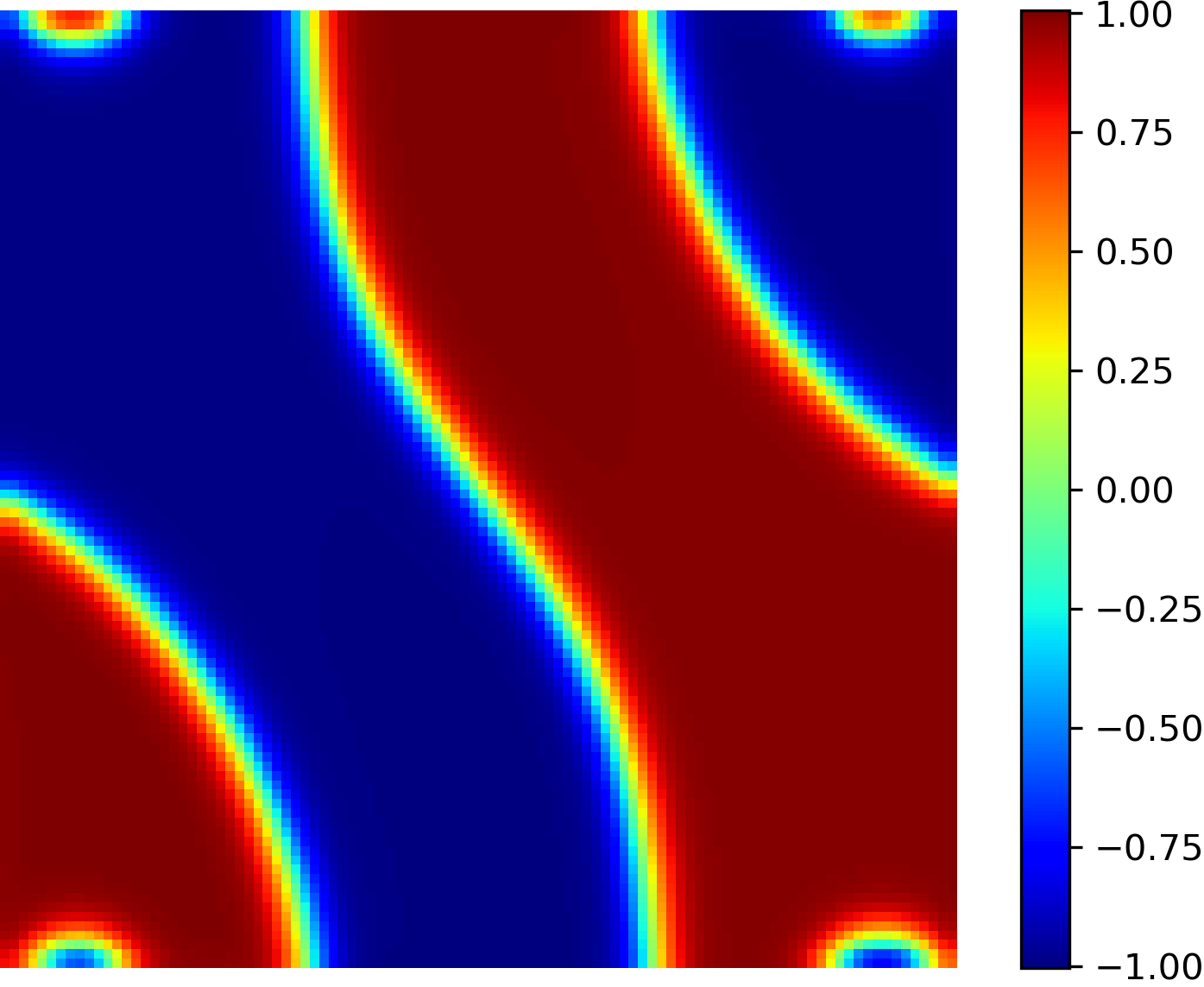}%{figure3/2update4_01_beta_0_beta2_0_dt_0.0005_nu_0.002/1000_phi.jpg}
        %\caption{t = 0.5, $\beta = 0$}
    \end{subfigure}

    \begin{subfigure}[t]{0.3\textwidth}
        \includegraphics[scale=0.4]{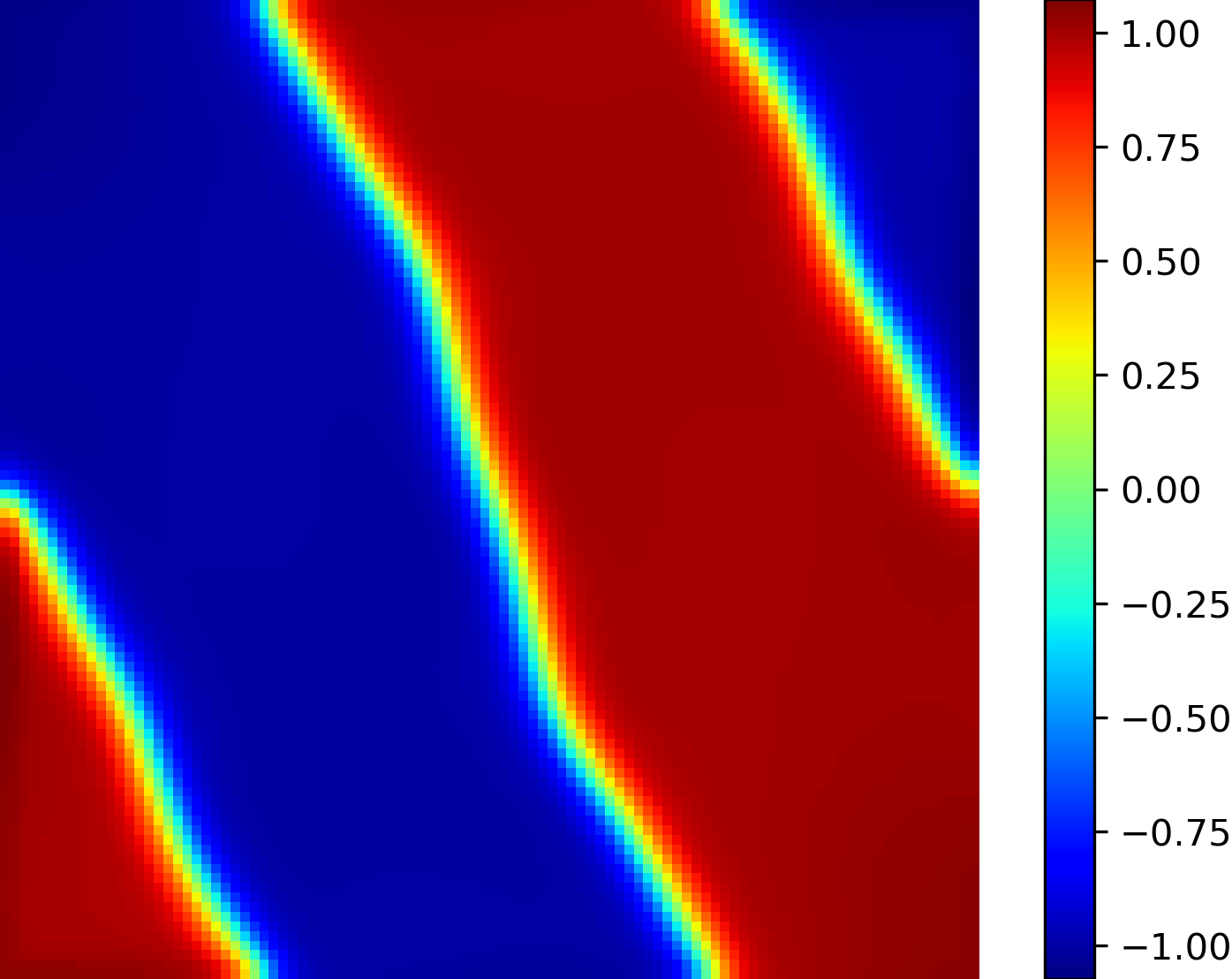}
        %\caption{t = 2.5, $\beta = 1$}
    \end{subfigure}
    \hfill
    \begin{subfigure}[t]{0.3\textwidth}
        \includegraphics[scale=0.4]{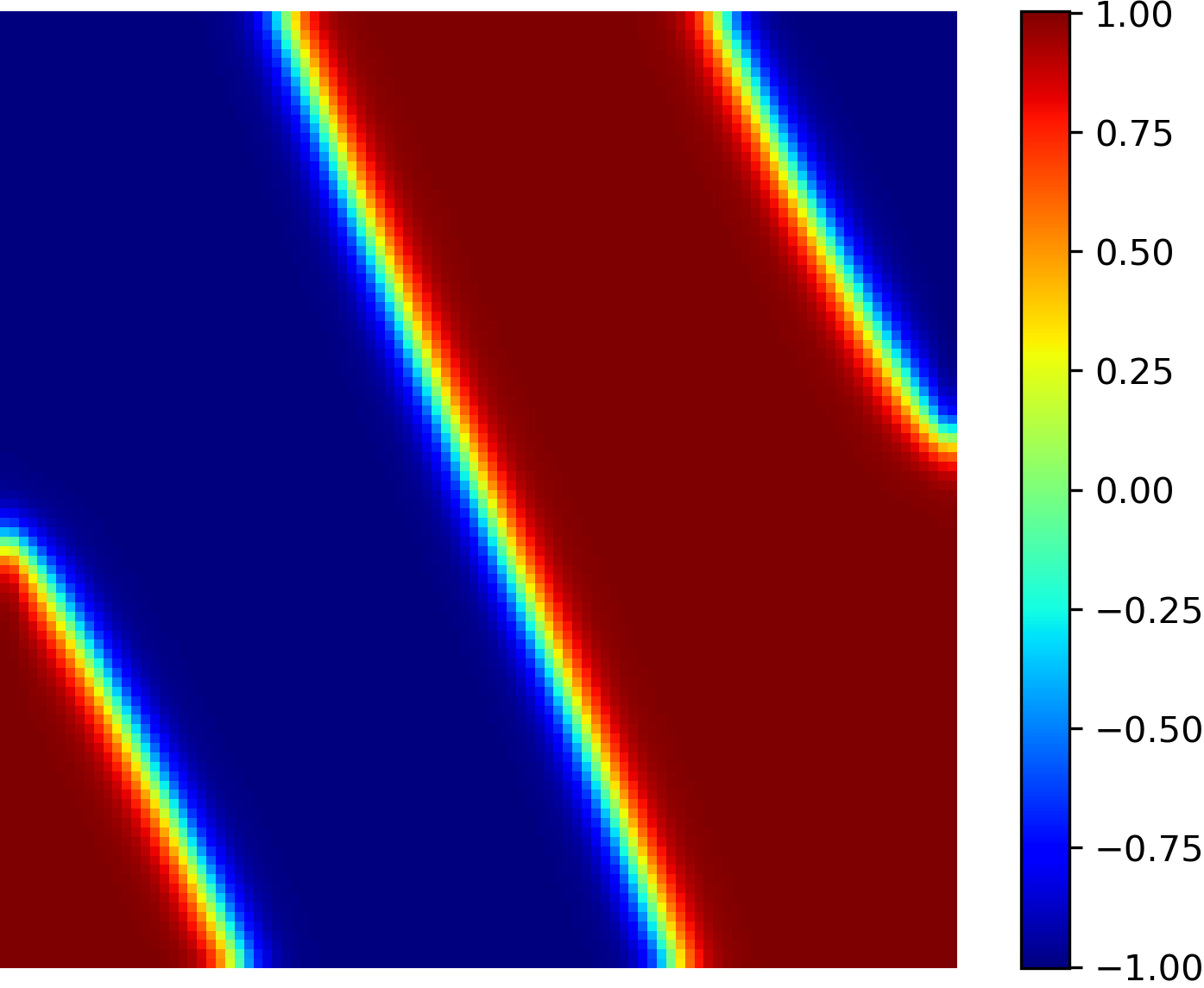}%{figure3/2update4_01_beta_0.1_beta2_0.1_dt_0.0005_nu_0.002/5000_phi.jpg}
        %\caption{t = 2.5, $\beta = 0.1$}
    \end{subfigure}
    \hfill
    \begin{subfigure}[t]{0.3\textwidth}
        \includegraphics[scale=0.4]{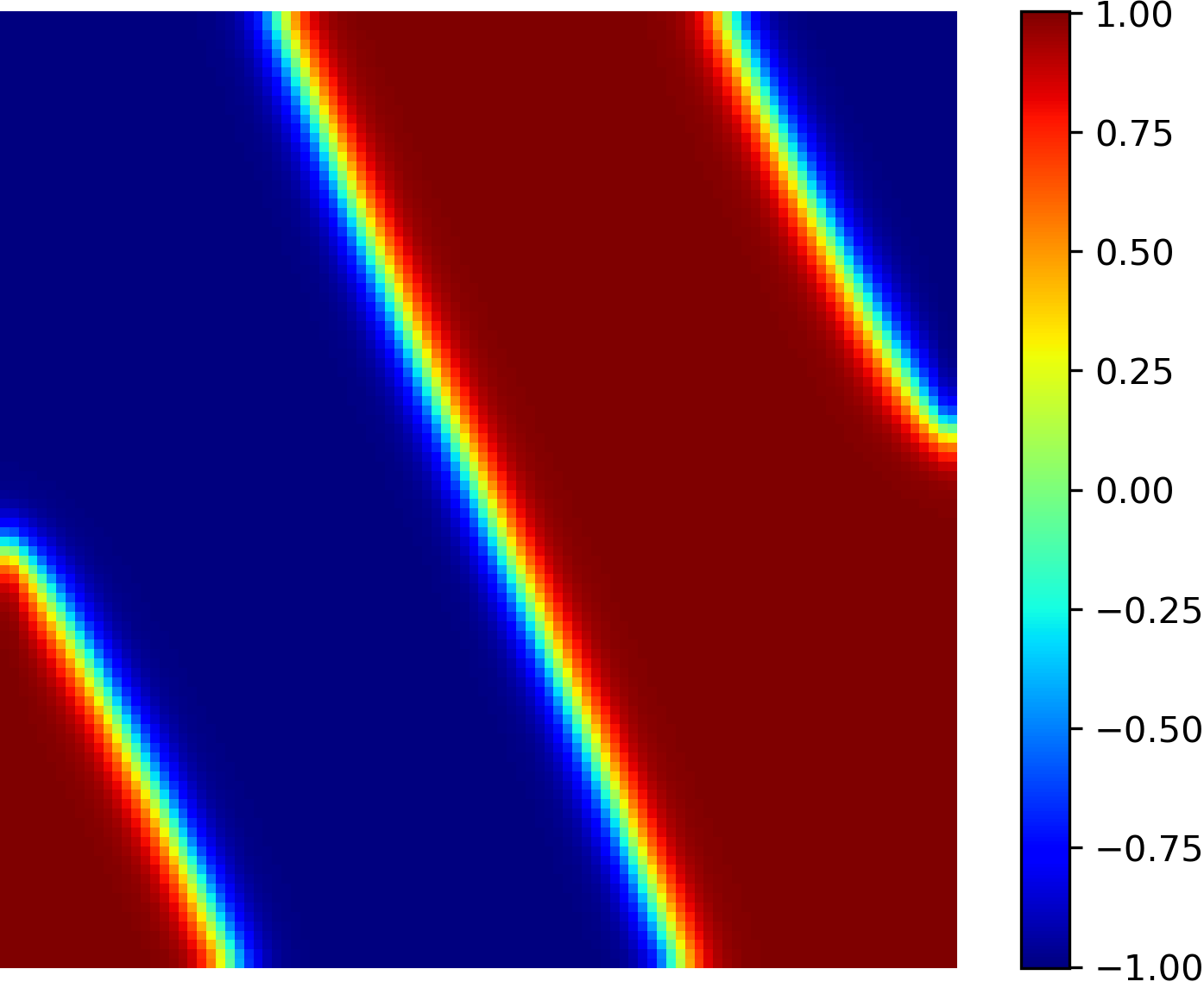}%{figure3/2update4_01_beta_0_beta2_0_dt_0.0005_nu_0.002/5000_phi.jpg}
        %\caption{t = 2.5, $\beta = 0$}
    \end{subfigure}

    \caption{Case 3: Snapshots of the numerical approximation are taken at $T=0.015$, $0.04$, $0.20$, $0.5$, and $2.5$ with different $\beta$. Left: $\beta=1$; Middle: $\beta=0.1$; Right: $\beta=0$.}
    \label{5.10}
\end{figure}

\textbf{Case 4:} We consider a rectangle-shaped droplet $\Omega_0$,  as shown in Figure \ref{5.11}. The phase inside the droplet is set to be $1$ and outside the droplet to be $-1$,
\begin{equation}
    \phi_0(x,y) = \left\{
        \begin{array}{ll}
        1, & (x,y)\in \Omega_0= [0.3,0.7]\times[0,0.5] , \\[8pt]
        0, &  (x,y)\in\Omega\cup\Gamma\backslash\Omega_0. 
        \end{array}
        \right.
\end{equation}
In this test we also consider the effect of the parameter $\beta$ on the system. In Figure \ref{5.13}, we can see that the rectangle-shaped droplet is gradually transforming into a circular shape. Moreover we find that the droplet changes its shape more lowly when $\beta$ is larger. Meanwhile  Figure \ref{5.12} indicates that the discreate  energy  is decreasing more rapidly as the value of $\beta$ becomes smaller. We also observe that the mass conservation in the bulk and on the boundary is maintaining for this case.

\begin{figure}[!htbp]
    \centering
    \includegraphics[width=0.4\textwidth]{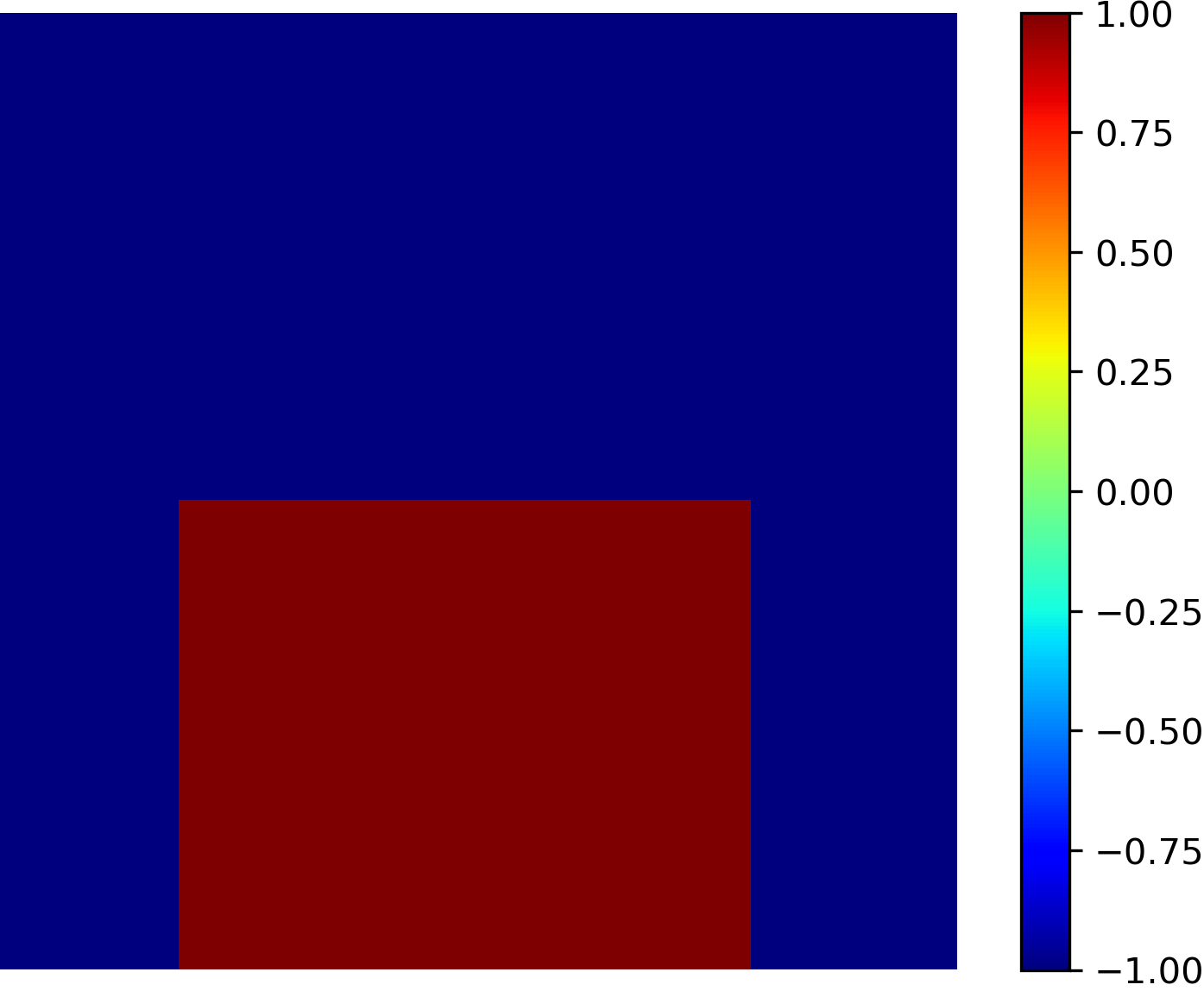}
    \caption{The initial data of Case 4.}

    \label{5.11}

\end{figure}

\begin{figure}[!htbp]
    \centering
    \scalebox{0.85}{
        \begin{minipage}{\textwidth}
            \begin{subfigure}[b]{0.49\textwidth}
                \includegraphics[width=\textwidth]{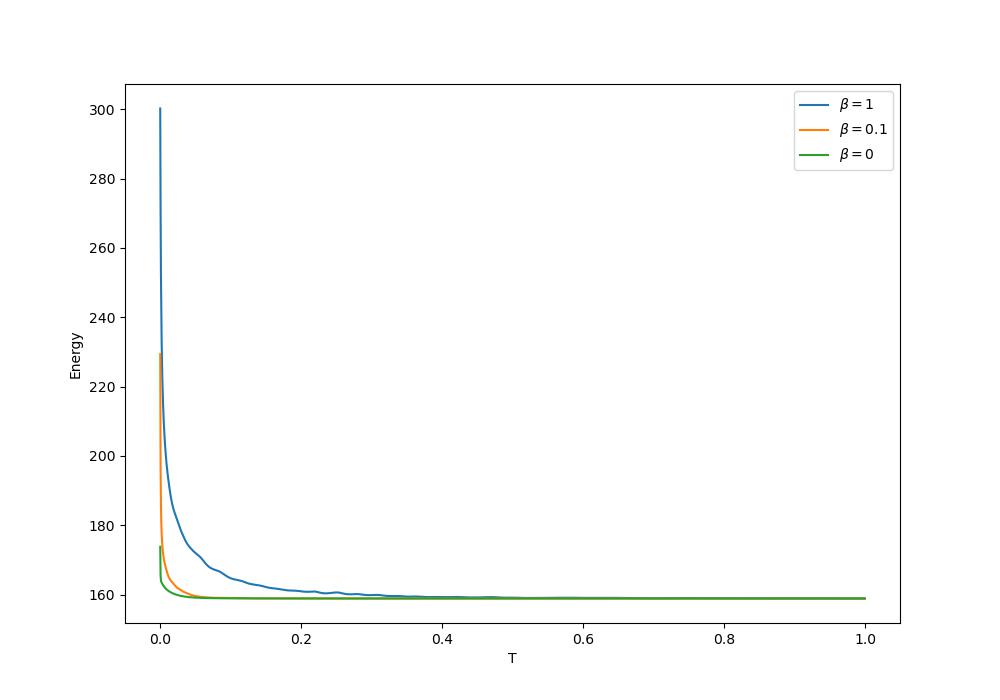}
                \caption{Energy curves with different $\beta$.}
            \end{subfigure}
            \hfill
            \begin{subfigure}[b]{0.49\textwidth}
                \includegraphics[width=\textwidth]{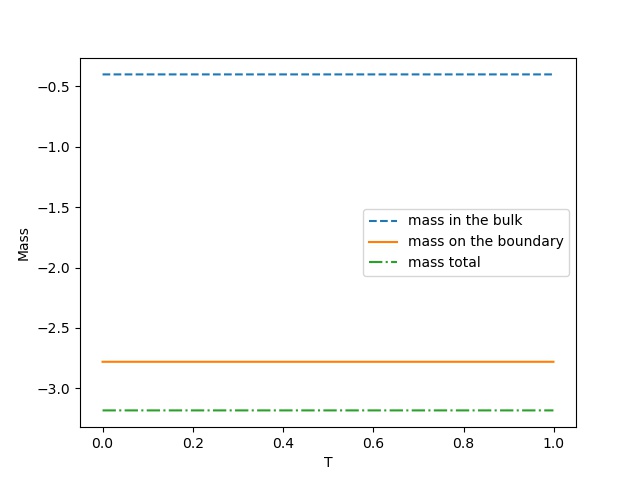}
                \caption{The mass with $\beta = 1$.}
            \end{subfigure}
            \hfill
            \begin{subfigure}[b]{0.49\textwidth}
                \includegraphics[width=\textwidth]{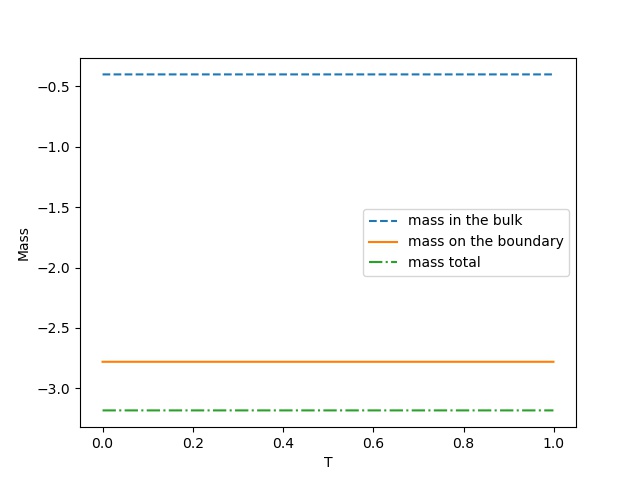}%{figure4/update4_01_beta_0.1_beta2_0.1_dt_0.0005/mass_combined_beta_0.1_beta2_0.1.jpg}
                \caption{The mass with $\beta = 0.1$.}
            \end{subfigure}
            \hfill
            \begin{subfigure}[b]{0.49\textwidth}
                \includegraphics[width=\textwidth]{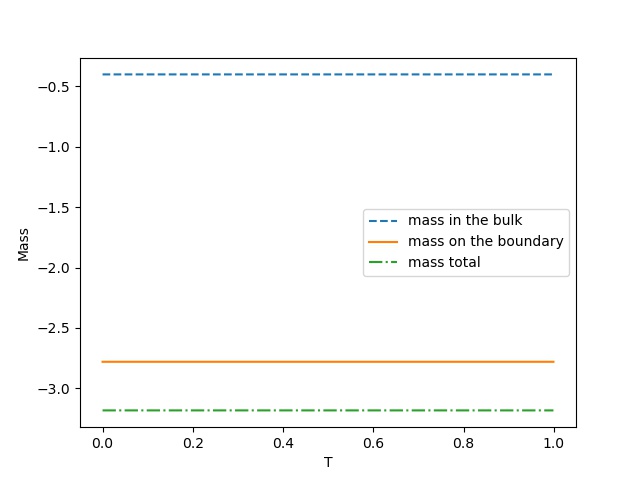}%{figure4/update4_01_beta_0_beta2_0_dt_0.0005/mass_combined_beta_0_beta2_0.jpg}
                \caption{The mass with $\beta = 0$.}
            \end{subfigure}
        \end{minipage}
    }
    \caption{The energy evolution and the mass evolutions of Case 4.}

    \label{5.12}

\end{figure}

\begin{figure}[!htbp]
    \centering
    \begin{subfigure}[t]{0.3\textwidth}
        \includegraphics[scale=0.4]{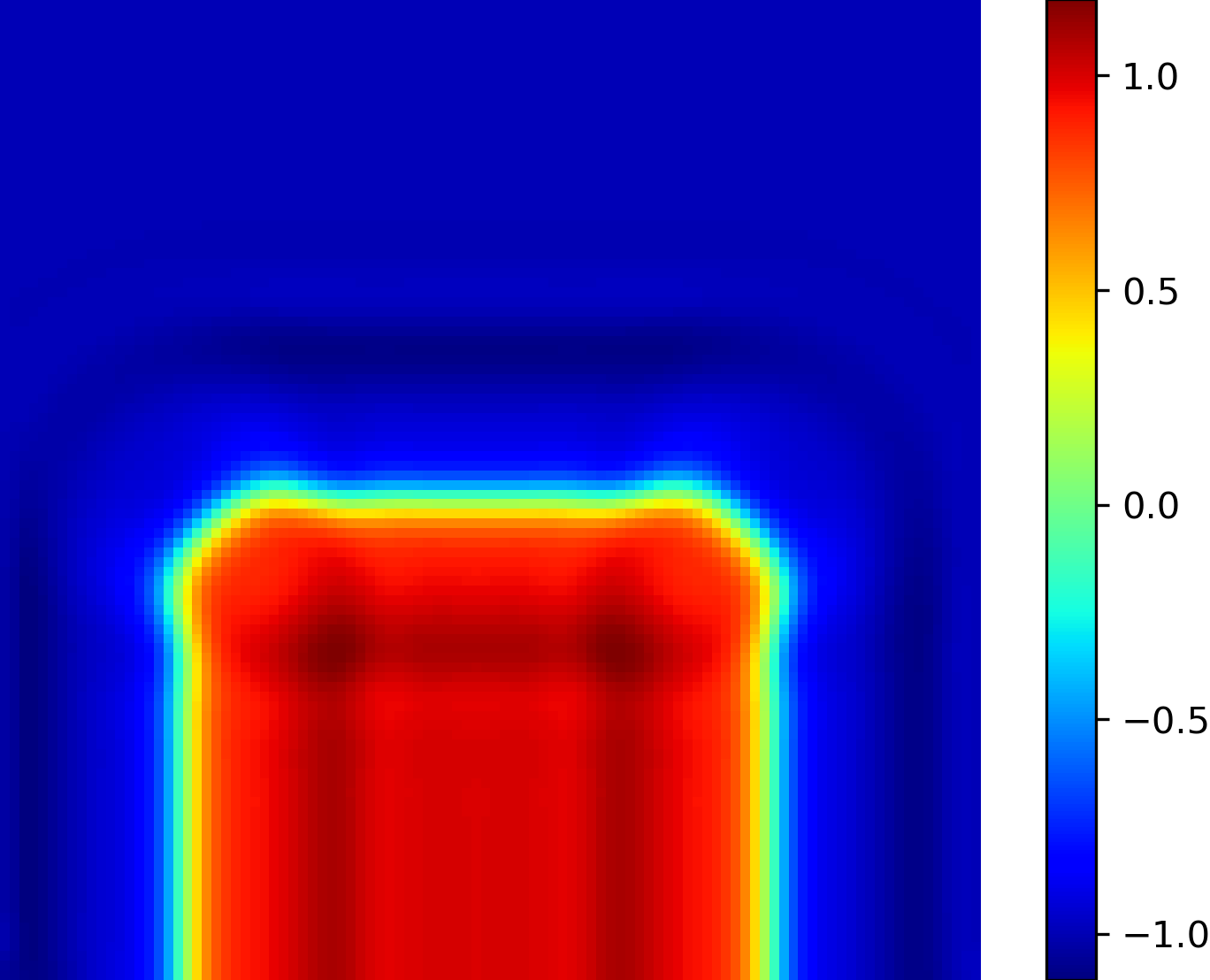}
        % \caption{t = 0.015, $\beta = 1$}
    \end{subfigure}
    \hfill
    \begin{subfigure}[t]{0.3\textwidth}
        \includegraphics[scale=0.4]{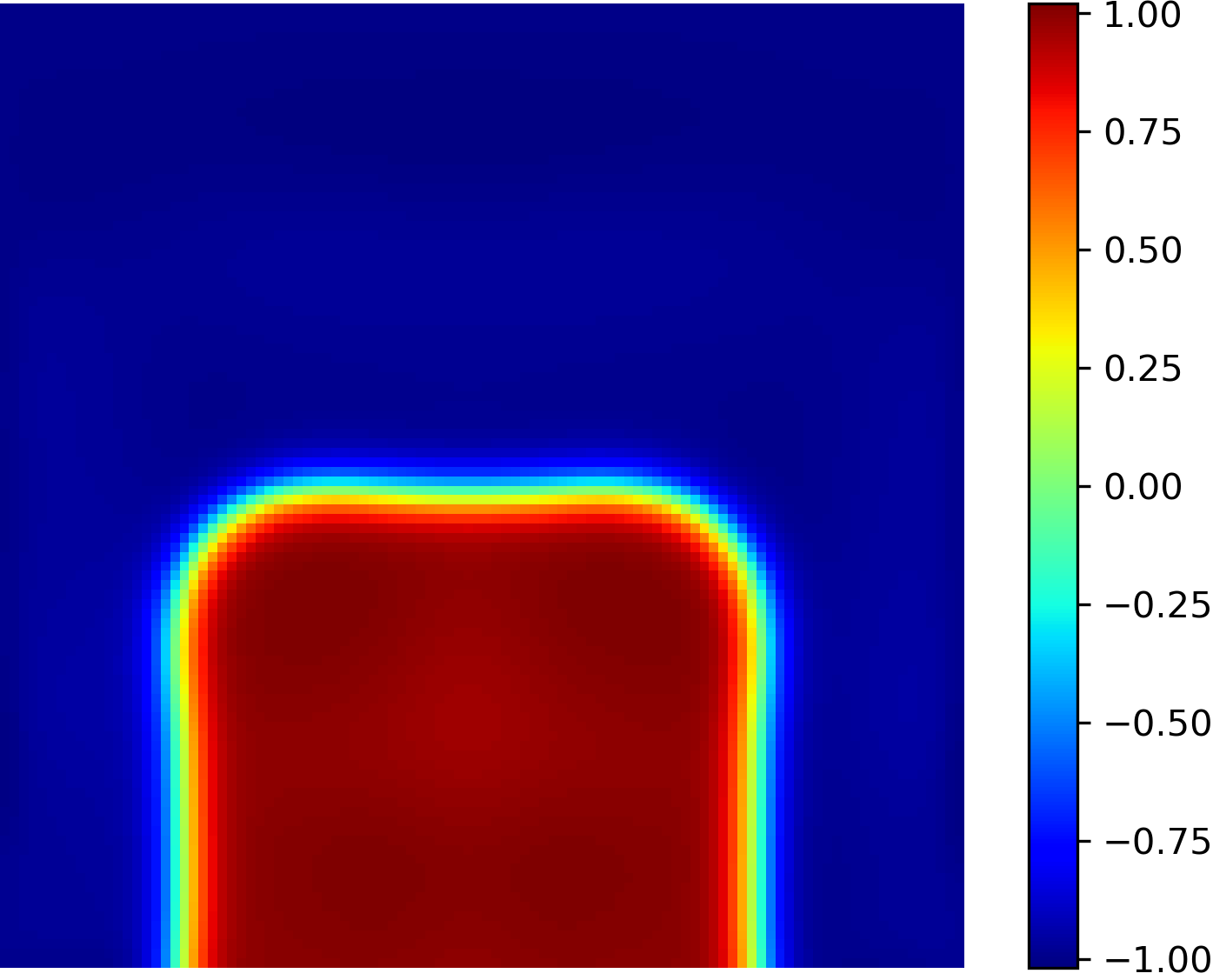}%{figure4/update4_01_beta_0.1_beta2_0.1_dt_0.0005/30_phi.jpg}
        % \caption{t = 0.015, $\beta = 0.1$}
    \end{subfigure}
    \hfill
    \begin{subfigure}[t]{0.3\textwidth}
        \includegraphics[scale=0.4]{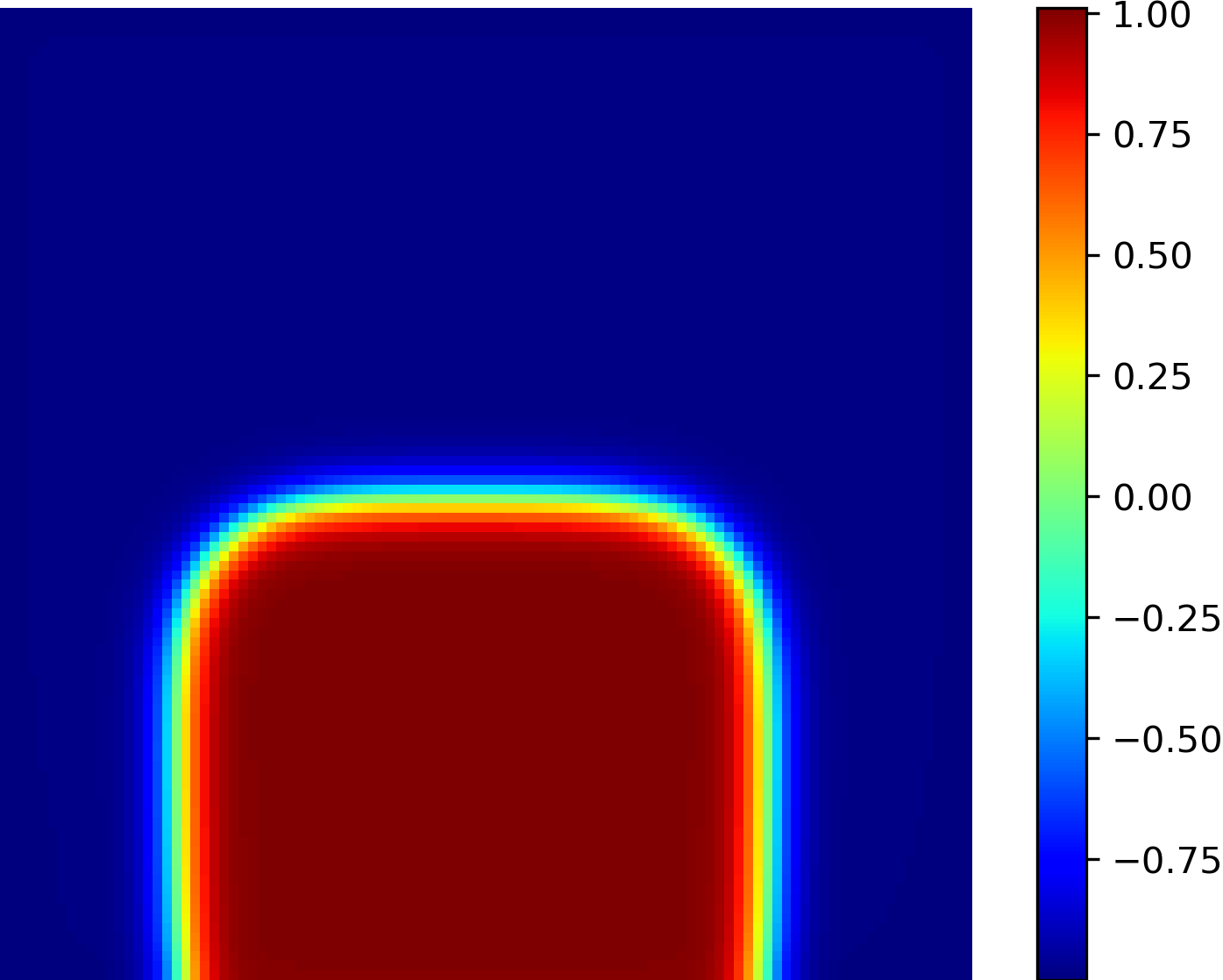}%{figure4/update4_01_beta_0_beta2_0_dt_0.0005/30_phi.jpg}
        % \caption{t = 0.015, $\beta = 0$}
    \end{subfigure}

    \begin{subfigure}[t]{0.3\textwidth}
        \includegraphics[scale=0.4]{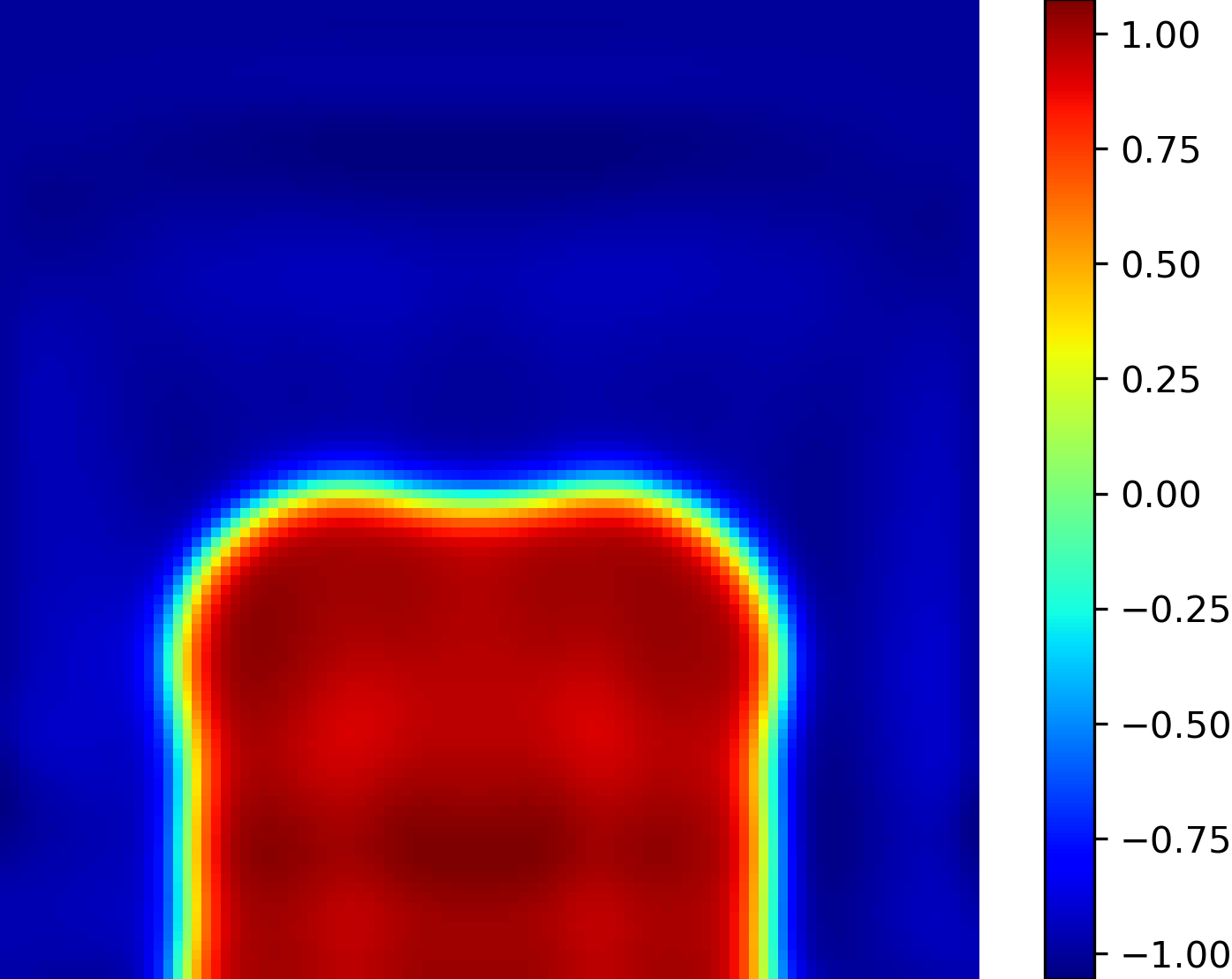}%{figure4/update4_01_beta_1_beta2_1_dt_0.0005/80_phi.jpg}
        % \caption{t = 0.04, $\beta = 1$}
    \end{subfigure}
    \hfill
    \begin{subfigure}[t]{0.3\textwidth}
        \includegraphics[scale=0.4]{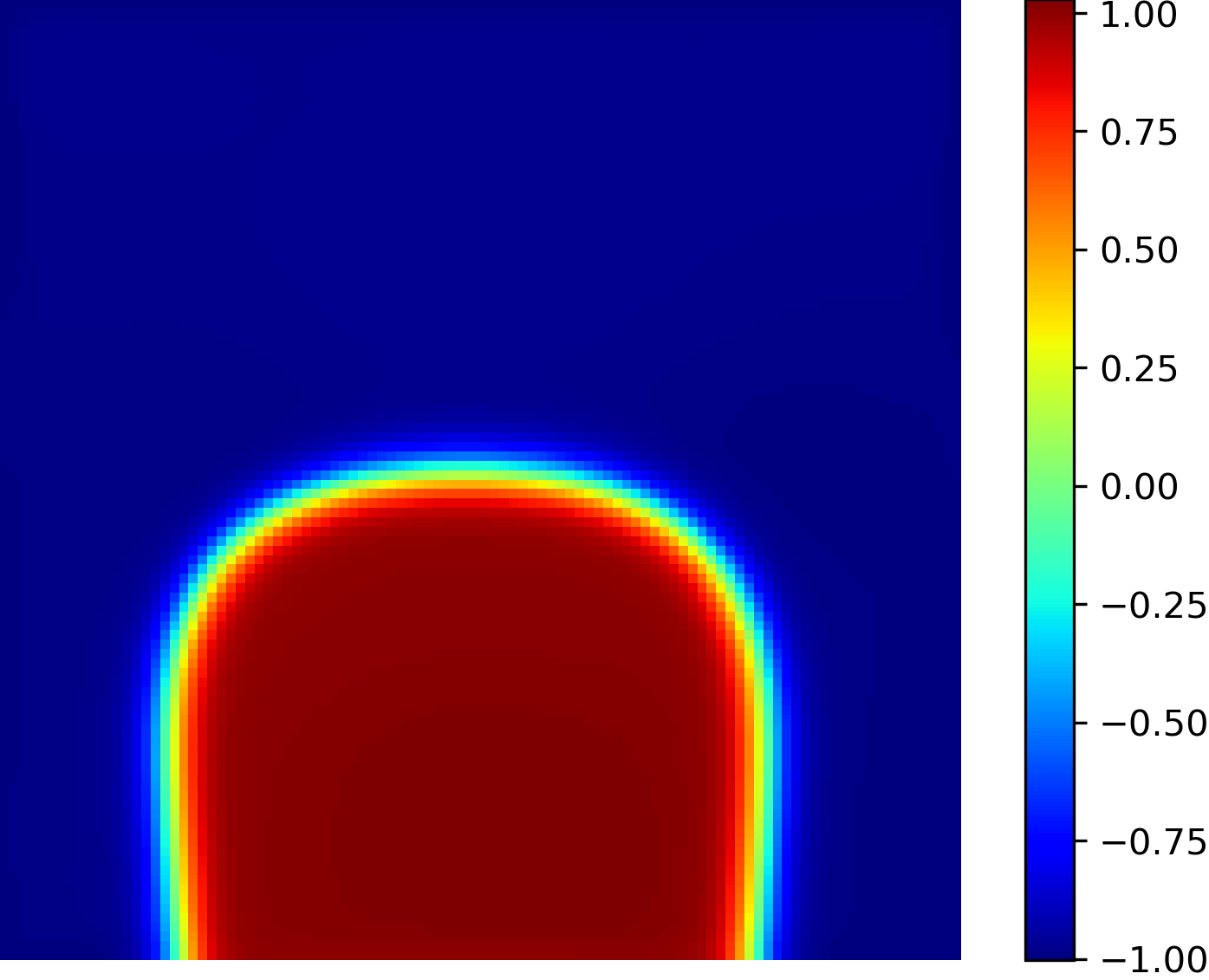}%{figure4/update4_01_beta_0.1_beta2_0.1_dt_0.0005/80_phi.jpg}
        % \caption{t = 0.04, $\beta = 0.1$}
    \end{subfigure}
    \hfill
    \begin{subfigure}[t]{0.3\textwidth}
        \includegraphics[scale=0.4]{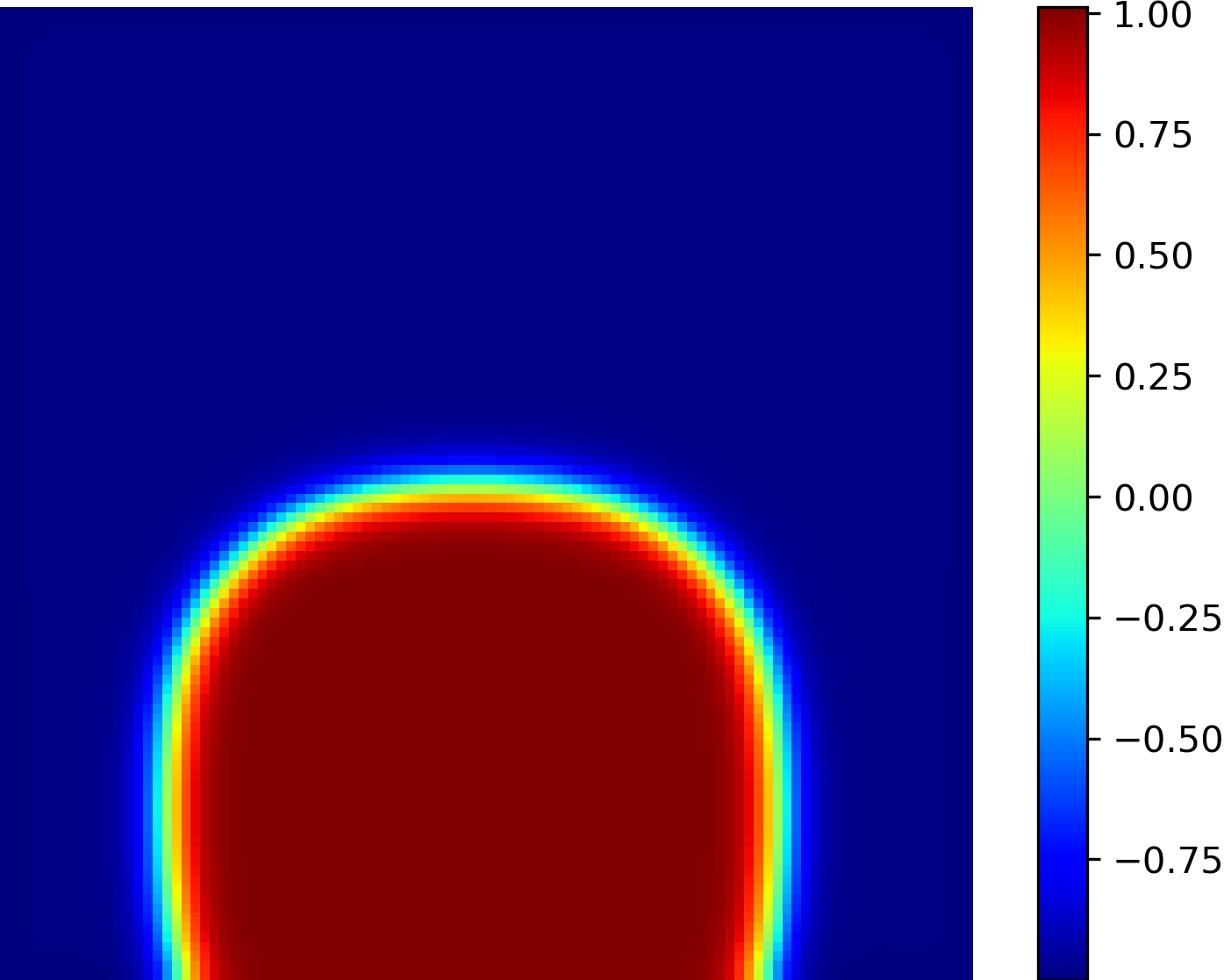}%{figure4/update4_01_beta_0_beta2_0_dt_0.0005/80_phi.jpg}
        % \caption{t = 0.04, $\beta = 0$}
    \end{subfigure}

    \begin{subfigure}[t]{0.3\textwidth}
        \includegraphics[scale=0.4]{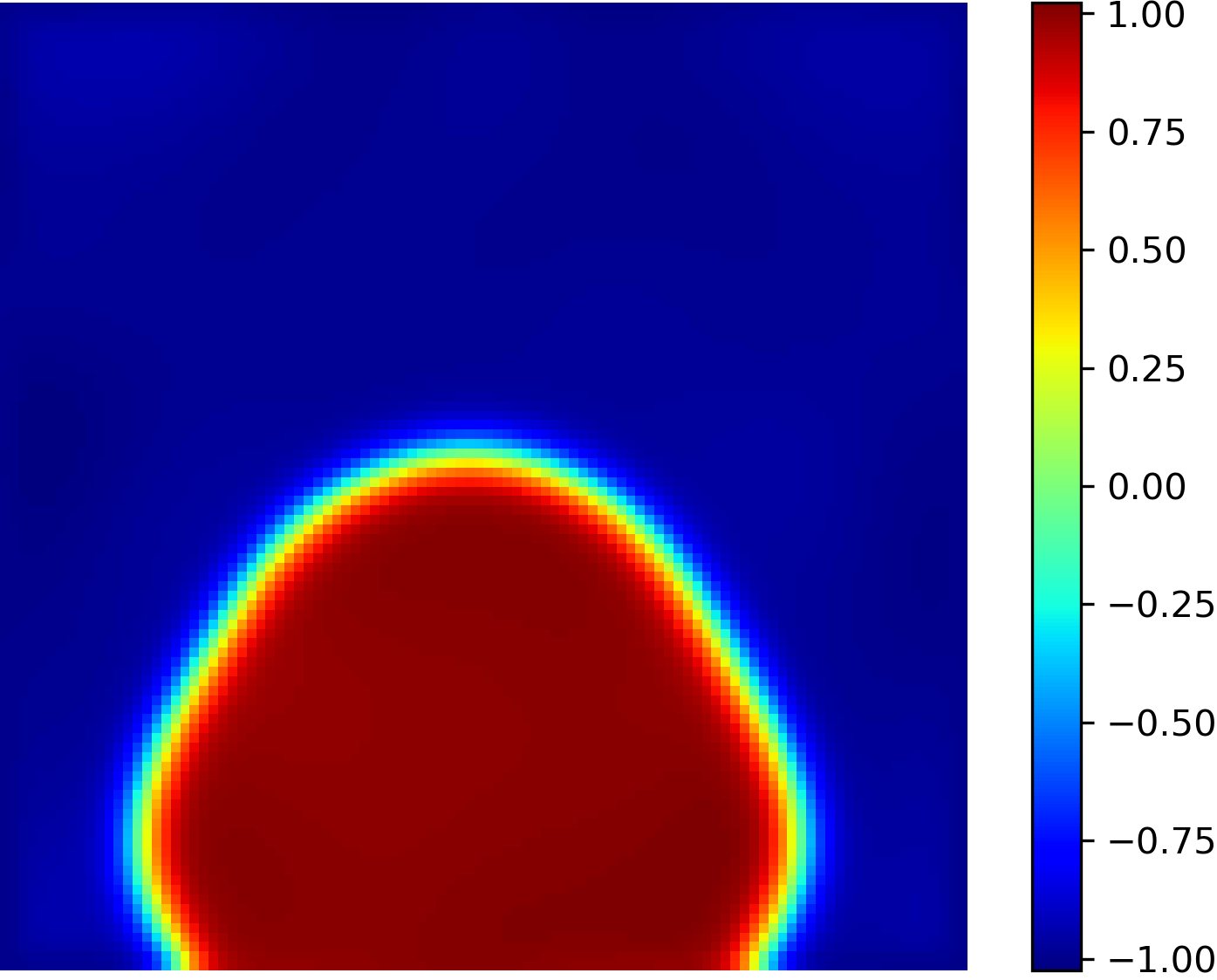}%{figure4/update4_01_beta_1_beta2_1_dt_0.0005/400_phi.jpg}
        % \caption{t = 0.2, $\beta = 1$}
    \end{subfigure}
    \hfill
    \begin{subfigure}[t]{0.3\textwidth}
        \includegraphics[scale=0.4]{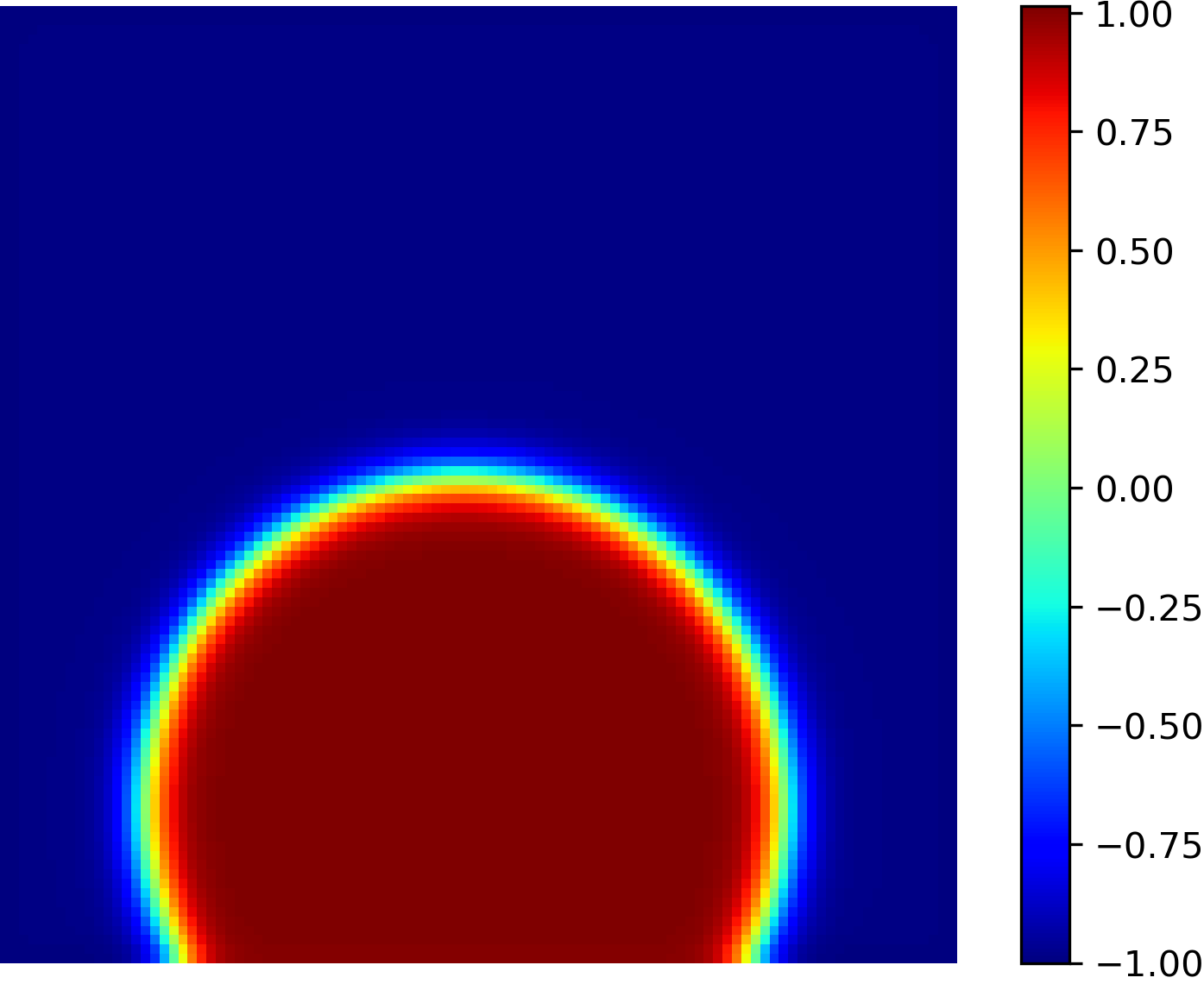}%{figure4/update4_01_beta_0.1_beta2_0.1_dt_0.0005/400_phi.jpg}
        % \caption{t = 0.2, $\beta = 0.1$}
    \end{subfigure}
    \hfill
    \begin{subfigure}[t]{0.3\textwidth}
        \includegraphics[scale=0.4]{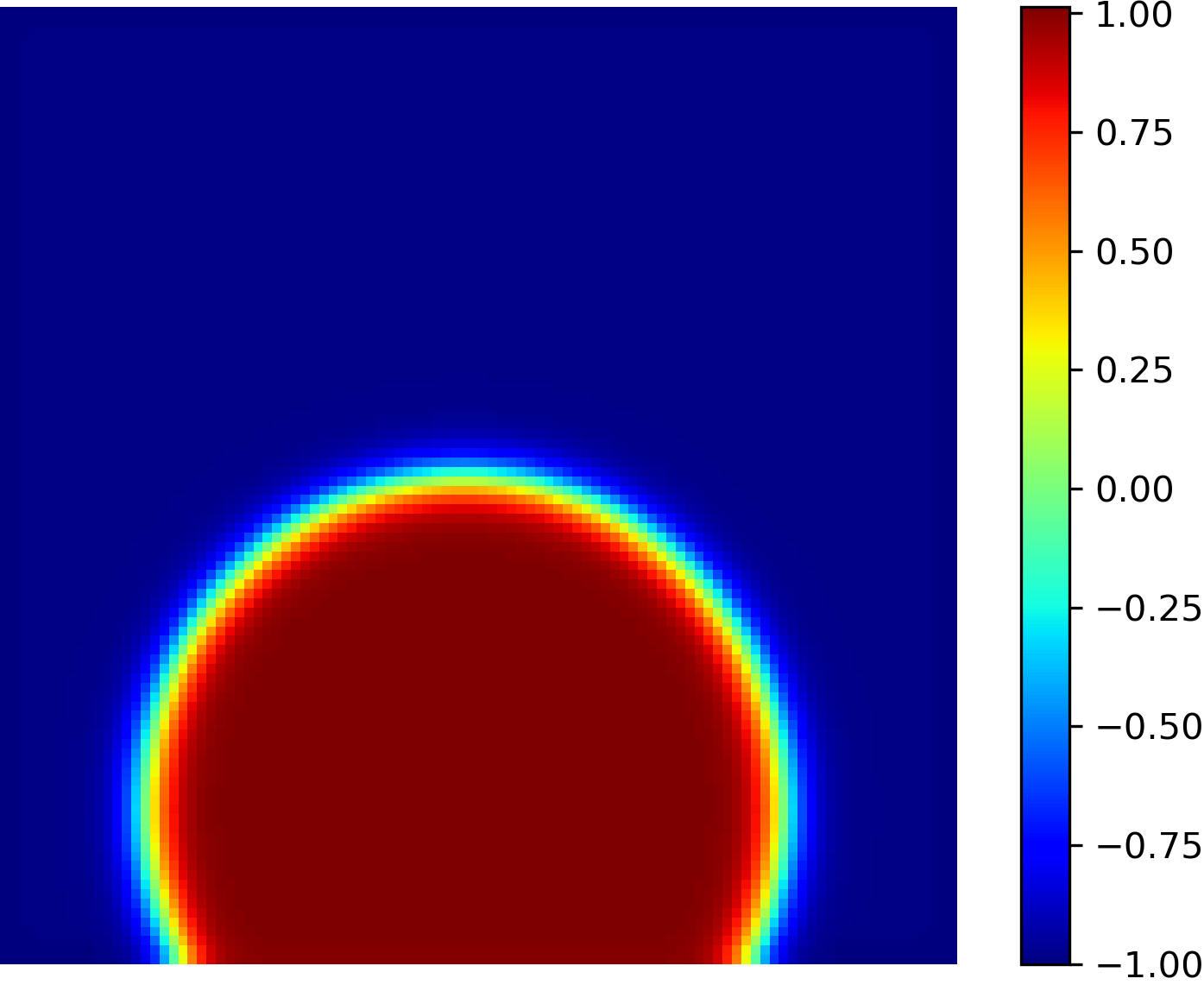}%{figure4/update4_01_beta_0_beta2_0_dt_0.0005/400_phi.jpg}
        % \caption{t = 0.2, $\beta = 0$}
    \end{subfigure}

    \begin{subfigure}[t]{0.3\textwidth}
        \includegraphics[scale=0.4]{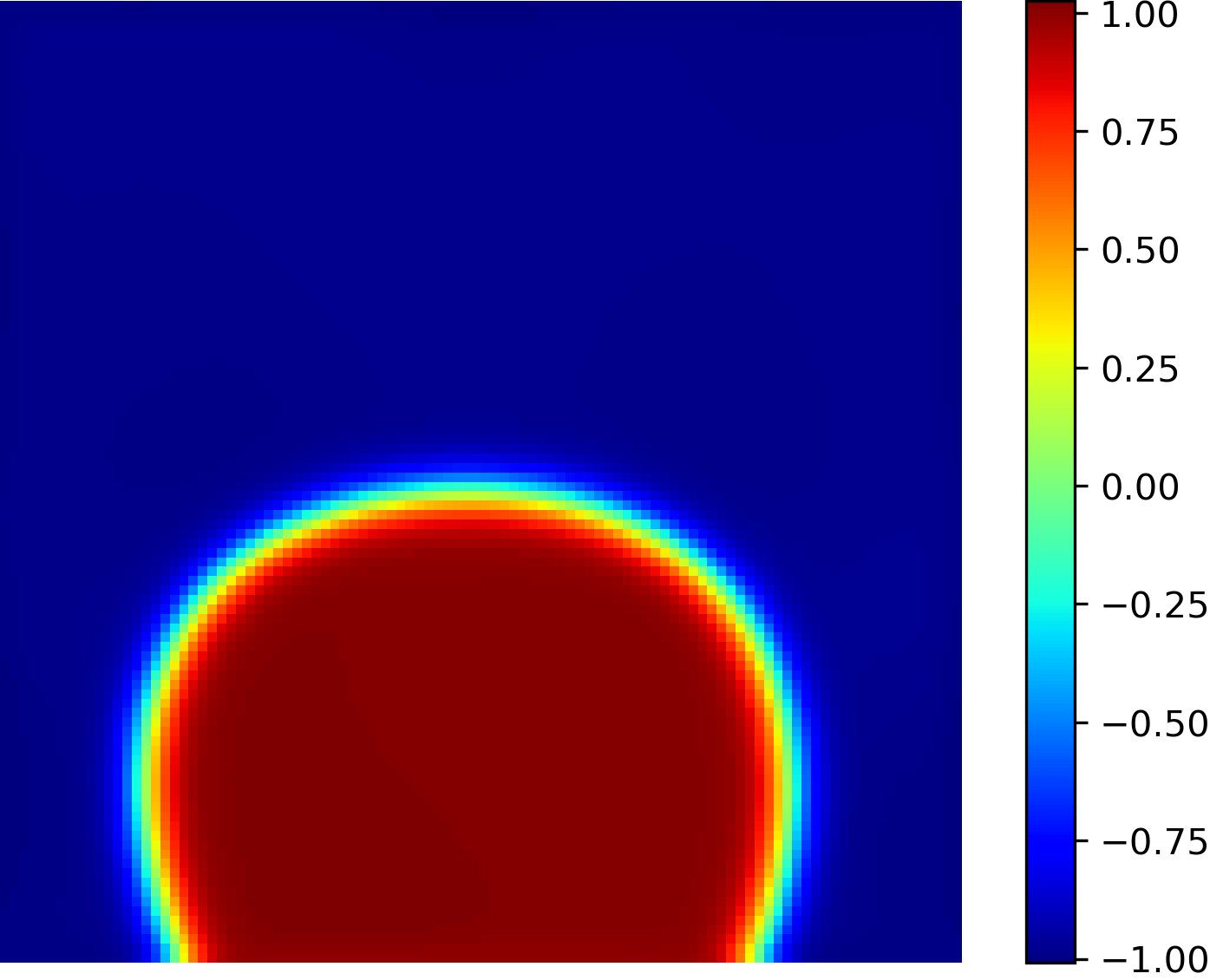}%{figure4/update4_01_beta_1_beta2_1_dt_0.0005/800_phi.jpg}
        % \caption{t = 0.4, $\beta = 1$}
    \end{subfigure}
    \hfill
    \begin{subfigure}[t]{0.3\textwidth}
        \includegraphics[scale=0.4]{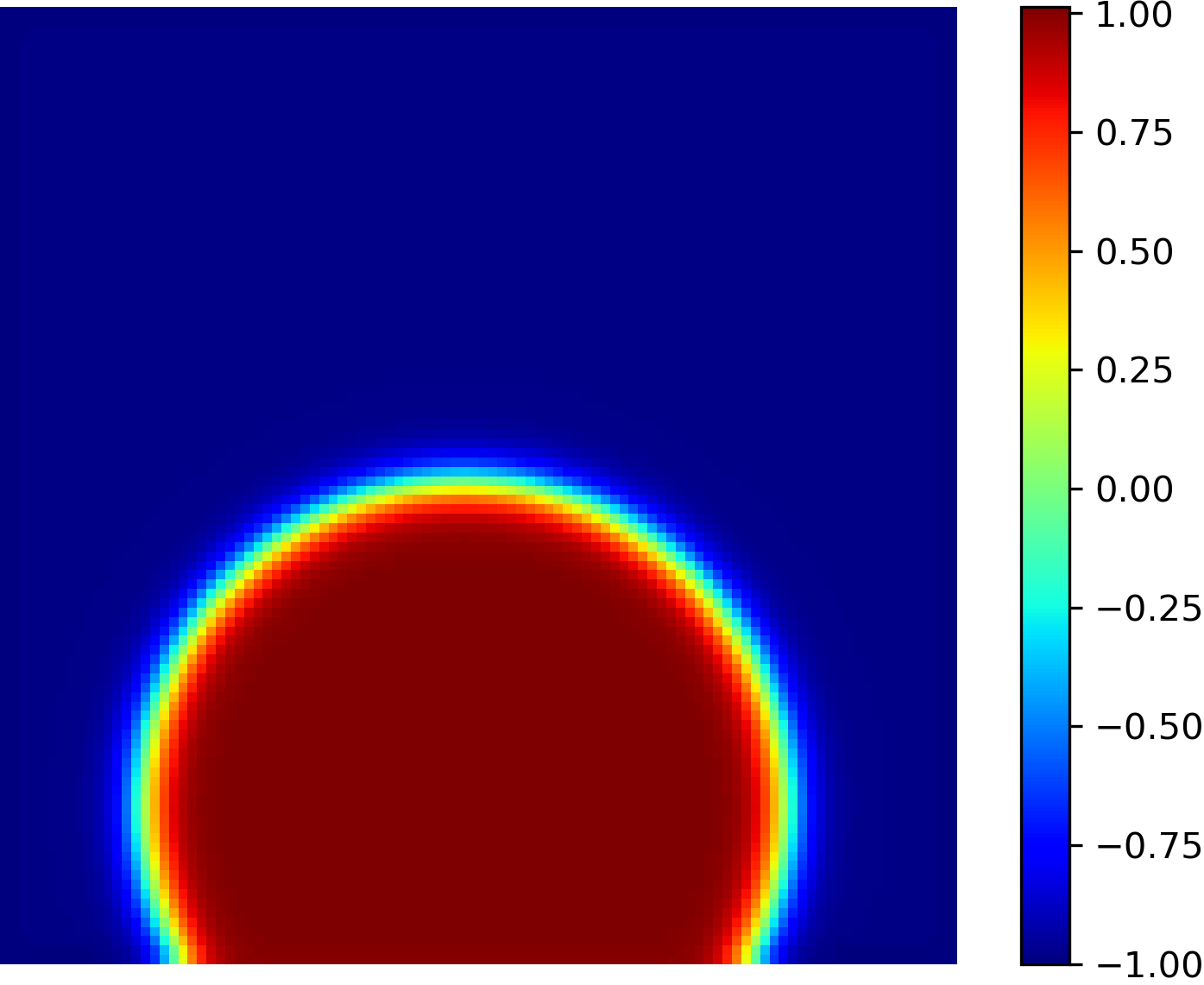}%{figure4/update4_01_beta_0.1_beta2_0.1_dt_0.0005/800_phi.jpg}
        % \caption{t = 0.4, $\beta = 0.1$}
    \end{subfigure}
    \hfill
    \begin{subfigure}[t]{0.3\textwidth}
        \includegraphics[scale=0.4]{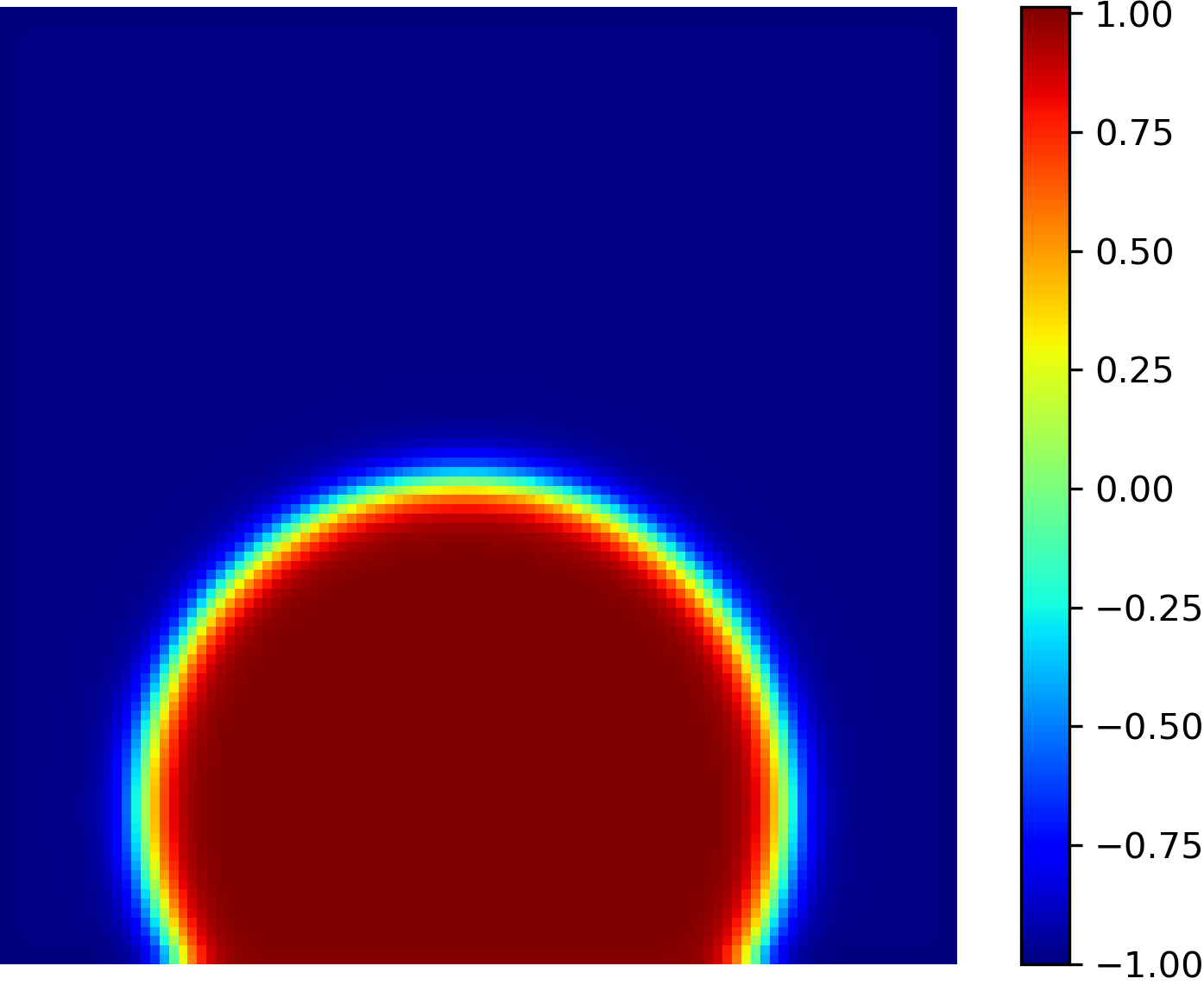}%{figure4/update4_01_beta_0_beta2_0_dt_0.0005/800_phi.jpg}
        % \caption{t = 0.4, $\beta = 0$}
    \end{subfigure}

    \begin{subfigure}[t]{0.3\textwidth}
        \includegraphics[scale=0.4]{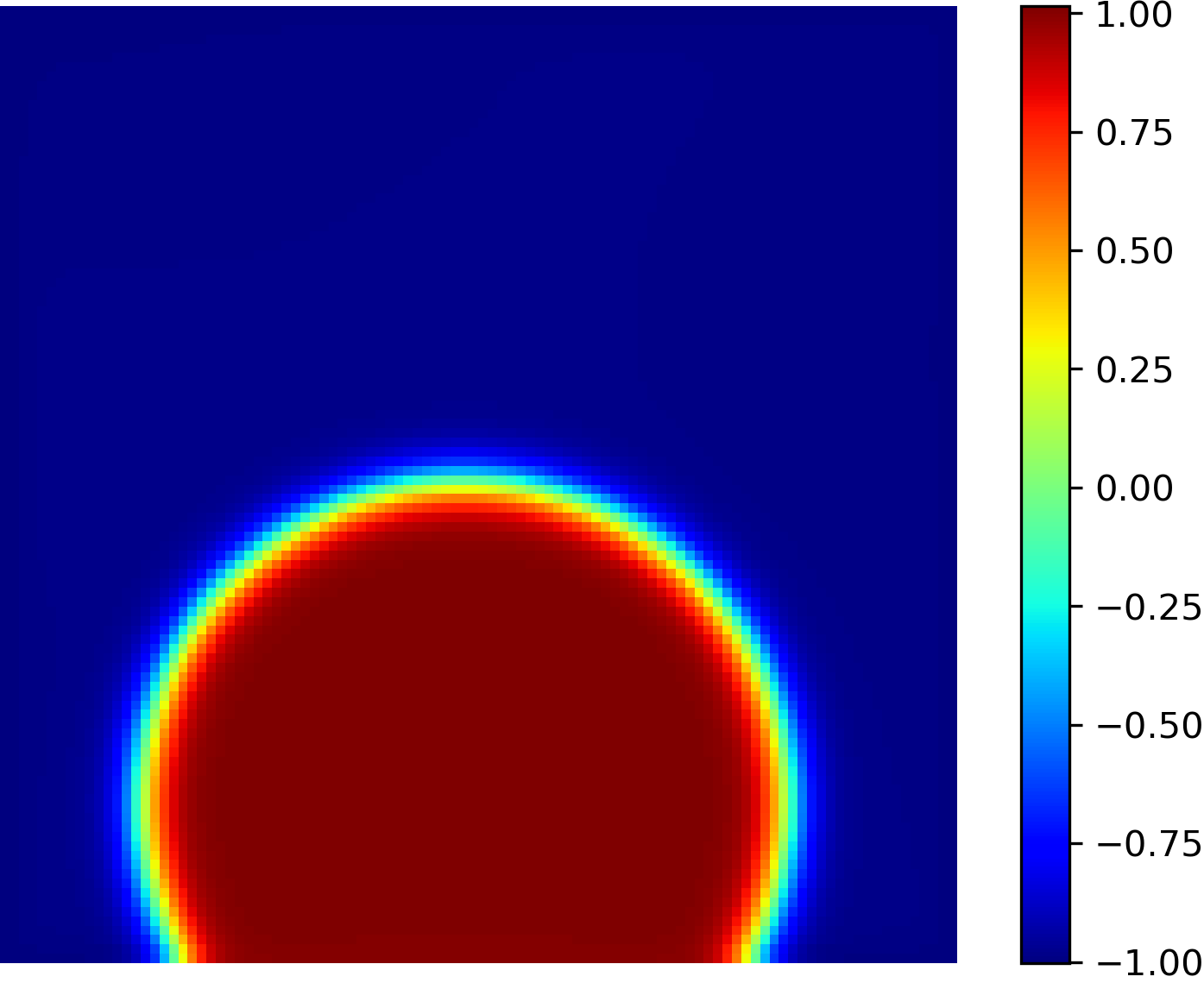}%{figure4/update4_01_beta_1_beta2_1_dt_0.0005/1800_phi.jpg}
        % \caption{t = 0.9, $\beta = 1$}
    \end{subfigure}
    \hfill
    \begin{subfigure}[t]{0.3\textwidth}
        \includegraphics[scale=0.4]{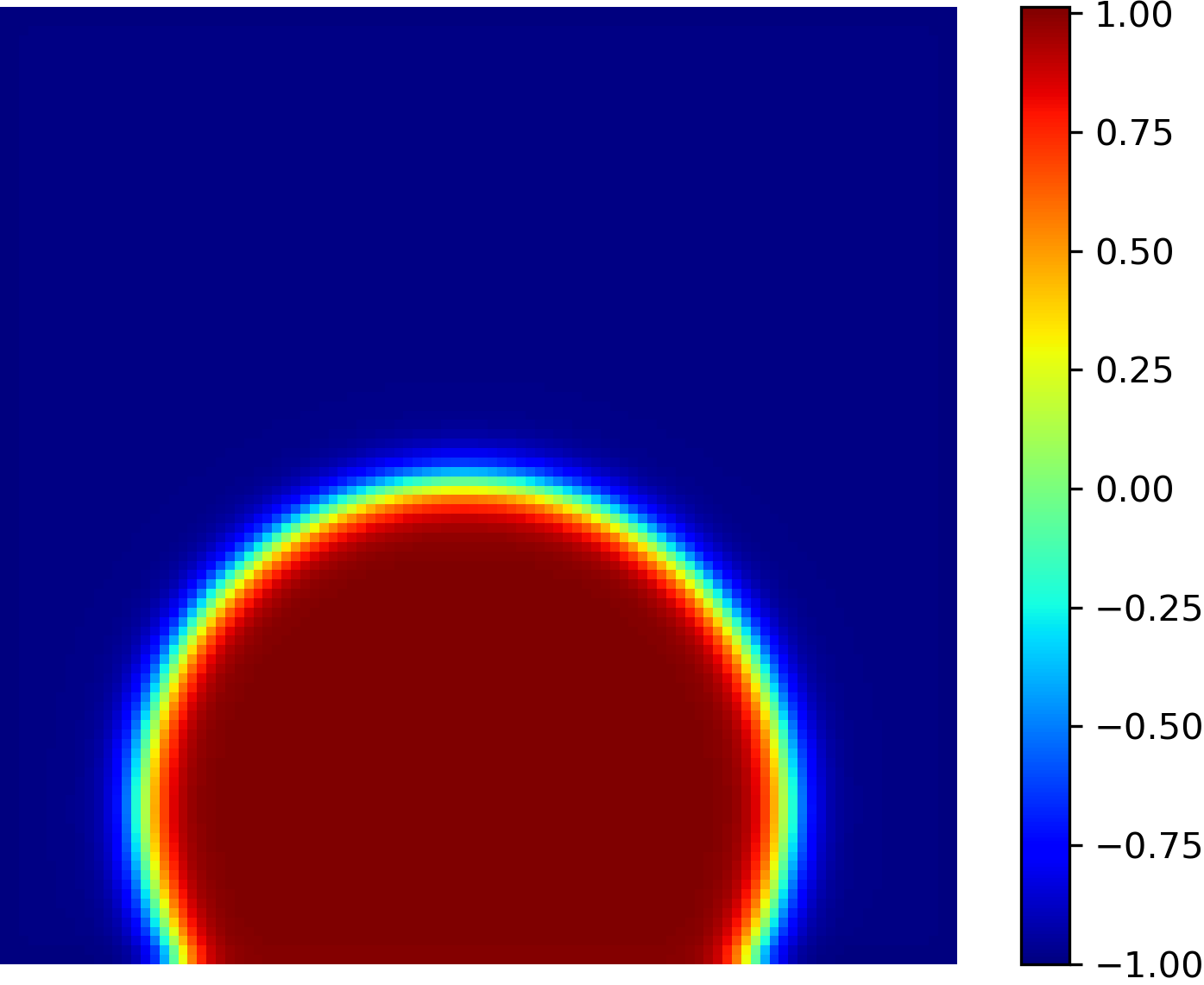}%{figure4/update4_01_beta_0.1_beta2_0.1_dt_0.0005/1800_phi.jpg}
        % \caption{t = 0.9, $\beta = 0.1$}
    \end{subfigure}
    \hfill
    \begin{subfigure}[t]{0.3\textwidth}
        \includegraphics[scale=0.4]{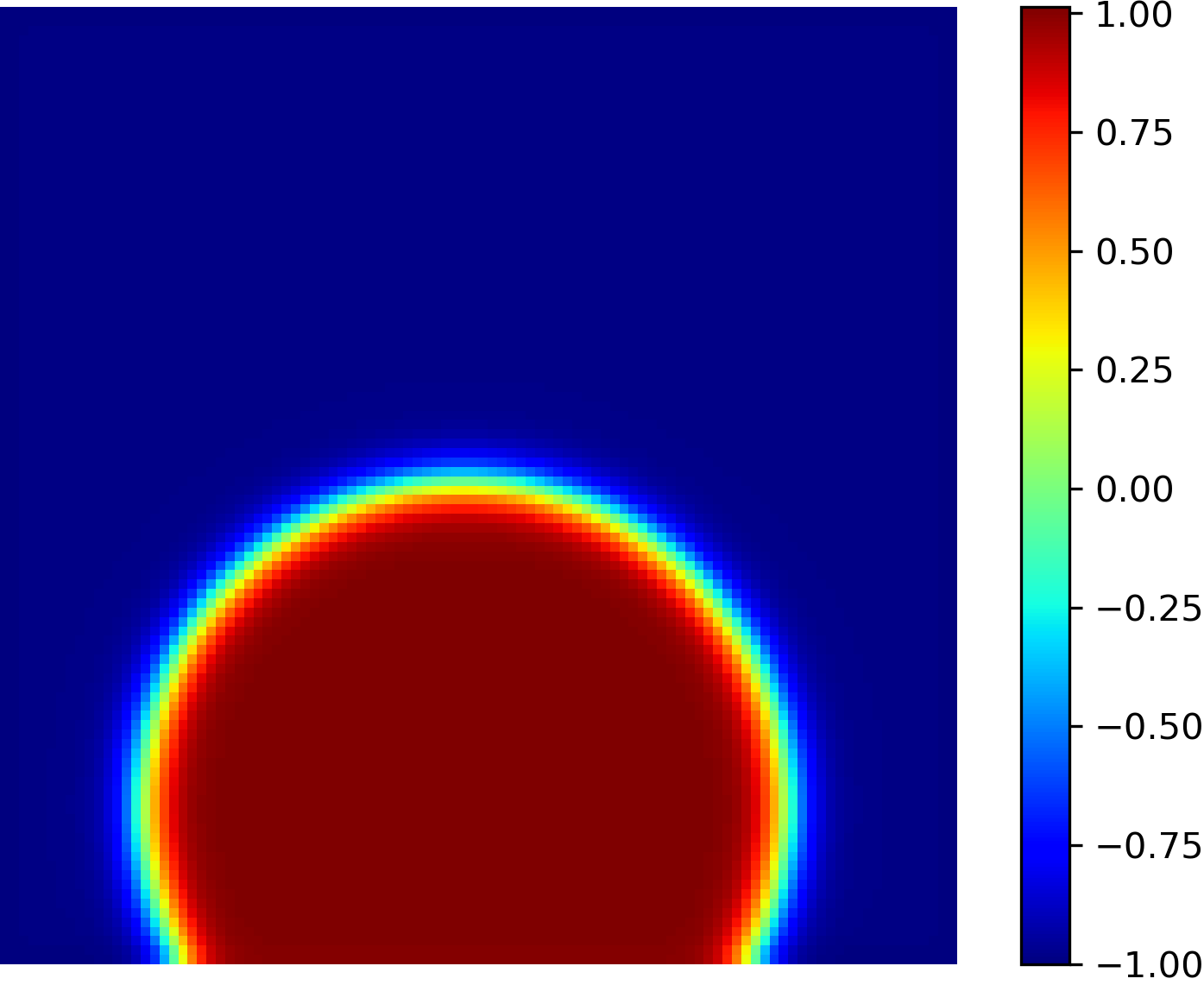}%{figure4/update4_01_beta_0_beta2_0_dt_0.0005/1800_phi.jpg}
        % \caption{t = 0.9, $\beta = 0$}
    \end{subfigure}

    \caption{Case 4: Snapshots of the numerical approximation are taken at $T= 0.015$, $0.04$, $0.2$, $0.4$, and $0.9$ with different $\beta$. Left: $\beta=1$; Middle: $\beta=0.1$; Right: $\beta=0$.
    }
    \label{5.13}
\end{figure}

\section{Conclusion}
\label{sec6}
In this paper we 
 have investigated the hyperbolic Cahn-Hilliard equation with the hyperbolic Cahn-Hilliard type dynamic boundary condition. By adding two stabilizing terms, we have designed a linear, first-order in time and energy stable scheme for the system. Meanwhile, we have also proved that the scheme is of first order in time by the error analysis. Finally there are enough numerical cases to show the temporal convergence, the mass conservation in the bulk and on the boundary, and the energy stability of the scheme. We also find that the hyperbolic terms can help the system to delay reaching the steady state. 

\section*{Acknowledgement}
The authors acknowledge the support of NSFC, China 12001055.

\section*{Statements.}
We state that the datasets generated during and/or analysed during the current study are available from the corresponding author on reasonable request.
We declare that we have no conflict of interest in the submission of this manuscript.

\bibliographystyle{elsarticle-num}  
\bibliography{references} 

\begin{thebibliography}{10}
\expandafter\ifx\csname url\endcsname\relax
  \def\url#1{\texttt{#1}}\fi
\expandafter\ifx\csname urlprefix\endcsname\relax\def\urlprefix{URL }\fi
\expandafter\ifx\csname href\endcsname\relax
  \def\href#1#2{#2} \def\path#1{#1}\fi

\bibitem{Cahn1958}
J.~W. Cahn, J.~E. Hilliard, Free energy of a nonuniform system. i. interfacial
  free energy, Journal of Chemical Physics 28 (1958) 258--267.

\bibitem{Cherfils2011}
L.~Cherfils, A.~Miranville, S.~Zelik, The cahn-hilliard equation with
  logarithmic potentials, Milan Journal of Mathematics 79 (2011) 561--596.

\bibitem{Novick-Cohen2008}
A.~Novick-Cohen, The cahn-hilliard equation: mathematical and modeling
  perspectives, Advances in mathematical sciences and applications 8 (1998)
  965--985.

\bibitem{Novick-Cohen2008b}
A.~Novick-Cohen, The cahn--hilliard equation, Handbook of differential
  equations: evolutionary equations 4 (2008) 201--228.

\bibitem{Bates1993}
P.~W. Bates, P.~C. Fife, The dynamics of nucleation for the cahn--hilliard
  equation, SIAM Journal on Applied Mathematics 53~(4) (1993) 990--1008.

\bibitem{Abels2012}
H.~Abels, H.~Garcke, G.~Gr{\"u}n, Thermodynamically consistent, frame
  indifferent diffuse interface models for incompressible two-phase flows with
  different densities, Mathematical Models and Methods in Applied Sciences
  22~(03) (2012) 1150013.

\bibitem{Anderson1998}
D.~M. Anderson, G.~B. McFadden, A.~A. Wheeler, Diffuse-interface methods in
  fluid mechanics, Annual review of fluid mechanics 30~(1) (1998) 139--165.

\bibitem{Cahn1961}
J.~W. Cahn, On spinodal decomposition, Acta metallurgica 9~(9) (1961) 795--801.

\bibitem{Du2020}
Q.~Du, X.~Feng, The phase field method for geometric moving interfaces and
  their numerical approximations, Handbook of numerical analysis 21 (2020)
  425--508.

\bibitem{Kim2016}
J.~Kim, S.~Lee, Y.~Choi, S.-M. Lee, D.~Jeong, Basic principles and practical
  applications of the cahn--hilliard equation, Mathematical Problems in
  Engineering 2016~(1) (2016) 9532608.

\bibitem{Bertozzi2007}
A.~L. Bertozzi, S.~Esedoglu, A.~Gillette, Inpainting of binary images using the
  cahn--hilliard equation, IEEE Transactions on image processing 16~(1) (2006)
  285--291.

\bibitem{Bertozzi2007b}
A.~Bertozzi, S.~Esedoglu, A.~Gillette, Analysis of a two-scale cahn--hilliard
  model for binary image inpainting, Multiscale Modeling \& Simulation 6~(3)
  (2007) 913--936.

\bibitem{Garcke2016}
H.~Garcke, K.~F. Lam, E.~Sitka, V.~Styles, A cahn--hilliard--darcy model for
  tumour growth with chemotaxis and active transport, Mathematical Models and
  Methods in Applied Sciences 26~(06) (2016) 1095--1148.

\bibitem{Oden2010}
J.~T. Oden, A.~Hawkins, S.~Prudhomme, General diffuse-interface theories and an
  approach to predictive tumor growth modeling, Mathematical Models and Methods
  in Applied Sciences 20~(03) (2010) 477--517.

\bibitem{Oden2013}
J.~T. Oden, E.~E. Prudencio, A.~Hawkins-Daarud, Selection and assessment of
  phenomenological models of tumor growth, Mathematical Models and Methods in
  Applied Sciences 23~(07) (2013) 1309--1338.

\bibitem{Gurtin1996}
M.~E. Gurtin, D.~Polignone, J.~Vinals, Two-phase binary fluids and immiscible
  fluids described by an order parameter, Mathematical Models and Methods in
  Applied Sciences 6~(06) (1996) 815--831.

\bibitem{Hohenberg1977}
P.~C. Hohenberg, B.~I. Halperin, Theory of dynamic critical phenomena, Reviews
  of Modern Physics 49~(3) (1977) 435.

\bibitem{Jacqmin2000}
D.~Jacqmin, Contact-line dynamics of a diffuse fluid interface, Journal of
  fluid mechanics 402 (2000) 57--88.

\bibitem{Pruss2006}
J.~Pr{\"u}ss, M.~Wilke, Maximal $l_p$-regularity and long-time behaviour of the
  non-isothermal cahn-hilliard equation with dynamic boundary conditions,
  Partial Differential Equations and Functional Analysis: The Philippe
  Cl{\'e}ment Festschrift (2006) 209--236.

\bibitem{Liu2019}
C.~Liu, H.~Wu, An energetic variational approach for the
  cahn{\textendash}hilliard equation with dynamic boundary condition: Model
  derivation and mathematical analysis, Arch. Ration. Mech. Anal. 233~(1)
  (2019) 167--247.

\bibitem{Goldstein2011}
G.~R. Goldstein, A.~Miranville, G.~Schimperna, A cahn{\textendash}hilliard
  model in a domain with non-permeable walls, Physica D 240~(8) (2011)
  754--766.

\bibitem{Knopf2021}
P.~Knopf, K.~F. Lam, C.~Liu, S.~Metzger, Phase-field dynamics with transfer of
  materials: The cahn{\textendash}hilliard equation with reaction rate
  dependent dynamic boundary conditions, {ESAIM}: Mathematical Modelling and
  Numerical Analysis 55~(1) (2021) 229--282.

\bibitem{HLT07}
Y.~He, Y.~Liu, T.~Tang, On large time-stepping methods for the cahn-hilliard
  equation, Appl. Numeri. Mathe. 57 (2007) 616--628.

\bibitem{Chen16}
W.~Chen, W.~Feng, Y.~Liu, C.~Wang, S.~Wise, A second order energy stable scheme
  for the cahn-hilliard-hele-shaw equations, Discrete. Cont. Dyn.-B 24 (2019)
  149--182.

\bibitem{Shen12}
J.~Shen, C.~Wang, X.~Wang, S.~Wise, Second-order converx splitting schemes for
  gradient flows with ehrlich-schewoebel type energy: application to thin film
  epitaxy, SIAM J. Numer. Anal. 50 (2012) 105--125.

\bibitem{Wise10}
S.~Wise, Unconditionally stable finite difference, nonliear multigrid
  simulation of the cahn-hilliard-hele-shaw system of equations, J. Sci.
  Comput. 44 (2010) 38--68.

\bibitem{BGG2011}
S.~Badia, F.~Guill\'{e}n-Gonz\'{e}, J.~Guti\'{e}rrez-Santacreu, Finite element
  approximation of nematic liquid crystal flows using a saddle-point structure,
  J. Comput. Phys. 230 (2011) 1686--1706.

\bibitem{GT2013}
F.~Guill\'{e}n-Gonz\'{e}, G.~Tierra, On linear schemes for a cahn-hilliard
  diffuse interface model, J. Comput. Phys. 23 (2013) 140--171.

\bibitem{GT2014}
F.~Guill\'{e}n-Gonz\'{e}, G.~Tierra, Second order schemes and time-step
  adaptivity allen-cahn and cahn-hilliard models, Comput. Math. Appl. 68 (2014)
  821--846.

\bibitem{TG2014}
G.~Tierra, F.~Guill\'{e}n-Gonz\'{e}, Numerical methods for solving the
  cahn-hilliard equation and its applicability to related energy-based models,
  Arch. Computat. Methods. Eng. 22 (2014) 269--289.

\bibitem{Chen17}
R.~Chen, X.~Yang, H.~Zhang, Second order, linear, and unconditionaly energy
  stable schemes for a hydrodynamic model of smectic-a liquid crystals, SIAM J.
  Sci. Comput. 39 (2017) A2808--A2833.

\bibitem{GX19}
R.~Guo, Y.~Xu, Semi-implicit spectral deferred correction method based on the
  invariant energy quadratization approach for phase field problems, Commun.
  Comput. Phys. 26 (2019) 87--113.

\bibitem{Yang16}
X.~Yang, Linear, first and second-order, unconditionaly energy stable numerical
  schems for the phase field model of homopolymer blends, J. Comput. Phys. 327
  (2016) 294--316.

\bibitem{Yang17}
X.~Yang, L.~Ju, Efficient linear schemes with unconditionally energy stability
  for the phase field elastic bending energy model, Comput. Methods Appl. Mech.
  Engrg. 315 (2017) 691--712.

\bibitem{Shen17}
J.~Shen, J.~Xu, J.~Yang, A new class of efficient and robust energy stable
  schemes for gradient flows, SIAM Review 61 (2019) 474--506.

\bibitem{HSY20}
F.~Huang, J.~Shen, Z.~Yang, A highly efficient and accurate new sav approch for
  gradient flows, SIAM J. Sci. Comput. 42 (2020) A2514--A2536.

\bibitem{CY19}
C.~Chen, X.~Yang, Efficient numerical scheme for a dendritic solidification
  phase field model with melt convection, J. Comp. Phys. 388 (2019) 41--62.

\bibitem{CLS19}
Q.~Cheng, C.~Liu, J.~Shen, A new lagrange multiplier approach for grdient
  flows, Comput. Methods Appl. Mech. Engrg. 367 (2020).

\bibitem{QSZ15}
Z.~Qiao, Z.~Sun, Z.~Zhang, Stability and convergence of second-order schemes
  for the nonlinear epitaxial growth model without slope selection, Math. Comp.
  84 (2015) 653--674.

\bibitem{ZY20}
J.~Zhang, X.~Yang, A fully decoupled, linear and unconditionally energy stable
  numerical scheme for a melt-convective phase-field dentritic solidification
  model, Comput. Methods Appl. Mech. Engrg. 363 (2020) 112779.

\bibitem{Metzger2023}
S.~Metzger, A convergent {SAV} scheme for cahn{\textendash}hilliard equations
  with dynamic boundary conditions, {IMA} Journal of Numerical Analysis (2023).

\bibitem{Bao2021}
X.~Bao, H.~Zhang, Numerical approximations and error analysis of the
  cahn{\textendash}hilliard equation with reaction rate dependent dynamic
  boundary conditions, J. Sci. Comput. 87~(3) (2021).

\bibitem{Bao2021a}
X.~Bao, H.~Zhang, Numerical approximations and error analysis of the
  cahn{\textendash}hilliard equation with dynamic boundary conditions, Commun.
  Math. Sci. 19~(3) (2021) 663--685.

\bibitem{Meng2023}
X.~Meng, X.~Bao, Z.~Zhang, Second order stabilized semi-implicit scheme for the
  cahn{\textendash}hilliard model with dynamic boundary conditions, J. Comput.
  Appl. Math. 428 (2023) 115145.

\bibitem{Knopf2021a}
P.~Knopf, A.~Signori, On the nonlocal cahn{\textendash}hilliard equation with
  nonlocal dynamic boundary condition and boundary penalization, J.
  Differential Equations 280 (2021) 236--291.

\bibitem{Cherfils2010}
L.~Cherfils, , M.~Petcu, M.~Pierre, A numerical analysis of the cahn-hilliard
  equation with dynamic boundary conditions, Discrete Contin. Dyn. Syst. 27~(4)
  (2010) 1511--1533.

\bibitem{Cherfils2014}
L.~Cherfils, M.~Petcu, A numerical analysis of the cahn{\textendash}hilliard
  equation with non-permeable walls, Numer. Math. 128~(3) (2014) 517--549.

\bibitem{Fukao2017}
T.~Fukao, S.~Yoshikawa, S.~Wada, Structure-preserving finite difference schemes
  for the cahn–hilliard equation with dynamic boundary conditions in the
  one-dimensional case, Commun. Pure Appl. Anal. 16~(5) (2017) 1915--1938.

\bibitem{Israel2014}
H.~Israel, A.~Miranville, M.~Petcu, Numerical analysis of a
  cahn{\textendash}hilliard type equation with dynamic boundary conditions,
  Ricerche Mat. 64~(1) (2014) 25--50.

\bibitem{Galenko2001}
P.~Galenko, Phase-field model with relaxation of the diffusion flux in
  nonequilibrium solidification of a binary system, Physics Letters A 287~(3-4)
  (2001) 190--197.

\bibitem{Galenko2005}
P.~Galenko, D.~Novikov, P.~S{\"o}ll, Diffusion and relaxation in cahn-hilliard
  dynamics, Physica A: Statistical Mechanics and its Applications 355 (2005)
  149--160.

\bibitem{Galenko2007}
P.~Galenko, D.~Novikov, P.~S{\"o}ll, Mathematical modeling of spinodal
  decomposition in binary alloys with and without relaxation of the diffusion
  flux, Journal of Applied Physics 102 (2007) 013518.

\bibitem{LZG09}
N.~Lecoq, H.~Zapolsky, P.~Galenko, Evolution of the structure factor in a
  hyperbolic model of spinodal decomposition, The European Physical Journal
  Special Topics 177~(1) (2009) 165--175.

\bibitem{Galenko2013}
P.~Galenko, D.~Novikov, P.~S{\"o}ll, Modeling of phase separation in binary
  alloys with inertial and diffusive relaxation, Phys. Rev. Lett. 110 (2013)
  257801.

\bibitem{YZH2018}
X.~Yang, J.~Zhao, X.~He, Linear, second order and unconditionally energy stable
  schemes for the viscous cahn{\textendash}hilliard equation with hyperbolic
  relaxation using the invariant energy quadratization method, J. Comput. Appl.
  Math. 343 (2018) 80--97.

\bibitem{CMY2023}
X.~Chen, L.~Ma, X.~Yang, Error analysis of second-order ieq numerical schemes
  for the viscous cahn-hilliard equation with hyperbolic relaxation, Computers
  and Mathematics with Applications 152 (2023) 112--128.

\bibitem{Wu2007}
H.~Wu, M.~Grasselli, S.~Zheng, Convergence to equilibrium for a
  parabolic–hyperbolic phase-field system with dynamical boundary condition,
  J. Math. Anal. Appl. 329~(2) (2007) 948--976.

\bibitem{Ma2017}
L.~Ma, R.~Chen, X.~Yang, H.~Zhang, Numerical approximations for allen-cahn type
  phase field model of two-phase incompressible fluids with moving contact
  lines, Commun. Comput. Phys. 21~(3) (2017) 867--889.

\bibitem{Chen2018a}
R.~Chen, X.~Yang, H.~Zhang, Decoupled, energy stable scheme for hydrodynamic
  allen-cahn phase field moving contact line model, J. Comput. Math. 36~(5)
  (2018) 661--681.

\end{thebibliography}

\end{document}